\documentclass[a4paper]{amsart}
\usepackage[T1]{fontenc}
\usepackage{lmodern}
\usepackage[utf8]{inputenc}
\usepackage{amssymb}
\usepackage[all]{xy}
\usepackage{enumitem}
\usepackage{mathrsfs,dsfont,mathtools}
\usepackage{nicefrac}

\usepackage[pdftitle={Equivariant embedding theorems and topological index maps},
  pdfauthor={Heath Emerson and Ralf Meyer},
  pdfsubject={Mathematics; MSC 19K35, 46L80}]{hyperref}
\newcommand*{\MRref}[2]{ \href{http://www.ams.org/mathscinet-getitem?mr=#1}{MR \textbf{#1}}}
\newcommand*{\arxiv}[1]{\href{http://www.arxiv.org/abs/#1}{arXiv: #1}}

\usepackage[lite]{amsrefs}
\usepackage{microtype}
\numberwithin{equation}{section}
\theoremstyle{plain}
\newtheorem{theorem}[equation]{Theorem}
\newtheorem{lemma}[equation]{Lemma}
\newtheorem{proposition}[equation]{Proposition}
\newtheorem{corollary}[equation]{Corollary}
\theoremstyle{definition}
\newtheorem{definition}[equation]{Definition}
\theoremstyle{remark}
\newtheorem{remark}[equation]{Remark}
\newtheorem{example}[equation]{Example}

\DeclareMathOperator{\Hom}{Hom}
\DeclareMathOperator{\KK}{KK}
\DeclareMathOperator{\RKK}{RKK}
\DeclareMathOperator{\K}{K}
\DeclareMathOperator{\KO}{KO}
\DeclareMathOperator{\RK}{RK}
\DeclareMathOperator{\Vect}{Vect}
\DeclareMathOperator{\Gl}{Gl}

\newcommand*{\Coh}{\textup F}%cohomology theory

\newcommand*{\bd}{\partial}

\newcommand*{\Id}{\textup{Id}}
\newcommand*{\ima}{\textup i}
\newcommand*{\diff}{\textup d}

\newcommand*{\Spinc}{\textup{Spin}^\textup c}

\newcommand*{\C}{\mathbb C}
\newcommand*{\Z}{\mathbb Z}
\newcommand*{\Torus}{\mathbb T}
\newcommand*{\Sphere}{\mathbb S}
\newcommand*{\Disk}{\mathbb D}
\newcommand*{\N}{\mathbb N}
\newcommand*{\R}{\mathbb R}
\newcommand*{\Comp}{\mathbb K}

\newcommand*{\Hils}{\mathcal H}
\newcommand*{\CONT}{\textup C}
\newcommand*{\EG}{\mathcal E}

\newcommand*{\Nor}{\mathfrak{Nor}}% normally non-singular map category

\newcommand*{\Tvert}{\textup T}%vertical tangent bundle
\newcommand*{\Grd}{\mathcal G}% groupoid
\newcommand*{\Base}{Z} % base space, object space of the groupoid \(\Grd\)
\newcommand*{\Tot}{X}  % total space, a space over \(\Base\)
\newcommand*{\Other}{Y}% another space over \(\Base\)
\newcommand*{\grd}{g}  % element of \(\Grd\)
\newcommand*{\base}{z} % element of \(\Base\)
\newcommand*{\tot}{x}  % element of \(\Tot\)
\newcommand*{\other}{y}% element of \(\Other\)
\newcommand*{\Source}{X}% source of a map over \(\Base\)
\newcommand*{\source}{x}% element of \(\Source\)
\newcommand*{\Target}{Y}% target of a map over \(\Base\)
\newcommand*{\target}{y}% element of \(\Target\)
\newcommand*{\Third}{U}% third target of a map over \(\Base\)
\newcommand*{\VB}{V}% vector bundle
\newcommand*{\vb}{v}% element of \(\VB\)
\newcommand*{\Triv}{E}% trivial vector bundle
\newcommand*{\triv}{e}% element of \(\Triv\)
\newcommand*{\Normal}{\textup N}% (stable) normal bundle
\newcommand*{\normal}{\nu}% element of \(\Normal\)
\newcommand*{\Kclass}{\xi}% class in K-theory
\newcommand*{\an}{\textup{an}}% analytic

\newcommand*{\NM}{\Phi}% normally non-singular map
\newcommand*{\anchor}{\varrho}% anchor map
\newcommand*{\mapr}{f}% right map in correspondence (forwards)
\newcommand*{\pr}{\mathrm{pr}}

\newcommand*{\nb}{\nobreakdash}
\newcommand*{\Cst}{\textup C^*}
\newcommand*{\sC}{\sigma\text{-C}^*}

\newcommand*{\abs}[1]{\lvert#1\rvert}
\newcommand*{\total}[1]{\lvert#1\rvert}%total space of a vector bundle
\newcommand*{\proj}[1]{\pi_{#1}}%bundle projection of a vector bundle
\newcommand*{\zers}[1]{\zeta_{#1}}%zero sectiona of a vector bundle
\newcommand*{\norm}[1]{\lVert#1\rVert}

\newcommand*{\cl}[1]{\overline{#1}}

\newcommand*{\defeq}{\mathrel{\vcentcolon=}}
\newcommand*{\eqdef}{\mathrel{=\vcentcolon}}
\newcommand*{\pt}{\star}

\newcommand*{\blank}{\textup{\textvisiblespace}}

\newcommand*{\mono}{\rightarrowtail}%monomorphism, especially zers
\newcommand*{\epi}{\twoheadrightarrow}%epimorphism, especially proj
\newcommand*{\opem}{\hookrightarrow}%open embedding
\newcommand*{\xopem}{\xhookrightarrow}%open embedding
\newcommand*{\congto}{\xrightarrow\cong}%isomorphism

\DeclareMathOperator{\supp}{supp}

\DeclareMathOperator{\Ind}{Index}

\begin{document}
\title[Embedding theorems and index maps]{Equivariant embedding theorems\\ and topological index maps}

\author{Heath Emerson}
\email{hemerson@math.uvic.ca}

\address{Department of Mathematics and Statistics\\
  University of Victoria\\
  PO BOX 3045 STN CSC\\
  Victoria, B.C.\\
  Canada V8W 3P4}

\author{Ralf Meyer}
\email{rameyer@uni-math.gwdg.de}

\address{Mathematisches Institut and
  Courant Research Centre ``Higher Order Structures''\\
  Georg-August Universit\"at G\"ottingen\\
  Bunsenstra{\ss}e 3--5\\
  37073 G\"ottingen\\
  Germany}

\begin{abstract}
  The construction of topological index maps for equivariant families of Dirac operators requires factoring a general smooth map through maps of a very simple type: zero sections of vector bundles, open embeddings, and vector bundle projections.  Roughly speaking, a normally non-singular map is a map together with such a factorisation.  These factorisations are models for the topological index map.  Under some assumptions concerning the existence of equivariant vector bundles, any smooth map admits a normal factorisation, and two such factorisations are unique up to a certain notion of equivalence.  To prove this, we generalise the Mostow Embedding Theorem to spaces equipped with proper groupoid actions.  We also discuss orientations of normally non-singular maps with respect to a cohomology theory and show that oriented normally non-singular maps induce wrong-way maps on the chosen cohomology theory.  For \(\K\)\nb-oriented normally non-singular maps, we also get a functor to Kasparov's equivariant \(\KK\)-theory.  We interpret this functor as a topological index map.
\end{abstract}

\thanks{Heath Emerson was supported by a National Science and Engineering Council of Canada (NSERC) Discovery grant.  Ralf Meyer was supported by the German Research Foundation (Deutsche Forschungsgemeinschaft (DFG)) through the Institutional Strategy of the University of G\"ottingen.}
\maketitle

\section{Introduction}
\label{sec:intro}

The claim that Kasparov theory for commutative \(\Cst\)\nb-algebras may be described using correspondences came up already in the 1980s (see \cites{Baum-Douglas:K-homology, Baum-Block:Bicycles, Connes-Skandalis:Longitudinal}).  But detailed proofs only appeared much more recently and only for special situations (see \cites{Raven:Thesis, Baum-Higson-Schick:Equivalence}).  This article prepares for a description of bivariant \(\K\)\nb-theory by geometric cycles in~\cite{Emerson-Meyer:Correspondences}.  In fact, it was part of a first draft of~\cite{Emerson-Meyer:Correspondences}.  We split it to make the results more easily accessible.

Correspondences combine the functoriality of \(\K\)\nb-theory for proper maps and its \emph{wrong-way} functoriality for \(\K\)\nb-oriented maps.  The construction of wrong-way functoriality in~\cite{Connes-Skandalis:Longitudinal} is an \emph{analytic} one, however, and a purely topological construction of wrong-way functoriality in \emph{equivariant} bivariant \(\K\)\nb-theory does not seem to exist in the same generality as the analytic construction.  In order to analyse equivariant topological wrong-way functoriality, we introduce a category of \(\K\)\nb-oriented equivariant normally non-singular maps and a covariant functor (wrong-way functor) from this category to equivariant Kasparov theory.  We show for many proper groupoids that \(\K\)\nb-oriented, equivariant, smooth, \emph{normally non-singular} maps are equivalent to \(\K\)\nb-oriented, equivariant, smooth maps in the usual sense.  This depends on the existence of enough equivariant vector bundles. 

The construction of wrong-way elements for smooth \(\K\)\nb-oriented maps in~\cite{Connes-Skandalis:Longitudinal} uses a factorisation into a smooth \(\K\)\nb-oriented embedding and a smooth \(\K\)\nb-oriented submersion.  For embeddings, the construction of wrong-way elements is purely topological, combining the Thom isomorphism for the normal bundle with the \(^*\)\nb-homomorphism on \(\CONT_0\)\nb-functions induced by the open embedding of a tubular neighbourhood.  The wrong-way element for a \(\K\)\nb-oriented submersion \(f\colon \Tot\to\Target\) is the \(\KK\)\nb-class~\(D_f\) of the family of Dirac operators along the fibres of~\(f\).  Hence we call the construction in~\cite{Connes-Skandalis:Longitudinal} \emph{analytic} wrong-way functoriality.

The Atiyah--Singer Index Theorem for families computes the action of~\(D_f\) on \(\K\)\nb-theory.  It may be strenghthened to a topological description of the class~\(D_f\) itself.  This topological model for~\(D_f\) only uses Thom isomorphisms and functoriality for open embeddings.  Like the Atiyah--Singer topological index map, it is based on an embedding~\(\iota\) of~\(\Tot\) into~\(\R^n\).  Then \((\iota,f)\colon \Tot\to\Target\times\R^n\) is an embedding as well.  Let~\(\VB\) be the normal bundle of \((\iota,f)\) and let~\(\total{\VB}\) be its total space.  This vector bundle is \(\K\)\nb-oriented because~\(f\) is \(\K\)\nb-oriented.  Hence the Thom isomorphism provides a natural class in \(\KK_*\bigl(\CONT_0(\Tot),\CONT_0(\total{\VB})\bigr)\).  A tubular neighbourhood for the embedding \((\iota,f)\) provides an open embedding \(\total{\VB}\opem\Target\times\R^n\), which induces a \(^*\)\nb-homomorphism \(\CONT_0(\total{\VB})\to\CONT_0(\Target\times\R^n)\).  Finally, Bott periodicity yields an invertible element in \(\KK_*\bigl(\CONT_0(\Target\times\R^n),\CONT_0(\Target)\bigr)\).  The Kasparov product of these three ingredients is a class~\(f!\) in \(\KK_*\bigl(\CONT_0(\Tot),\CONT_0(\Target)\bigr)\) --~the \emph{topological} wrong-way element of~\(f\).  The functoriality of the analytic wrong-way construction implies
\[
f! = D_f \qquad \text{in \(\KK_*\bigl(\CONT_0(\Tot),\CONT_0(\Target)\bigr)\).}
\]
In particular, \(f!\) depends neither on the chosen embedding~\(\iota\) nor on the chosen tubular neighbourhood for \((\iota,f)\).

Now consider the equivariant situation where a groupoid~\(\Grd\) with object space~\(\Base\) acts on~\(\Tot\).  If we assume that~\(\Grd\) is proper and acts fibrewise smoothly and that~\(f\) is \(\Grd\)\nb-equivariantly \(\K\)\nb-oriented, then the Dirac operators along the fibres of~\(f\) define an equivariant class \(D_f\in \KK^\Grd_*\bigl(\CONT_0(\Tot),\CONT_0(\Target)\bigr)\).  For the topological index map, we need an embedding of~\(\Tot\) into the total space~\(\total{\Triv}\) of a \(\Grd\)\nb-equivariant vector bundle~\(\Triv\) over~\(\Target\).  If such an embedding exists, we may also assume that~\(\Triv\) is \(\Grd\)\nb-equivariantly \(\K\)\nb-oriented.  Then the normal bundle of the embedding is \(\Grd\)\nb-equivariantly \(\K\)\nb-oriented, so that a Thom isomorphism applies, and an equivariant tubular neighbourhood theorem provides an open embedding from the total space of the normal bundle into~\(\total{\Triv}\).  Thus we may construct a topological wrong-way element \(f!\in \KK^\Grd_*\bigl(\CONT_0(\Tot),\CONT_0(\Target)\bigr)\) exactly as above.  The same argument as in the non-equivariant case shows that~\(f!\) is equal to~\(D_f\).  In the special case of bundles of compact groups, this construction is already carried out by Victor Nistor and Evgenij Troitsky in~\cite{Nistor-Troitsky:Thom_gauge}.

The above construction motivates the definition of a \emph{normally non-singular map}.  In a first approximation, a \(\Grd\)\nb-equivariant normally non-singular map from~\(\Source\) to~\(\Target\) consists of a \(\Grd\)\nb-vector bundle~\(\VB\) over~\(\Source\), a \(\Grd\)\nb-vector bundle~\(\Triv^\Target\) over~\(\Target\), and an open embedding~\(f\) from the total space of~\(\VB\) to the total space of~\(\Triv^\Target\).  Two additional assumptions are necessary for certain technical purposes: the \(\Grd\)\nb-vector bundle over~\(\Target\) should be the pull-back of a \(\Grd\)\nb-vector bundle over~\(\Base\) --~we call such \(\Grd\)\nb-vector bundles \emph{trivial}~-- and the \(\Grd\)\nb-vector bundle over~\(\Source\) should be a direct summand in a trivial \(\Grd\)\nb-vector bundle (subtrivial).

We define an appropriate notion of equivalence of normally non-singular maps, based on isotopy of open embeddings and lifting along trivial \(\Grd\)\nb-vector bundles, and we construct a composition and an exterior product that turn equivalence classes of normally non-singular maps into a symmetric monoidal category.  For all these considerations, it is irrelevant whether the maps in question are smooth.

Let~\(\Grd\) be a proper groupoid with object space~\(\Base\) and let \(\Source\) and~\(\Target\) be bundles of smooth manifolds over~\(\Base\) with fibrewise smooth actions of~\(\Grd\).  If there is a \(\Grd\)\nb-equivariant smooth embedding of~\(\Source\) into a \(\Grd\)\nb-vector bundle over~\(\Base\), then any smooth \(\Grd\)\nb-equivariant map \(\Source\to\Target\) is the trace of an essentially unique smooth normally non-singular map.  There are, however, proper groupoids with no non-trivial \(\Grd\)\nb-vector bundles over~\(\Base\).  For them, we lack normally non-singular maps to~\(\Base\), so that smooth maps need not admit the factorisation required for a normally non-singular map.  This is why the theory of normally non-singular maps is needed in~\cite{Emerson-Meyer:Correspondences}.  Under some technical assumptions, smooth normally non-singular maps are essentially equivalent to ordinary smooth maps.  A general theory of correspondences based on smooth maps would need such technical assumptions in all important theorems.  When we replace smooth maps by normally non-singular maps, the theory goes through much more smoothly.

A simple counterexample of a smooth map with no normal factorisation is the following.  Let~\(A\) be a matrix in \(\Gl(2,\Z)\).  Form the locally trivial bundle of torus groups~\(\Torus^2\) over the circle~\(\Sphere^1\) with monodromy induced by~\(A\).  This defines a compact groupoid~\(\Grd_A\) with Haar system.  If~\(A\) is hyperbolic, then it turns out that any \(\Grd_A\)\nb-equivariant vector bundle over~\(\Sphere^1\) carries a trivial action of~\(\Grd_A\).  The morphism space of~\(\Grd_A\) with the smooth translation action of~\(\Grd_A\) cannot embed equivariantly in any \(\Grd_A\)\nb-vector bundle over~\(\Sphere^1\) because the translation action is non-trivial.

As a consequence, the \(\Grd_A\)\nb-equivariant index in \(\K^*_{\Grd_A}(\Sphere^1)\) of the fibrewise Dolbeault operators on the fibres \(\C/\Z^2\) of~\(\Grd_A\) cannot be computed along the lines of the Atiyah--Singer procedure: since there is no equivariant embedding of~\(\Grd_A\) into an equivariant vector bundle over~\(\Sphere^1\), new ideas would be needed to construct a topological index map in this case.  There are similar counterexamples where~\(\Grd\) is a compact group and~\(\Tot\) is a non-compact smooth \(\Grd\)\nb-manifold that is not of finite orbit type.  In this case, Mostow's Embedding Theorem does not apply and we do not know how to construct a topological index map.

Whereas the functoriality of wrong-way elements is a difficult issue in \cites{Connes:Survey_foliations, Connes-Skandalis:Longitudinal}, it is straightforward here thanks to our restrictive notion of normally non-singular map.  For us, the difficulty is to lift smooth maps to normally non-singular maps; once this lifting is achieved, smooth structures become irrelevant.  Furthermore, the equality in Kasparov theory of the analytic and topological wrong-way elements is a strong version of the Atiyah--Singer Index Theorem, whose proof is essentially equivalent to the proof that analytic wrong-way maps are functorial.

And whereas special features of bivariant \(\K\)\nb-theory are needed to construct the \emph{analytic} wrong-way functoriality for smooth submersions in~\cite{Connes-Skandalis:Longitudinal}, the topological wrong-way functoriality for \(\K\)\nb-oriented normally non-singular maps only uses Thom isomorphisms and functoriality for open embeddings.  Therefore, we first discuss the category of normally non-singular maps without orientations and without reference to any cohomology theory.  Then we introduce orientations with respect to any equivariant multiplicative cohomology theory and construct wrong-way maps in this generality.  Finally, we specialise to \(\K\)\nb-theory and compare our construction with the analytic wrong-way maps for smooth \(\K\)\nb-oriented submersions.  The generalisation to arbitrary equivariant cohomology theories should be useful, for instance, to construct natural bivariant Chern characters.

Since analysis plays no role in the construction of our bivariant cohomology theories, we do not need spaces to be locally compact~-- paracompact Hausdorff is good enough.  Our notion of a (numerably) proper groupoid combines Abels' numerably proper group actions (\cite{Abels:Universal}) with Haar systems.  Without assuming local compactness, it ensures that pull-backs of equivariant vector bundles along equivariantly homotopic maps are isomorphic; that equivariant vector bundles carry invariant inner products; and that extensions of equivariant vector bundles split.  These are reasons why we need \(\Grd\)\nb-spaces to be numerably proper and paracompact.

The restriction to proper groupoids looks like a severe limitation of generality at first sight because it seems to exclude a description of \(\KK^\Grd_*(\CONT_0(\Tot),\C)\) for an infinite discrete group~\(\Grd\) and a smooth manifold~\(\Tot\) with a proper smooth action of~\(\Grd\).  A direct approach would require \(\Grd\)\nb-equivariant normally non-singular maps from~\(\Tot\) to the point, which rarely exist.  Instead, we consider \(\RKK^\Grd_*(\EG\Grd; \CONT_0(\Tot),\C)\) for a universal proper \(\Grd\)\nb-space~\(\EG\Grd\).  This group is isomorphic to \(\KK^\Grd_*(\CONT_0(\Tot),\C)\) if~\(\Grd\) acts properly on~\(\Tot\) (see~\cite{Meyer-Nest:BC}), and it requires \(\Grd\ltimes\EG\Grd\)-equivariant normally non-singular maps from \(\EG\Grd\times\Tot\) to~\(\EG\Grd\).  Thus we replace the discrete \emph{group}~\(\Grd\) by the \emph{proper groupoid}~\(\Grd\ltimes\EG\Grd\), and a smooth \(\Grd\)\nb-manifold by a bundle of smooth manifolds over~\(\EG\Grd\) with a fibrewise smooth action of~\(\Grd\).

\smallskip

In Section~\ref{sec:groupoids_actions}, we discuss the class of groupoids that we will be working with, define actions, proper actions, equivariant vector bundles and prove various basic results about them.  In particular, we introduce a class of numerably proper groupoids which behave nicely even without assuming local compactness.

Section~\ref{sec:equivariant_embeddings} contains the main geometric results of this article.  We prove embedding theorems for bundles of smooth manifolds, equivariant with respect to a proper groupoid satisfying some conditions about equivariant vector bundles.

In Section~\ref{sec:normal_maps}, we introduce equivariant normally non-singular maps and define an equivalence relation and a composition for them.  We study some properties of the resulting category of normally non-singular maps and use the embedding theorem to relate it to the homotopy category of smooth maps.

In Section~\ref{sec:orientations}, we assume, in addition, that we are given a multiplicative cohomology theory on the category of \(\Grd\)\nb-spaces.  We discuss the resulting notions of orientation and Thom isomorphisms for oriented \(\Grd\)\nb-vector bundles, and we use the latter to construct wrong-way maps on cohomology for oriented normally non-singular maps.

Finally, Section~\ref{sec:index} discusses wrong-way elements of \(\K\)\nb-oriented normally non-singular maps in bivariant \(\K\)\nb-theory.  We briefly recall the analytic wrong-way functoriality by Connes and Skandalis and compare it to our topological analogue.

\section{Preliminaries on groupoids and their actions}
\label{sec:groupoids_actions}

The authors at first only had locally compact groupoids with Haar system in mind.  But it is interesting to allow also non-locally compact groups such as loop groups or the projective unitary group \(\textup{PU}(\Hils)\) for a Hilbert space~\(\Hils\).  The category of free and proper actions of \(\textup{PU}(\Hils)\) seems a good setting to study twisted \(\K\)\nb-theory.  This is why we allow more general topological groupoids in our constructions, as long as this creates no serious additional difficulties.  We have not explored extensions to non-Hausdorff, locally Hausdorff groupoids as in~\cite{Tu:Non-Hausdorff}.

\begin{definition}
  \label{def:G-map}
  Let~\(\Grd\) be a topological groupoid with object space~\(\Base\) and range and source maps \(r,s\colon \Grd\rightrightarrows \Base\).

  A \emph{\(\Grd\)\nb-space} is a topological space~\(\Tot\) with a continuous action of~\(\Grd\); this means that~\(\Tot\) comes with a continuous map \(\anchor\colon \Tot\to\Base\), its \emph{anchor map}, and a homeomorphism
  \[
  \Grd\times_{s,\anchor}\Tot \congto \Grd\times_{r,\anchor} \Tot,
  \qquad (\grd,\tot)\mapsto (\grd,\grd\cdot\tot),
  \]
  subject to the usual associativity and unitality conditions for groupoid actions.

  A \emph{\(\Grd\)\nb-map} between two \(\Grd\)\nb-spaces is a continuous \(\Grd\)\nb-equivariant map.
\end{definition}

\begin{definition}
  \label{def:transformation_groupoid}
  Let~\(\Grd\) be a topological groupoid and let~\(\Tot\) be a \(\Grd\)\nb-space.  A \emph{\(\Grd\)\nb-space over~\(\Tot\)} is a \(\Grd\)\nb-space with a \(\Grd\)\nb-map to~\(\Tot\).  We define the \emph{transformation groupoid} \(\Grd\ltimes\Tot\) such that a \(\Grd\ltimes\Tot\)-space is nothing but a \(\Grd\)\nb-space over~\(\Tot\).  Its object space is~\(\Tot\), its morphism space~\(\Grd\times_{s,\anchor}\Tot\); its range and source maps are \(r(\grd,\tot) = \grd\cdot\tot\) and \(s(\grd,\tot)=\tot\); and its composition is \((\grd_1,\tot_1)\cdot (\grd_2,\tot_2) \defeq (\grd_1\cdot\grd_2,\tot_2)\).
\end{definition}

\begin{definition}
  \label{def:embedding}
  A \(\Grd\)\nb-map \(f\colon \Source\to\Target\) is called an \emph{embedding} if it is a homeomorphism onto its range, equipped with the subspace topology from~\(\Target\).  An embedding is called \emph{closed} or \emph{open} if \(f(\Source)\) is closed or open in~\(\Target\), respectively.  We reserve the arrow~\(\opem\) for open embeddings.
\end{definition}

\begin{definition}
  \label{def:G-vector_bundle}
  Let~\(\Grd\) be a topological groupoid.  A \emph{\(\Grd\)\nb-vector bundle} over a \(\Grd\)\nb-space~\(\Source\) is a vector bundle with \(\Grd\)\nb-action such that the bundle projection, addition, and scalar multiplication are \(\Grd\)\nb-equivariant.  We denote the total space of a \(\Grd\)\nb-vector bundle~\(\VB\) over~\(\Source\) by~\(\total{\VB}\) (this is a \(\Grd\)\nb-space), the bundle projection \(\total{\VB}\epi\Source\) by~\(\proj{\VB}\), and the zero section \(\Source\mono\total{\VB}\) by~\(\zers{\VB}\) (these are \(\Grd\)\nb-maps).  We reserve the arrows \(\epi\) and~\(\mono\) for vector bundle projections and zero sections.
\end{definition}

\begin{definition}
  \label{def:trivial}
  Since any \(\Grd\)\nb-space~\(\Tot\) comes with an anchor map \(\anchor\colon \Tot\to\Base\), we may pull back a \(\Grd\)\nb-vector bundle~\(\Triv\) over~\(\Base\) to one on~\(\Tot\), which we denote by~\(\Triv^\Tot\); its total space is \(\total{\Triv^\Tot} \defeq \total{\Tot}\times_\Base\Triv\).  A \(\Grd\)\nb-vector bundle over~\(\Tot\) is called \emph{trivial} if it is isomorphic to~\(\Triv^\Tot\) for some \(\Grd\)\nb-vector bundle~\(\Triv\) over~\(\Base\); direct summands of trivial \(\Grd\)\nb-vector bundles are called \emph{subtrivial}.
\end{definition}

Open embeddings and \(\Grd\)\nb-vector bundles are the two ingredients in normally non-singular maps.  For technical reasons, we often require \(\Grd\)\nb-vector bundles to be trivial or subtrivial.

\begin{example}
  \label{exa:trivial_compact}
  Let~\(\Grd\) be a \emph{group}, so that~\(\Base\) is a single point.  A \(\Grd\)\nb-vector bundle over~\(\Base\) is a finite-dimensional representation of~\(\Grd\).  Hence a vector bundle over a \(\Grd\)\nb-space~\(\Tot\) is trivial if and only if it is of the form \(\Tot\times\R^n\) with~\(\Grd\) acting by \(\grd\cdot (\tot,\xi) \defeq (\grd\tot,\pi_\grd\xi)\) for some linear representation~\(\pi\) of~\(\Grd\) on~\(\R^n\).
\end{example}

\begin{example}
  \label{exa:vb_not_subtrivial}
  Let \(\Grd=\Torus\) be the circle group.  Let \(\Tot\defeq\hat{\Torus}\cong\Z\) with the trivial action of~\(\Grd\).  Let~\(\VB\) be the constant rank-one vector bundle over~\(\Tot\) with~\(\Grd\) acting by the character~\(\chi\) at \(\chi\in\hat{\Torus}\).  This equivariant vector bundle is not subtrivial because it involves infinitely many different representations of~\(\Torus\).
\end{example}

\begin{example}
  \label{ex:subtrivial_nontrivial}
  If~\(\Grd\) is trivial and~\(\Tot\) is paracompact with finite covering dimension, then every vector bundle over~\(\Tot\) is subtrivial.
\end{example}

We recall some equivalent definitions of proper maps, which apply even to non-Hausdorff spaces:

\begin{definition}[\cite{Bourbaki:Topologie_generale}*{I.10.2}]
  \label{def:proper_map}
  Let \(\Tot\) and~\(\Other\) be topological spaces.  A continuous map \(f\colon \Tot\to\Other\) is called \emph{proper} if it satisfies the following equivalent conditions:
  \begin{itemize}
  \item For every topological space~\(\Third\), the map \(f\times\Id_\Third\colon \Tot\times\Third\to\Other\times\Third\) is closed.

  \item \(f\) is closed and for each \(\other\in\Other\), the pre-image \(f^{-1}(\other)\) is quasi-compact.

  \item \(f\) is closed and for each quasi-compact subset \(K\subseteq\Other\), the pre-image \(f^{-1}(K)\) is quasi-compact.
  \end{itemize}
\end{definition}

\emph{From now on, all topological spaces, including all topological groupoids, are assumed to be paracompact and Hausdorff.}  Since paracompact spaces are normal, the Tietze Extension Theorem applies and allows us to extend continuous scalar-valued functions from closed subsets.  Even more, we may continuously extend continuous sections of vector bundles or vector bundle homomorphisms that are defined on closed subsets.  We are going to define a notion of proper groupoid action that provides similar \(\Grd\)\nb-equivariant extension results.

\begin{definition}
  \label{def:compactly_supported_probability_measure}
  A \emph{compactly supported probability measure} on a space~\(\Tot\) is a positive, unital, linear functional on the space \(\CONT(\Tot)\) of continuous functions \(\Tot\to\C\) that factors through the restriction map \(\CONT(\Tot)\to\CONT(K)\) for some compact subset \(K\subseteq \Tot\).
\end{definition}

Under our standing assumptions on topological spaces, the restriction map to \(\CONT(K)\) in Definition~\ref{def:compactly_supported_probability_measure} is surjective by the Tietze Extension Theorem, so that a compactly supported probability measure on~\(\Tot\) uniquely determines a regular Borel probability measure on \(K\subseteq\Tot\).

\begin{definition}
  \label{def:proper_groupoid}
  A topological groupoid~\(\Grd\) is called (numerably) \emph{proper} if there is a family of compactly supported probability measures \((\mu^\base)_{\base\in\Base}\) on the fibres \(\Grd^\base\defeq r^{-1}(\base)\) of the range map \(r\colon \Grd\to\Base\) with the following properties:
  \begin{itemize}
  \item \((\mu^\base)_{\base\in\Base}\) is \(\Grd\)\nb-invariant in the sense that \(\grd_*(\mu^{s(\grd)}) = \mu^{r(\grd)}\) for all \(\grd\in\Grd\);

  \item let \(\supp\mu\) be the closure of \(\bigcup_{\base\in\Base} \supp \mu^\base \subseteq \Grd\); the map \(r\colon \supp\mu \to \Base\) is proper;

  \item \((\mu^\base)_{\base\in\Base}\) depends continuously on~\(\base\) in the sense that, for each \(f\in\CONT(\Grd)\), the function \(\base\mapsto \int_{\Grd^\base} f(\grd)\,\diff\mu^\base(\grd)\) on~\(\Base\) is continuous.
  \end{itemize}

  A \(\Grd\)\nb-space~\(\Tot\) is called \emph{proper} if the groupoid \(\Grd\ltimes\Tot\) is proper.
\end{definition}

If we replace the second condition above by the requirement that each~\(\mu^\base\) have full support \(\Grd^\base\), we get a Haar system on~\(\Grd\) in the sense of~\cite{Paterson:Groupoids}.  Our definition is inspired by the definition of numerably proper group actions by Herbert Abels~\cite{Abels:Universal}.  We use measures instead of functions to avoid assumptions about Haar systems.

Equip the \(^*\)\nb-algebras \(\CONT(\Grd)\) and \(\CONT(\Base)\) with the topology of uniform convergence on compact subsets (the compact--open topology) and the action of~\(\Grd\) by left translations.  Definition~\ref{def:proper_groupoid} implies that there is a well-defined positive, unital, \(\Grd\)\nb-equivariant, continuous, linear map
\[
\mu\colon \CONT(\Grd) \to \CONT(\Base),\qquad
(\mu f)(\base) \defeq \int_{\Grd^\base} f(\grd)\,\diff\mu^\base(\grd).
\]
Continuity means that for each compact subset \(K\subseteq\Base\), there is a compact subset \(L\subseteq\Grd\) such that \((\mu f)|_K\) only depends on~\(f|_L\).  This holds with \(L\defeq \supp\mu\cap \Grd^K\), which is compact because of the assumed properness of \(r\colon \supp\mu\to\Base\).  In addition, we require the latter map to be closed.

\begin{example}
  \label{exa:pair_groupoid_proper}
  Let~\(\Base\) be a topological space and let \(\Grd\defeq\Base\times\Base\) be the pair groupoid, equipped with the subspace topology.  Let \(\base_0\in\Base\).  Then~\(\Grd\) is numerably proper with \(\mu^\base\defeq \delta_{(\base,\base_0)}\).  The resulting map \(\mu\colon \CONT(\Base\times\Base)\to\CONT(\Base)\) is induced by the embedding \(\Base\to\Base\times\Base\), \(\base\mapsto (\base,\base_0)\).
\end{example}

\begin{theorem}
  \label{the:numerably_proper_groupoid_proper}
  Let~\(\Grd\) be a groupoid and let~\(\Tot\) be a numerably proper \(\Grd\)\nb-space.  Then the map
  \begin{equation}
    \label{eq:proper_groupoid_via_map}
    \Grd\times_{s,\anchor}\Tot\to\Tot\times \Tot,\qquad
    (\grd,\tot)\mapsto (\grd\cdot\tot,\tot),
  \end{equation}
  is proper.  For each \(\tot\in\Tot\), the map \(\Grd_{\anchor(\tot)}\to\Tot\), \(\grd\mapsto\grd\cdot\tot\), is proper.  Thus the stabiliser~\(\Grd_\tot^\tot\) is compact and the orbit \(\Grd\cdot\tot\) is closed and homeomorphic to the homogeneous space \(\Grd_{\anchor(\tot)}/\Grd_\tot^\tot\) via the map \(\grd\Grd_\tot^\tot\mapsto \grd\tot\).
\end{theorem}

\begin{proof}
  Replacing~\(\Grd\) by \(\Grd\ltimes\Tot\), we may assume that the groupoid~\(\Grd\) itself is numerably proper and that \(\Tot=\Base\).  Let \(S\defeq \supp\mu\subseteq\Grd\).  Since~\(\mu\) is \(\Grd\)\nb-invariant, \(S\) is invariant under left multiplication with elements in~\(\Grd\), that is, \(\Grd\cdot S=S\).  Hence \(S=s^{-1}(S^{(0)})\) for some subset \(S^{(0)}\subseteq\Base\).  If \(\base\in\Base\), then \(\mu^\base\neq0\) and hence \(S\cap\Grd^\base\neq\emptyset\).  Therefore, there is \(\grd\in S\) with \(r(\grd)=\base\).  This means that \(\Grd\cdot S^{(0)}=\Base\).  It follows that \(S\cdot S^{-1} = \Grd\): any \(\grd\in\Grd\) may be written as \(\grd=\grd_1\cdot\grd_2^{-1}\) with \(s(\grd_1)=s(\grd_2)\in S_0\).  By assumption, the map \(r|_S\colon S\to\Base\) is proper.  The following argument only uses the existence of a subset \(S\subseteq\Grd\) with \(S\cdot S^{-1}=\Grd\) such that~\(r|_S\) is proper.

  Since the exterior product \(f\times g\) of two proper maps \(f\) and~\(g\) is again proper, the map
  \begin{equation}
    \label{eq:closed_map_for_proper}
    \Other\times S\times S\to \Other\times \Base\times\Base,\qquad
    (\other,\grd_1,\grd_2)\mapsto \bigl(\other, r(\grd_1),r(\grd_2)\bigr)
  \end{equation}
  is closed for any topological space~\(\Other\).  Now let \(A\subseteq\Other\times\Grd\) be closed.  Its pre-image~\(A'\) in \(\Other\times S\times_{s,s} S\) under the continuous map \((\other,\grd_1,\grd_2)\mapsto (\other,\grd_1\cdot\grd_2^{-1})\) is closed in \(\Other\times S\times S\).  The set~\(A''\) of \(\bigl(\other,r(\grd_1),r(\grd_2)\bigr)\) with \((\other,\grd_1,\grd_2)\in A'\) is closed because the map in~\eqref{eq:closed_map_for_proper} is closed.  Since any \(\grd\in A\) is of the form \(\grd_1\cdot\grd_2^{-1}\) for some \(\grd_1,\grd_2\in S\), \(A''\) is the image of~\(A\) under \(\Id_\Other\times r\times s\).  Thus the map \(\Id_\Other\times r\times s\) is closed for all topological spaces~\(\Other\), that is, \(r\times s\) is a proper map.

  Let \(\base\in\Base\).  The subset \(\Grd_\base\subseteq\Grd\) is closed, so that the restriction of \((r,s)\) to~\(\Grd_\base\) remains proper.  Since~\(s\) is constant on~\(\Grd_\base\), this means that \(r\colon \Grd_\base\to\Base\) is a proper map.  Equivalently, the pre-image~\(\Grd_\base^\base\) of~\(\base\) is compact and the map is closed.  In particular, the image~\(\Grd\cdot\base\) is closed in~\(\Base\).  The induced map \(\Grd_\base/\Grd_\base^\base\to\Grd\base\) is continuous, closed, and bijective, hence it is a homeomorphism.
\end{proof}

Conversely, if the map in~\eqref{eq:proper_groupoid_via_map} is proper, then the groupoid is numerably proper under some additional assumptions:

\begin{lemma}
  \label{lem:def_proper_compare}
  Let~\(\Grd\) be a locally compact groupoid with Haar system and let~\(\Tot\) be a locally compact \(\Grd\)\nb-space.  Assume that the orbit space \(\Grd\backslash\Tot\) is paracompact.  If the map \(\Grd\times_{s,\anchor}\Tot \to \Tot\times\Tot\), \((\grd,\tot)\mapsto (\grd\cdot\tot,\tot)\), is proper, then~\(\Tot\) is a numerably proper \(\Grd\)\nb-space.
\end{lemma}

\begin{proof}
  For each \(\tot\in\Tot\), let~\(U_\tot\) be a relatively compact open neighbourhood.  Since~\(\Grd\) has a Haar system, so has \(\Grd\ltimes\Tot\); hence the projection \(p\colon \Tot\to\Grd\backslash\Tot\) is open (see~\cite{Paterson:Groupoids}*{Proposition 2.2.1}).  The subsets \(p(U_\tot)\) form an open covering of \(\Grd\backslash\Tot\).  By paracompactness, we may find a locally finite covering \((W_i)_{i\in I}\) of \(\Grd\backslash\Tot\) that refines the covering \(p(U_\tot)_{\tot\in\Tot}\).  For each \(i\in I\), choose \(\tot\in\Tot\) with \(W_i\subseteq p(U_\tot)\) and let \(U_i\defeq p^{-1}(W_i)\cap U_\tot\).  Let \(U\defeq \bigcup_{i\in I} U_i\).  By construction, \(p(U_i)=W_i\) and hence \(p(U)=\Grd\backslash\Tot\), that is, \(\Grd\cdot U=\Tot\); moreover, each~\(U_j\) is relatively compact, so that \((A,U_j)\) is relatively compact for compact \(A\subseteq\Tot\) and \(i\in I\).  Since the covering \((W_i)_{i\in I}\) is locally finite, \((A,U)\) is relatively compact as well for all compact subsets~\(A\) of~\(\Tot\).

  There is a continuous function \(\varphi\colon \Tot\to[0,\infty)\) with \(\varphi(x)\neq0\) if and only if \(x\in U\).  Let \(\anchor\colon \Tot\to\Base\) be the anchor map and let \((\lambda^\base)_{\base\in\Base}\) be a Haar system for~\(\Grd\).  Then
  \[
  \mu^\tot(\tot,\grd,\tot') \defeq \varphi(\tot')\cdot \lambda^{\anchor(\tot)}(\grd)
  \qquad\text{for \(\tot,\tot'\in\Tot\), \(\grd\in\Grd_{\anchor(\tot')}^{\anchor(\tot)}\) with \(\grd\cdot\tot'=\tot\)}
  \]
  defines a \(\Grd\)\nb-invariant continuous family of compactly supported positive Borel measures on the fibres of the range map of \(\Grd\ltimes\Tot\).  The measures~\(\mu^\tot\) are not yet probability measures, but they are non-negative and non-zero and hence may be normalised to probability measures, dividing by the \(\Grd\)\nb-invariant, positive, continuous function \(\tot\mapsto \int \varphi(\grd^{-1}\tot) \,\diff\lambda^{\varrho(\tot)}(\grd)\); this function is positive because \(\Grd\cdot U=\Tot\).
\end{proof}

\begin{proposition}
  \label{pro:extend_section_vb}
  Let~\(\Grd\) be a topological groupoid, let~\(\Tot\) be a numerably proper \(\Grd\)\nb-space, and let \(\Other\subseteq\Tot\) be a closed \(\Grd\)\nb-invariant subset.  Then the following kinds of objects may be extended from~\(\Other\) to~\(\Tot\):
  \begin{itemize}
  \item scalar-valued \(\Grd\)\nb-invariant continuous functions;

  \item \(\Grd\)\nb-equivariant continuous sections of \(\Grd\)\nb-vector bundles;

  \item \(\Grd\)\nb-equivariant vector bundle homomorphisms between \(\Grd\)\nb-vector bundles.
  \end{itemize}
\end{proposition}

\begin{proof}
  Scalar-valued \(\Grd\)\nb-invariant continuous functions are \(\Grd\)\nb-equivariant sections of a constant \(\Grd\)\nb-vector bundle, and \(\Grd\)\nb-equivariant vector bundle homomorphisms between two \(\Grd\)\nb-vector bundles \(\VB_1\) and~\(\VB_2\) are \(\Grd\)\nb-equivariant sections of the \(\Grd\)\nb-vector bundle \(\Hom(\VB_1,\VB_2)\).  Hence it suffices to prove: if~\(\VB\) is a \(\Grd\)\nb-vector bundle on~\(\Tot\) and~\(\sigma\) is a \(\Grd\)\nb-equivariant continuous section of~\(\VB|_\Other\), then there is a \(\Grd\)\nb-equivariant continuous section \(\bar\sigma\colon \Tot\to \total{\VB}\) extending~\(\sigma\).

  Since~\(\VB\) is locally trivial, each \(\other\in\Other\) has a neighbourhood~\(U_\other\) on which~\(\VB\) is trivial, so that a section on~\(U_\other\) is equivalent to a family of scalar-valued functions.  Since~\(\Tot\) is paracompact, it is completely regular, that is, scalar-valued functions on~\(\Other\) extend to scalar-valued functions on~\(\Tot\).  Therefore, \(\sigma\) extends to a section of~\(\VB\) on~\(U_\other\).  Since~\(\Tot\) is paracompact, there is a partition of unity subordinate to the covering of~\(\Tot\) by the sets \(\Tot\setminus\Other\) and~\(U_\other\) for \(\other\in\Other\).  This allows us to piece the local sections on~\(U_\other\) and the zero section on \(\Tot\setminus\Other\) together to a continuous section \(h\colon \Tot\to \total{\VB}\) that extends~\(\sigma\).  But~\(h\) need not be \(\Grd\)\nb-equivariant.  We let
  \[
  \bar\sigma(\tot) \defeq \int_{\Grd^{\anchor(\tot)}}
  \grd\cdot \bigl(h(\grd^{-1}\tot)\bigr)
  \,\diff\mu^\tot(\grd),
  \]
  where \(\anchor\colon \Tot\to\Base\) is the anchor map and~\((\mu^\tot)_{\tot\in\Tot}\) is a family of probability measures as in Definition~\ref{def:proper_groupoid}.  This is a \(\Grd\)\nb-equivariant continous section of~\(\VB\).  Since \(h|_\Other=\sigma\) is \(\Grd\)\nb-equivariant and each~\(\mu^\tot\) is a probability measure, \(\bar\sigma|_\Other=\sigma\).
\end{proof}

\begin{corollary}
  \label{cor:vb_iso_extend}
  Let~\(\Grd\) act properly on~\(\Tot\).  If two \(\Grd\)\nb-vector bundles restrict to isomorphic \(\Grd\)\nb-vector bundles on a closed \(\Grd\)\nb-invariant subset~\(\Other\), then they remain isomorphic on some \(\Grd\)\nb-invariant open neighbourhood of~\(\Other\).
\end{corollary}

\begin{proof}
  Use Proposition~\ref{pro:extend_section_vb} to extend a \(\Grd\)\nb-equivariant vector bundle isomorphism on~\(\Other\) to~\(\Tot\).  The subset where it is invertible is open and \(\Grd\)\nb-invariant.
\end{proof}

\begin{proposition}
  \label{pro:vb_inner_product}
  Let~\(\Grd\) be a topological groupoid and let~\(\Tot\) be a numerably proper \(\Grd\)\nb-space.  Then any \(\Grd\)\nb-vector bundle over~\(\Tot\) has a \(\Grd\)\nb-invariant inner product.
\end{proposition}

\begin{proof}
  Since~\(\Tot\) is paracompact, there is a non-equivariant continuous family of inner products~\((h_\tot)_{\tot\in\Tot}\) on the fibres of a \(\Grd\)\nb-vector bundle (compare the proof of Proposition~\ref{pro:extend_section_vb}).  Then \(\int_{\Grd^{\anchor(\tot)}} \grd\cdot h_{\grd^{-1}\tot} \,\diff\mu^\tot(\grd)\) is a \(\Grd\)\nb-invariant inner product.
\end{proof}

\begin{corollary}
  \label{cor:vb_extension_splits}
  Let~\(\Grd\) be a topological groupoid and let~\(\Tot\) be a numerably proper \(\Grd\)\nb-space.  Then any extension of \(\Grd\)\nb-vector bundles over~\(\Tot\) splits.
\end{corollary}

\begin{proof}
  Let \(\VB'\mono \VB\epi \VB''\) be an extension of \(\Grd\)\nb-vector bundles over~\(\Tot\).  The fibrewise orthogonal complement of~\(\VB'\) with respect to a \(\Grd\)\nb-invariant inner product on~\(\VB\) provides a \(\Grd\)\nb-equivariant section for the extension, so that \(\VB\cong \VB'\oplus\VB''\).
\end{proof}

\begin{proposition}
  \label{pro:invariant_partition_unity}
  Let~\(\Grd\) be a topological groupoid, let~\(\Tot\) be a numerably proper \(\Grd\)\nb-space, and let \((U_i)_{i\in I}\) be a covering of~\(\Tot\) by \(\Grd\)\nb-invariant open subsets.  Then there is a \(\Grd\)\nb-invariant partition of unity \((\varphi_i)_{i\in I}\) subordinate to the covering.
\end{proposition}

\begin{proof}
  The space~\(\Tot\) is paracompact by our standing assumption, so that there is a partition of unity \((\varphi'_i)_{i\in I}\) subordinate to the covering.  Let
  \[
  \varphi_i(\tot) \defeq \int_{\Grd^{\anchor(\tot)}} \grd\cdot \varphi'_i(\grd^{-1}\tot) \,\diff\mu^\tot(\grd)
  \qquad\text{for \(\tot\in\Tot\).}
  \]
  These functions are still positive and satisfy \(\sum \varphi_i=1\) because the measures~\(\mu^\tot\) are probability measures.  Moreover, they are continuous and \(\Grd\)\nb-invariant.  Since~\(\varphi_i\) is supported in the \(\Grd\)\nb-orbit of the support of~\(\varphi_i'\), we also get \(\varphi_i(x)=0\) for \(x\notin U_i\).
\end{proof}

\begin{proposition}
  \label{pro:vb_homotopy}
  Let~\(\Grd\) be a topological groupoid, let \(\Source\) and~\(\Target\) be \(\Grd\)\nb-spaces, let~\(\VB\) be a \(\Grd\)\nb-vector bundle over~\(\Target\), and let \(f_0,f_1\colon \Source\rightrightarrows\Target\) be homotopic \(\Grd\)\nb-maps, that is, \(f_t=f|_{\Source\times\{t\}}\) for \(t=0,1\) for a \(\Grd\)\nb-map \(f\colon \Source\times[0,1]\to\Target\).  Then the \(\Grd\)\nb-vector bundles \(f_0^*(\VB)\) and~\(f_1^*(\VB)\) are isomorphic.  Even more, there is a choice of isomorphism that is canonical up to \(\Grd\)\nb-equivariant homotopy.
\end{proposition}

\begin{proof}
  We are going to prove the following claim.  Let~\(\VB\) be a \(\Grd\)\nb-vector bundle over \(\Source'\defeq \Source\times[0,1]\) and let~\(\VB_0\) denote its restriction to \(\Source\times\{0\}\); then the space of \(\Grd\)\nb-equivariant vector bundle isomorphisms \(\VB\cong \VB_0\times[0,1]\) that extend the identity map \(\VB_0\to\VB_0\) over \(\Source\times\{0\}\) is non-empty and connected.  We get the assertion of the proposition when we apply this to the vector bundle \(f^*(\VB)\) over \(\Source\times[0,1]\) and restrict to \(1\in[0,1]\).

  Let \(\pi_1,\pi_2\colon \Source'\times_\Base\Source'\rightrightarrows \Source'\) be the coordinate projections.  The \(\Grd\)\nb-vector bundles \(\pi_1^*\VB\) and~\(\pi_2^*\VB\) on \(\Source'\times_\Base\Source'\) become isomorphic on the diagonal.  Corollary~\ref{cor:vb_iso_extend} shows that they remain isomorphic on some \(\Grd\)\nb-invariant open neighbourhood~\(U\) of the diagonal.  This provides a \(\Grd\)\nb-equivariant continuous family of isomorphisms \(\gamma_{\source_1,\source_2}\colon \VB_{\source_1} \to \VB_{\source_2}\) for \((\source_1,\source_2)\in U\).

  For each \(\source\in\Source\) there is \(\varrho_\source>0\) such that \(\bigl((\source,s), (\source,t)\bigr) \in U\) for all \(s,t\in[0,1]\) with \(\abs{t-s}\le\varrho_\source\) because~\([0,1]\) is compact.  As in the proof of Proposition~\ref{pro:extend_section_vb}, we may construct a \(\Grd\)\nb-invariant continuous function \(\varrho\colon \Tot\to[0,1]\) such that the above holds with \(\varrho_\source=\varrho(\source)\).  We abbreviate \(\gamma_{s,t} \defeq \gamma_{(\source,s),(\source,t)}\) for \(\source\in\Source\), \(s,t\in[0,1]\).  We get a well-defined isomorphism \(\VB_{(\source,0)}\to \VB_{(\source,t)}\) for any \(t\in[0,1]\) by composing \(\gamma_{j\varrho(\source),(j+1)\varrho(\source)}\) for \(0\le j < \lfloor t/\varrho(\source)\rfloor\) and \(\gamma_{\lfloor t/\varrho(\source)\rfloor\varrho(\source),t}\).  This defines a \(\Grd\)\nb-vector bundle isomorphism \(\VB_0\times[0,1] \cong \VB\) that extends the identity map over \(\Source\times\{0\}\).

  Two such isomorphisms differ by composing with a \(\Grd\)\nb-vector bundle automorphism of \(\VB_0\times[0,1]\); this is a continuous path in the group of \(\Grd\)\nb-vector bundle automorphisms of~\(\VB_0\).  Any such path is homotopic to a constant path by reparametrization.  Hence the set of vector bundle isomorphisms under consideration is connected as asserted.
\end{proof}

\subsection{From groupoids to proper groupoids}
\label{sec:reduction}

A good source of examples of proper groupoids are the transformation groupoids \(\Grd \defeq G\ltimes \EG G\), where~\(\EG G\) is a universal numerably proper action of a group or groupoid~\(G\).  Replacing~\(G\) by~\(\Grd\) loses no information as far as equivariant vector bundles are concerned.  Now we explain this observation in more detail.

\begin{definition}
  \label{def:EG}
  A numerably proper \(\Grd\)\nb-space~\(\EG\Grd\) is \emph{universal} if any numerably proper \(\Grd\)\nb-space admits a \(\Grd\)\nb-map to it and if any two parallel \(\Grd\)\nb-maps to~\(\EG\Grd\) are (\(\Grd\)\nb-equivariantly) homotopic.
\end{definition}

This weak universal property characterises~\(\EG\Grd\) uniquely up to \(\Grd\)\nb-homotopy equivalence.  Furthermore, \(\EG\Grd=\Base\) if and only if~\(\Grd\) is numerably proper.

For a locally compact groupoid~\(G\), there is a locally compact, universal proper \(G\)\nb-space by a construction due to Gennadi Kasparov and Georges Skandalis~\cite{Kasparov-Skandalis:Bolic} (see \cite{Tu:Novikov}*{Proposition 6.13} for the groupoid case); it makes no difference whether we use numerably proper actions or proper actions here by Lemma~\ref{lem:def_proper_compare}.

Let~\(\EG G\) be a universal proper \(G\)\nb-space and let~\(\Grd\) be the crossed product groupoid \(\Grd \defeq G\ltimes\EG G\).  Then a \(\Grd\)\nb-space is nothing but a \(G\)\nb-space equipped with a \(G\)\nb-equivariant map to~\(\EG G\).  As a \(G\)\nb-space, \(\EG G\) carries a canonical map to the object space~\(\Base\) of~\(G\), which we use to pull back a \(G\)\nb-space~\(\Tot\) to a \(\Grd\)\nb-space \(\Tot\times_\Base\EG G\).

\begin{lemma}
  \label{lem:EG_times_X}
  \(\EG G\times_\Base\Tot\) is canonically \(G\)\nb-equivariantly homotopy equivalent to~\(\Tot\) if~\(\Tot\) is a proper \(G\)\nb-space.
\end{lemma}

\begin{proof}
  By the Yoneda Lemma, it suffices to show that the canonical projection \(\EG G\times_\Base\Tot\to\Tot\) induces a bijection on the sets of \(G\)\nb-homotopy classes of \(G\)\nb-maps \(\Other\to\blank\) for any \emph{proper} \(G\)\nb-space~\(\Other\).  But \(G\)\nb-homotopy classes of \(G\)\nb-maps \(\Other\to\EG G\times_\Base\Tot\) are just pairs consisting of a \(G\)\nb-homotopy class of a \(G\)\nb-map from~\(\Other\) to~\(\EG G\) and one from~\(\Other\) to~\(\Tot\).  Since there is a unique \(G\)\nb-homotopy class of \(\Grd\)\nb-maps \(\Other\to\EG G\), we get the desired bijection.
\end{proof}

\begin{corollary}
  \label{cor:inflating_vector_bundles}
  If~\(G\) is a locally compact groupoid, then the set of isomorphism classes of \(G\)\nb-vector bundles over a proper \(G\)\nb-space~\(\Tot\) is in bijective correspondence with the set of isomorphism classes of \(G\)\nb-vector bundles over~\(\Tot\times_\Base \EG G\).

  With this identification, trivial \(G\ltimes\EG G\)\nb-vector bundles over \(\Tot\times_\Base \EG G\) agree with the \(G\)\nb-equivariant vector bundles on~\(\Tot\) which are pulled back from~\(\EG G\) under the classifying map \(\Tot \to \EG G\).
\end{corollary}

Furthermore, since the classifying map~\(\chi\) is unique up to \(G\)\nb-homotopy, and pull-backs of \(G\)\nb-vector bundles along \(G\)\nb-homotopic maps are isomorphic by Proposition~\ref{pro:vb_homotopy}, this more general notion of trivial \(G\)\nb-vector bundle does not depend on the choice of the auxiliary map~\(\chi\); since any two universal proper \(G\)\nb-spaces are homotopy equivalent, it does not depend on the choice of~\(\EG G\) either.

\section{Equivariant embeddings of bundles of smooth manifolds}
\label{sec:equivariant_embeddings}

\subsection{Full vector bundles and enough vector bundles}
\label{sec:full}

We define full \(\Grd\)\nb-vector bundles and what it means to have enough \(\Grd\)\nb-vector bundles over a \(\Grd\)\nb-space.  These definitions emerged out of the work of Wolfgang L\"uck and Bob Oliver~\cite{Lueck-Oliver:Completion}.

\begin{definition}[\cite{Emerson-Meyer:Equivariant_K}]
  \label{def:enough_equivariant_vb}
  Let~\(\Grd\) be a topological groupoid and let~\(\Tot\) be a numerably proper \(\Grd\)\nb-space.  There are \emph{enough \(\Grd\)\nb-vector bundles} on~\(\Tot\) if for every \(\tot\in\Tot\) and every finite-dimensional representation of the stabiliser~\(\Grd_\tot^\tot\), there is a \(\Grd\)\nb-vector bundle over~\(\Tot\) whose fibre at~\(\tot\) contains the given representation of~\(\Grd_\tot^\tot\).

  A \(\Grd\)\nb-vector bundle~\(\VB\) on~\(\Tot\) is \emph{full} if for every \(\tot\in\Tot\), the fibre of~\(\VB\) at~\(\tot\) contains all irreducible representations of the stabiliser~\(\Grd_\tot^\tot\).
\end{definition}

If there is a full \(\Grd\)\nb-vector bundle over~\(\Tot\), then~\(\Tot\) has enough \(\Grd\)\nb-vector bundles.

\begin{example}
  \label{exa:constant_full_vb}
  We always have the constant \(\Grd\)\nb-vector bundles \(\Tot\times\R^n\epi\Tot\) for \(n\in\N\) with the trivial representation of~\(\Grd\) on~\(\R^n\).  Such a \(\Grd\)\nb-vector bundle is full if and only if~\(\Grd\) acts freely on~\(\Tot\).

  A groupoid \(G\ltimes\EG G\) for a groupoid~\(G\) is free if and only if~\(G\) is torsion-free in the sense that the isotropy groups~\(G_\base^\base\) for \(\base\in\Base\) contain no compact subgroups.
\end{example}

\begin{example}
  \label{ex:enough_bundles_for_compact_grps}
  Let~\(\Grd\) be a compact group.  Any finite-dimensional representation of a closed subgroup of a compact group can be embedded in the restriction of a finite-dimensional representation of the group itself (\cite{Palais:Slices}*{Theorem 3.1}).  Hence any \(\Grd\)\nb-space has enough \(\Grd\)\nb-vector bundles: trivial \(\Grd\)\nb-vector bundles suffice.  By Lemma~\ref{lem:full_finite}, a full \(\Grd\)\nb-vector bundle on a \(\Grd\)\nb-space can only exist if the size of stabilisers is uniformly bounded.  This necessary condition is not yet sufficient.
\end{example}

\begin{example}
  \label{exa:not_enough_vb}
  We examine a class of compact Lie groupoids that may or may not have enough equivariant vector bundles (see also~\cite{Nistor-Troitsky:Thom_gauge}).  Let~\(K\) be the Lie group~\(\Torus^n=(\R/\Z)^n\) for some \(n\in\N\).  A locally trivial group bundle~\(\Grd\) over the circle \(\Base\defeq\R/\Z\) with fibre~\(K\) is determined uniquely up to isomorphism by an isomorphism \(\sigma\colon K\to K\): we have \(\Grd = K\times [0,1]/{\sim}\) with \((k,0)\sim (\sigma(k),1)\) for all \(k\in K\).  The automorphism group of~\(\Torus^n\) is isomorphic to \(\Gl(n,\Z)\), so that we now write~\(\Grd_A\) for the group bundle associated to \(A\in\Gl(n,\Z)\).

  Before we study when such groupoids have enough equivariant vector bundles, we mention another equivalent construction.  Given \(A\in\Gl(n,\Z)\), we may also form a crossed product Lie group \(K\rtimes_A\Z\).  The trivial action of~\(K\) and the translation action of~\(\Z\) on~\(\R\) combine to an action of \(K\rtimes_A\Z\) on~\(\R\).  This is a universal proper action of \(K\rtimes_A\Z\).  The resulting transformation groupoid is Morita equivalent to the groupoid~\(\Grd_A\) because~\(\Z\) acts freely and properly on~\(\R\).

  A \(\Grd_A\)\nb-equivariant vector bundle over the circle is equivalent to a \(K\)\nb-equivariant vector bundle over~\([0,1]\) together with an appropriate identification of the fibres at \(0\) and~\(1\).  But \(K\)\nb-equivariant vector bundles over~\([0,1]\) are all trivial, so that we just get a finite-dimensional representation~\(\pi\) of~\(K\) on some vector space~\(V\) together with an invertible map \(\tau\colon V\to V\) that satisfies \(\tau\pi_k = \pi_{A(k)}\tau\) for all \(k\in\Torus^n\).  Equivalently, \(\tau\) is an invertible intertwiner \(\pi\cong \pi\circ A\).

  Taking multiplicities, we interpret the representation~\(\pi\) as a finitely supported function \(f_\pi\colon \widehat{K} \cong \Z^n \to \Z\).  The automorphism \(\sigma_A \colon K \to K\) dualises to a map \(\hat{\sigma}_A\colon \widehat{K}\to \widehat{K}\), which is represented by the transpose of the matrix~\(A\).  If~\(\pi\) admits a map~\(\tau\) as above, then the associated function~\(f_\pi\) must be \(A\)\nb-invariant as a function on~\(\Z^n\).  Conversely, if~\(f_\pi\) is \(A\)\nb-invariant, then there is an isomorphism \(\pi\cong \pi\circ A\); two such isomorphisms differ by a unitary intertwiner of~\(\pi\), and these unitary intertwiners form a connected group.  Since homotopic~\(\tau\) give isomorphic \(\Grd_A\)\nb-vector bundles, we conclude that isomorphism classes of \(\Grd_A\)\nb-vector bundles correspond bijectively to \(A\)\nb-invariant functions \(\widehat{K}\to\Z\) with finite support.  Such functions descend to the space of \(A\)\nb-orbits, and they \emph{vanish on infinite orbits}.  This yields the free Abelian group spanned by the characteristic functions of finite \(A\)\nb-orbits in~\(\widehat{K}\).

  It follows immediately from this discussion that~\(\Grd_A\) has enough equivariant vector bundles if and only if all \(A\)\nb-orbits in~\(\widehat{K}\) are finite.  More precisely, an irreducible representation~\(\chi\) of the stabiliser~\(K\) of a point in~\(\Base\) occurs in a \(\Grd_A\)\nb-equivariant vector bundle over~\(\Base\) if and only if~\(\chi\) has a finite \(A\)\nb-orbit.

  If \(n=1\), so that~\(\Grd_A\) is a bundle of circles, then \(\sigma \in \{\Id, -1\} = \Gl(1,\Z)\) and any \(A\)\nb-orbit is finite, so that we do not yet get counterexamples.  If \(n=2\) and \(A\in\Gl(2,\Z)\), we must distinguish the elliptic, parabolic, and hyperbolic cases.  If~\(A\) is elliptic, that is, \(A\) has two different eigenvalues of modulus~\(1\), then all \(A\)\nb-orbits on~\(\Z^2\) are finite because~\(A\) is unitary in some scalar product; hence there are enough \(\Grd_A\)\nb-vector bundles.  The same happens for \(A=1\).  Otherwise, if~\(A\) is parabolic and not~\(1\), then~\(A\) is conjugate to the matrix \(\bigl(\begin{smallmatrix}1&1\\0&1\end{smallmatrix}\bigr)\).  The resulting action on~\(\Z^2\) fixes the points \((m,0)\) and has infinite orbits otherwise.  Finally, if~\(A\) is hyperbolic, that is, \(A\) has eigenvalues of modulus different from~\(1\), then the only finite orbit is~\(\{0\}\).

  As a consequence, \(\Grd_A\) for \(A\in\Gl(2,\Z)\) has enough equivariant vector bundles if and only if~\(A\) is elliptic or \(A=1\).  If~\(A\) is hyperbolic, then all \(\Grd_A\)\nb-vector bundles over~\(\Base\) carry the trivial action of~\(K\), and so there are not enough equivariant vector bundles.  If~\(A\) is parabolic, then there are many, but not enough non-equivalent irreducible \(\Grd_A\)\nb-vector bundles over~\(\Base\).
\end{example}

We now return to the general theory of equivariant vector bundles.

\begin{lemma}
  \label{lem:pullback_enough_trivial_vb}
  Let \(f\colon \Source\to\Target\) be a \(\Grd\)\nb-map.  If~\(\Target\) has enough \(\Grd\)\nb-vector bundles, so has~\(\Source\).  If~\(\VB\) is a full \(\Grd\)\nb-vector bundle over~\(\Target\), then \(f^*(\VB)\) is a full \(\Grd\)\nb-vector bundle over~\(\Source\).
\end{lemma}

\begin{proof}
  Pick \(\source\in\Source\) and a representation~\(\varrho\) of its stabiliser~\(\Grd_\source^\source\).  There is a representation~\(\hat{\varrho}\) of \(\Grd_{f(\source)}^{f(\source)}\supseteq \Grd_\source^\source\) whose restriction to~\(\Grd_\source^\source\) contains~\(\varrho\) (see \cite{Palais:Slices}*{Theorem 3.1}).  Hence the pull-back of a full \(\Grd\)\nb-vector bundle over~\(\Target\) is a full \(\Grd\)\nb-vector bundle over~\(\Source\).  The first assertion is proved similarly.
\end{proof}

As a result, if~\(\Base\) has enough \(\Grd\)\nb-vector bundles or a full \(\Grd\)\nb-vector bundle, then so have all \(\Grd\)\nb-spaces.  Furthermore, so have all \(\Grd'\)\nb-spaces for \(\Grd'\subseteq\Grd\).

\begin{lemma}
  \label{lem:morita_invariance_enough_vbs}
  The property of having enough \(\Grd\)\nb-vector bundles or a full \(\Grd\)\nb-vector bundle is invariant under Morita equivalence in the sense that if \(\Grd_1\) and~\(\Grd_2\) are equivalent groupoids, then~\(\Base_1\) has enough \(\Grd_1\)\nb-vector bundles \textup(respectively a full \(\Grd_1\)\nb-vector bundle\textup) if and only if~\(\Base_2\) has enough \(\Grd\)\nb-vector bundles \textup(respectively a full \(\Grd_2\)\nb-vector bundle\textup).
\end{lemma}

For proper, locally compact groupoids, this is an immediate consequence of \cite{Emerson-Meyer:Equivariant_K}*{Theorem 6.14}, which asserts that~\(\Base\) has enough \(\Grd\)\nb-vector bundles if and only if \(\sC(\Grd)\otimes \Comp\) has an approximate unit of projections.  The latter condition is obviously Morita invariant.  Similarly for having a full \(\Grd\)\nb-vector bundle: this holds if and only if \(\sC(\Grd)\otimes \Comp\) contains a full projection.  We omit the argument for non-locally compact groupoids.

\begin{remark}
  \label{rem:full_vb_reduction_proper}
  We usually replace a non-proper groupoid~\(\Grd\) by \(\Grd\ltimes\EG\Grd\) as explained in Section~\ref{sec:reduction}.  Since any numerably proper \(\Grd\)\nb-space~\(\Tot\) is \(\Grd\)\nb-equivariantly homotopy equivalent to \(\EG\Grd\times\Tot\) (Lemma \ref{lem:EG_times_X}) the category of \(\Grd\ltimes\EG\Grd\)\nb-vector bundles over \(\EG\Grd\times\Tot\) is equivalent to the category of \(\Grd\)\nb-vector bundles over~\(\Tot\) by Proposition~\ref{pro:vb_homotopy}.  It follows that there is a full \(\Grd\)\nb-vector bundle over~\(\Tot\) if and only if there is a full \(\Grd\ltimes\EG\Grd\)\nb-vector bundle over \(\Tot\times\EG\Grd\), and there are enough \(\Grd\)\nb-vector bundles over~\(\Tot\) if and only if there are enough \(\Grd\ltimes\EG\Grd\)\nb-vector bundles over \(\Tot\times\EG\Grd\).
\end{remark}

Some cases where the existence of enough \(\Grd\)\nb-vector bundles or of a full \(\Grd\)\nb-vector bundle are known are listed in \cite{Emerson-Meyer:Equivariant_K}*{\S6.2}.  These include the following cases.
\begin{itemize}\label{list:enough_vb}
\item A constant vector bundle \(\Tot\times\R^n\) is full if and only if~\(\Grd\) acts freely on~\(\Tot\) (Example~\ref{exa:constant_full_vb}).

\item If~\(G\) is a closed subgroup of an almost connected locally compact group, then there are enough \(G\)\nb-vector bundles on any proper \(G\)\nb-space.

  It suffices to prove this if~\(G\) itself is almost connected because this property is inherited by closed subgroups.  We may restrict attention to \(\EG G=G/K\) for a maximal compact subgroup \(K\subseteq G\) by Lemma~\ref{lem:pullback_enough_trivial_vb}.  The latter is Morita equivalent to the compact group~\(K\), which has enough equivariant vector bundles on a point (see Example~\ref{exa:trivial_compact} and Lemma~\ref{lem:morita_invariance_enough_vbs}).

\item There is a full \(G\ltimes\Tot\)\nb-vector bundle on~\(\Tot\) if~\(G\) is a discrete group and~\(\Tot\) is a finite-dimensional proper \(G\)\nb-space with uniformly bounded isotropy groups; this result is due to Wolfgang L\"uck and Bob Oliver (\cite{Lueck-Oliver:Completion}*{Corollary 2.7}).

\item Let~\(\Tot\) be a smooth and connected \(\Grd\)\nb-manifold.  Suppose that~\(\Grd\) acts faithfully on~\(\Tot\) in the sense that if \(\grd\in\Grd_\base^\base\) acts identically on the fibre~\(\Tot_\base\), then \(\grd=1\).  We claim that there are enough \(\Grd\)\nb-vector bundles over~\(\Tot\).  (Example~\ref{exa:not_enough_vb} shows that this may fail if the action is not faithful).

  Equip~\(\Tot\) with a \(\Grd\)\nb-invariant Riemannian metric.  Recall that an isometry of a connected Riemannian manifold that fixes a point \(\tot \in \Tot\) and such that the induced map on~\(\Tvert_\tot \Tot\) is the identity must act as the identity map on \(\Tot\).  Hence differentiation gives an embedding \(\Grd_\tot^\tot \to \textup{O}(n,\R)\) for \(n=\dim \Tot\).  Since~\(\Grd_\tot^\tot\) is compact its image is a closed subgroup.  The basic representation theory of the orthogonal groups now implies that any irreducible representation of a subgroup of~\(\Grd_\tot^\tot\) occurs in \(\Tvert_\tot\Tot^{\otimes k} \otimes \Tvert_\tot^*\Tot^{\otimes l}\) for some \(k,l\in\N\), and we are done since these obviously extend to \(\Grd\)\nb-vector bundles over~\(\Tot\).

  If the stabilisers are finite and of uniformly bounded size, then the sum of \(\Tvert_\tot\Tot^{\otimes k} \otimes \Tvert_\tot^*\Tot^{\otimes l}\) for all \(k,l\in\N\) with \(k,l\le N\) for some~\(N\) is a full \(\Grd\)\nb-vector bundle.

\item There is a full \(\Grd\)\nb-vector bundle on any \(\Grd\)-space if~\(\Grd\) is an orbifold groupoid (see \cite{Emerson-Meyer:Equivariant_K}, Example 6.17).
\end{itemize}

The following lemma shows that many groupoids cannot have a \emph{full} equivariant vector bundle.

\begin{lemma}
  \label{lem:full_finite}
  If there is a full \(\Grd\)\nb-vector bundle on~\(\Tot\) of rank \(n\in\N\), then the stabilisers~\(\Grd_\tot^\tot\) for \(\tot\in\Tot\) are finite with at most \((n-1)^2+1\) elements.
\end{lemma}

\begin{proof}
  Since~\(\Grd\) is proper, the stabiliser~\(\Grd_\tot^\tot\) is a compact group.  Having only finitely many irreducible representations, it must be finite.  The sum of the dimensions of its irreducible representations is at most~\(n\) by assumption, and the sum of their squares is the size of~\(\Grd_\tot^\tot\).  Since there is always the trivial representation of dimension~\(1\), we get \(\abs{\Grd_\tot^\tot}-1\le (n-1)^2\).
\end{proof}

\subsection{Subtrivial equivariant vector bundles}
\label{sec:vb_subtrivial}

Recall that a \(\Grd\)\nb-vector bundle~\(\VB\) over~\(\Tot\) is called \emph{subtrivial} if it is a direct summand of a trivial \(\Grd\)\nb-vector bundle.  Swan's Theorem asserts that all vector bundles over paracompact topological spaces of finite covering dimension are subtrivial.  Equivariant versions of this theorem need additional assumptions.  Here we give several sufficient conditions for a \(\Grd\)\nb-vector bundle to be subtrivial.  The necessary and sufficient condition in Lemma~\ref{lem:vb_subtrivial} requires the existence of a trivial vector bundle with special properties and, therefore, tends to be impractical.  Theorem~\ref{the:vb_subtrivial_full} requires the existence of a full equivariant vector bundle on~\(\Base\); this covers, in particular, many proper actions of discrete groups (Theorem~\ref{the:vb_subtrivial_discrete}).  Finally, Theorem~\ref{the:vb_subtrivial_cocompact} requires enough equivariant vector bundles on~\(\Base\) and a cocompact action of~\(\Grd\) on~\(\Tot\); this covers actions of compact groups on compact spaces.

A similar pattern will emerge for equivariant embeddings: there are several similar sufficient conditions for these to exist.

\begin{lemma}
  \label{lem:vb_subtrivial}
  Let~\(\Tot\) be a \(\Grd\)\nb-space with anchor map \(\anchor\colon \Tot\to\Base\).  Assume that the orbit space~\(\Grd\backslash\Tot\) has finite covering dimension.  A \(\Grd\)\nb-vector bundle~\(\VB\) over~\(\Tot\) is subtrivial if and only if there is a \(\Grd\)\nb-vector bundle~\(\Triv\) over~\(\Base\) such that, for each \(\tot\in\Tot\), there is a \(\Grd_\tot^\tot\)\nb-equivariant linear embedding \(\VB_\tot\to \Triv_{\anchor(\tot)} = \Triv^\Tot_\tot\).
\end{lemma}

\begin{proof}
  The necessity of the condition is obvious.  Assume now that there is a \(\Grd\)\nb-vector bundle~\(\Triv\) over~\(\Base\) with the required property.  Hence there is an injective \(\Grd_\tot^\tot\)\nb-equivariant linear map \(\phi \colon \VB_\tot \to \Triv^\Tot_\tot\) for each \(\tot\in\Tot\).  This extends to a continuous equivariant embedding over the orbit of \(\tot\) (explicitly by \(v\mapsto g\phi (g^{-1}v)\) for \(v\in \VB_{g\tot}\)) by Theorem~\ref{the:numerably_proper_groupoid_proper}.

  This map then extends to a \(\Grd\)\nb-equivariant linear map \(\eta_\tot\colon \VB\to\Triv^\Tot\) by Proposition~\ref{pro:extend_section_vb}.  Since injectivity is an open condition, \(\eta_\tot\) is still injective in some \(\Grd\)\nb-invariant open neighbourhood~\(U_\tot\) of~\(\tot\).

  Thus we get a covering of~\(\Tot\) by \(\Grd\)\nb-invariant open subsets on which we have \(\Grd\)\nb-equivariant linear embeddings of~\(\VB\) into~\(\Triv^\Tot\).  We may view this covering as an open covering of \(\Grd\backslash\Tot\).  It has a refinement with finite Lebesgue number because~\(\Grd\backslash\Tot\) has finite covering dimension.  That is, there are finitely many families \(\mathcal{U}_0,\dotsc,\mathcal{U}_n\) of disjoint, \(\Grd\)\nb-invariant open subsets of~\(\Tot\) with \(\bigcup_{j=0}^n \bigcup \mathcal{U}_j = \Tot\).  Since nothing obstructs combining our embeddings on disjoint open subsets, we get equivariant embeddings \(\eta_j\colon \VB|_{U_j}\to \Triv^\Tot|_{U_j}\) on \(U_j \defeq \bigcup \mathcal{U}_j\) for \(j=0,\dotsc,n\).  Proposition~\ref{pro:invariant_partition_unity} provides a \(\Grd\)\nb-invariant partition of unity on~\(\Tot\) subordinate to the covering \((U_j)_{j=0,\dotsc,n}\).  The resulting linear map \(\bigoplus \varphi_j\cdot\eta_j\colon \VB\to (\Triv^\Tot)^{n+1}\) is a \(\Grd\)\nb-equivariant embedding.

  The \(\Grd\)\nb-vector bundle~\(\Triv^\Tot\) admits a \(\Grd\)\nb-invariant inner product by Proposition~\ref{pro:vb_inner_product}.  This provides an orthogonal direct sum decomposition \(\VB\oplus\VB^\bot \cong \Triv^\Tot\).
\end{proof}

\begin{theorem}
  \label{the:vb_subtrivial_full}
  Let~\(\Tot\) be a \(\Grd\)\nb-space.  Assume that~\(\Grd\backslash\Tot\) has finite covering dimension and that there is a full \(\Grd\)\nb-equivariant vector bundle on~\(\Base\).  Then any \(\Grd\)\nb-vector bundle over~\(\Tot\) is subtrivial.
\end{theorem}

\begin{proof}
  Let~\(n\) be the rank of a \(\Grd\)\nb-vector bundle~\(\VB\) over~\(\Tot\); let \(\anchor\colon \Tot\to\Base\) be the anchor map; and let~\(\Triv\) be a full \(\Grd\)\nb-equivariant vector bundle over~\(\Base\).  Since~\(\Triv_{\anchor(\tot)}\) for \(\tot\in\Tot\) contains all irreducible representation of \(\Grd_{\anchor(\tot)}^{\anchor(\tot)}\) and hence of~\(\Grd_\tot^\tot\), the fibres of the \(\Grd\)\nb-vector bundle \(\anchor^*(\Triv)^n\) over~\(\Tot\) contain all representations of~\(\Grd_\tot^\tot\) of rank at most~\(n\).  Hence the assumptions of Lemma~\ref{lem:vb_subtrivial} are satisfied, and we get the assertion.
\end{proof}

\begin{theorem}
  \label{the:vb_subtrivial_discrete}
  Let~\(\Grd\) be a discrete group and let~\(\Tot\) be a finite-dimensional, proper \(\Grd\)\nb-CW-complex with uniformly bounded isotropy groups.  Let~\(\Other\) be a \(\Grd\)\nb-space over~\(\Tot\).  Then any \(\Grd\ltimes\Tot\)\nb-vector bundle over~\(\Other\) is subtrivial; that is, any \(\Grd\)\nb-vector bundle over~\(\Other\) is a direct summand in the pull-back of a \(\Grd\)\nb-vector bundle over~\(\Tot\).
\end{theorem}

\begin{proof}
  This follows from Theorem~\ref{the:vb_subtrivial_full} and \cite{Lueck-Oliver:Completion}*{Corollary 2.7}, which provides the required full equivariant vector bundle on~\(\Tot\).
\end{proof}

\begin{theorem}
  \label{the:vb_subtrivial_cocompact}
  If~\(\Tot\) is a cocompact \(\Grd\)\nb-space and there are enough \(\Grd\)\nb-vector bundles on~\(\Base\), then any \(\Grd\)\nb-vector bundle over~\(\Tot\) is subtrivial.
\end{theorem}

\begin{proof}
  Let~\(\VB\) be a \(\Grd\)\nb-vector bundle over~\(\Tot\) and let \(\anchor\colon\Tot\to\Base\) be the anchor map.  Then \(\Grd_{\anchor(\tot)}^{\anchor(\tot)} \supseteq \Grd_\tot^\tot\) for all \(\tot\in\Tot\).  By assumption, any representation of~\(\Grd_{\anchor(\tot)}^{\anchor(\tot)}\) occurs in some \(\Grd\)\nb-vector bundle over~\(\Base\).  Hence any representation of~\(\Grd_\tot^\tot\) occurs in~\(\Triv^\Tot_\tot\) for some \(\Grd\)\nb-vector bundle~\(\Triv\) over~\(\Base\) by \cite{Palais:Slices}*{Theorem 3.1}.  Therefore, for each \(\tot\in\Tot\) there is a \(\Grd\)\nb-vector bundle~\(\Triv(\tot)\) over~\(\Base\) and a \(\Grd\)\nb-equivariant linear map \(f(\tot)\colon \VB\to\Triv(\tot)^\Tot\) that is injective over some \(\Grd\)\nb-invariant open neighbourhood~\(U_\tot\) of~\(\tot\), compare the proof of Lemma~\ref{lem:vb_subtrivial}.  Since \(\Grd\backslash\Tot\) is compact, there is a finite set \(F\subseteq\Tot\) such that the open neighbourhoods~\(U_\tot\) for \(\tot\in F\) cover~\(\Tot\).  Now use a partition of unity as in the proof of Lemma~\ref{lem:vb_subtrivial} to embed~\(\VB\) into \(\bigoplus_{\tot\in F} \Triv(\tot)^\Tot\).
\end{proof}

See Example~\ref{exa:vb_not_subtrivial} for an example of an equivariant vector bundle that is not subtrivial.  More examples where~\(\Grd\) is a bundle of compact groups are described in~\cite{Nistor-Troitsky:Thom_gauge}.

\subsection{Smooth \texorpdfstring{$\Grd$}{G}-manifolds}
\label{sec:smooth_G-manifolds}

\begin{definition}
  \label{def:bundle_smooth}
  Let \(\anchor\colon \Tot\to\Base\) be a space over~\(\Base\).  A \emph{chart} on~\(\Tot\) is a homeomorphism~\(\varphi\) from an open subset~\(U\) of~\(\Tot\) onto \(V\times\R^n\) for some open subset \(V\subseteq\Base\) and some \(n\in\N\), such that \(\pi_1\circ\varphi = \anchor\).  Two charts are compatible if the coordinate change map is fibrewise smooth.  A \emph{bundle of smooth manifolds over~\(\Base\)} is a space~\(\Tot\) over~\(\Base\) with a maximal compatible family of charts whose domains cover~\(\Tot\) (see also \cite{Emerson-Meyer:Dualities}*{\S7}; this does not quite imply that the bundle is locally trivial unless its total space is compact).
\end{definition}

\begin{definition}[\cite{Emerson-Meyer:Dualities}*{\S7}]
  \label{def:smooth_G-manifold}
  A \emph{smooth \(\Grd\)\nb-manifold} is a bundle of smooth manifolds over~\(\Base\) on which~\(\Grd\) acts continuously by fibrewise smooth maps.
\end{definition}

We also consider bundles of smooth manifolds with boundary.

A \emph{smooth \(\Grd\)\nb-manifold with boundary} is defined like a smooth \(\Grd\)\nb-manifold, but also allowing charts taking values in a half-space \(V\times\R^{n-1}\times[0,\infty)\) for \(V\subseteq\Base\).  If~\(\Tot\) is a smooth \(\Grd\)\nb-manifold with boundary, then its boundary~\(\bd\Tot\) is a smooth \(\Grd\)\nb-manifold.  Furthermore, there is a \emph{collar neighbourhood} around the boundary, that is, the embedding \(\bd\Tot\to\Tot\) extends to a \(\Grd\)\nb-equivariant open embedding \(\bd\Tot\times[0,1) \opem \Tot\).  Therefore, we may form a smooth \(\Grd\)\nb-manifold without boundary
\[
\Tot^\circ \defeq \Tot \cup_{\bd\Tot} (-\infty,0]\times\bd\Tot.
\]
Moreover, the coordinate projection \((-\infty,0]\times\bd\Tot\to(-\infty,0]\) extends to a smooth \(\Grd\)\nb-invariant function \(h\colon \Tot^\circ\to\R\) with \(h(\tot)>0\) for all \(\tot\in\Tot\setminus\bd\Tot\).

We let~\(\Tvert\Tot\) denote the \emph{vertical tangent bundle} of a bundle of smooth manifolds~\(\Tot\) (with boundary); on the domain of a chart \(U\cong V\times\R^{n-1}\times[0,\infty)\) for \(V\subseteq\Base\), we have \(\Tvert\Tot|_U \cong V\times\R^n\).  The projection \(\Tvert\Tot\epi\Tot\), the addition \(\Tvert\Tot\times_\Base \Tvert\Tot \to \Tvert\Tot\), and the scalar multiplication \(\R\times\Tvert\Tot \to \Tvert\Tot\) in~\(\Tvert\Tot\) are fibrewise smooth.

If~\(\Tot\) is a smooth \(\Grd\)\nb-manifold, then~\(\Grd\) acts on~\(\Tvert\Tot\) by fibrewise smooth maps, that is, \(\Tvert\Tot\) is a smooth \(\Grd\)\nb-manifold.  Furthermore, \(\Tvert\Tot\) is a \(\Grd\)\nb-vector bundle over~\(\Tot\).

\begin{definition}
  \label{def:Riemannian_metric}
  A \emph{Riemannian metric} on a bundle of smooth manifolds is a fibrewise smooth inner product on the vertical tangent bundle.
\end{definition}

The same argument as in the proof of Proposition~\ref{pro:vb_inner_product} shows that a fibrewise smooth \(\Grd\)\nb-vector bundle over a smooth \(\Grd\)\nb-manifold carries a fibrewise smooth inner product.  Thus any smooth \(\Grd\)\nb-manifold carries a Riemannian metric.  Even more, this metric may be chosen complete (\cite{Emerson-Meyer:Dualities}*{Lemma 7.7} shows how to achieve completeness).

\begin{example}
  \label{ex:smooth_bundles_of_manifolds_triv}
  If~\(\Grd\) is a compact group, then a smooth \(\Grd\)\nb-manifold is a smooth manifold with a smooth \(\Grd\)\nb-action.  Here the tangent space and Riemannian metrics have their usual meanings.

  Recall that we always replace a non-compact group~\(G\) by the proper groupoid \(G\ltimes \EG G\).  If~\(\Source\) is a smooth manifold with a smooth \(G\)\nb-action, then \(\Source\times\EG G\) is a smooth \(G\ltimes\EG G\)\nb-manifold.  Its tangent space is \(\Tvert\Source\times \EG G\), and a Riemannian metric is a family of Riemannian metrics on~\(\Source\) parametrised by points of~\(\EG G\).  This exists even if~\(\Source\) carries no \(G\)\nb-invariant Riemannian metric.
\end{example}

\begin{definition}
  \label{def:smooth_embedding}
  A \emph{smooth embedding} between two smooth \(\Grd\)\nb-manifolds \(\Source\) and~\(\Target\) is a \(\Grd\)\nb-equivariant embedding \(f\colon \Source\to\Target\) (homeomorphism onto its image) which is, in addition, fibrewise smooth with injective fibrewise derivative \(Df\colon \Tvert\Tot\to f^*(\Tvert\Target)\); the cokernel of~\(Df\) is called the \emph{normal bundle} of~\(f\).

  If \(\Source\) and~\(\Target\) have boundaries, then for an embedding \(f\colon \Source \to \Target\) we require also that \(f (\bd\Source) = f(\Source) \cap \bd \Target\) and that \(f(\Source)\) is transverse to~\(\bd\Target\).  The normal bundle~\(\Normal_f\) in the case of manifolds-with-boundary is defined in the same way.
\end{definition}

\begin{theorem}
  \label{the:tubular_for_embedding}
  Let \(\Source\) and~\(\Target\) be smooth \(\Grd\)\nb-manifolds with boundary.  Let \(f\colon \Source\to\Target\) be a smooth \(\Grd\)\nb-equivariant embedding with normal bundle~\(\Normal_f\).  There is a smooth open embedding \(\hat{f}\colon \Normal_f \to \Target\) with \(\hat{f}\circ \zers{\Normal_f} = f\); that is, \(\hat{f}\) is a fibrewise diffeomorphism onto its range.
\end{theorem}

\begin{proof}
  We generalise the well-known argument in the non-equivariant case (see~\cite{Spivak:Diffgeo1}*{p.~9-60}).  For simplicity, we only prove the result where \(\Source\) and~\(\Target\) have no boundaries.  Equip~\(\Target\) with a \(\Grd\)\nb-invariant complete Riemannian metric.  It generates a \(\Grd\)\nb-equivariant exponential map and a \(\Grd\)\nb-equivariant section for the vector bundle extension \(\Tvert\Source \mono f^*(\Tvert\Target) \epi \Normal_f\).  We compose this section with the exponential map to get a fibrewise smooth \(\Grd\)\nb-map \(h\colon \Normal_f\to\Target\).

  On the zero section of~\(\Normal_f\), the map~\(h\) is injective and its fibrewise derivative is invertible.  We claim that there is an open neighbourhood~\(U'\) of the zero section in~\(\Normal_f\) such that~\(h\) is injective and has invertible derivative on~\(U'\) -- we briefly say that~\(h\) is invertible on~\(U'\).  To begin with, the subset where~\(Dh\) is invertible is open and hence an open neighbourhood~\(U''\) of the zero section in~\(\Normal_f\).

  The restriction of~\(h\) to~\(U''\) is a fibrewise local diffeomorphism.  In local charts, we may estimate \(h(\source',\normal')-h(\source'',\normal'')\) from below using \((Dh)^{-1}\).  Therefore, each \((\source,\normal)\in\Normal_f\) has an open neighbourhood \(U(\source,\normal)\subseteq U''\) on which~\(h\) is injective.

  Let \(\source\in\Source\).  Then there are \(\varrho_1(\source)>0\) and an open neighbourhood \(U_1(\source)\) of~\(\source\) in~\(\Source\) such that all \((\source_1,\normal_1)\) with \(\source_1\in U_1(\source)\) and \(\norm{\normal_1} < \varrho_1(\source)\) belong to \(U(\zers{\Normal_f}\source)\).  Since \(f\colon \Source\to\Target\) is an embedding, there is an open subset~\(V(\source)\) of~\(\Target\) with \(f^{-1}(V(\source))= U_1(\source)\).  There are a smaller open neighbourhood~\(U_2(\source)\) of~\(\source\) in~\(\Source\) and \(\varrho_2(\source)\in (0,\varrho_1(\source))\) with \(\target\in V(\source)\) whenever there is \(\source_2\in U_2(\source)\) such that \(d\bigl(f(\source_2), \target\bigr)<\varrho_2(\source)\); here~\(d\) denotes the fibrewise distance with respect to the Riemannian metric on~\(\Target\).

  Let \(U_3(\source)\subseteq \Normal_f\) be the set of all \((\source_3,\normal_3)\in \Normal_f\) with \(\source_3\in U_2(\source)\) and \(\norm{\normal_3}<\varrho_2/2\).  We claim that~\(h\) is invertible on \(U'\defeq \bigcup_{\source} U_3(\source)\).  The invertibility of~\(Dh\) follows because \(U'\subseteq U''\).  Assume that \((\source,\normal), (\source',\normal')\in U'\) satisfy \(h(\source,\normal) = h(\source',\normal')\).  Choose \(\bar\source,\bar\source'\in\Source\) with \((\source,\normal)\in U_3(\bar\source)\) and \((\source',\normal')\in U_3(\bar\source')\).  We may exchange \((\source,\normal)\) and~\((\source',\normal')\) if necessary, so that \(\varrho_2(\bar\source)\ge \varrho_2(\bar\source')\).  The triangle inequality and the definition of~\(U_3(\bar\source)\) yield
  \[
  d\bigl(f(\source), f(\source')\bigr) \le
  d\bigl(f(\source), h(\source,\normal)\bigr) +
  d\bigl(h(\source',\normal'), f(\source')\bigr) <
  \varrho_2(\bar\source).
  \]
  Hence \(f(\source')\in V(\bar\source)\), so that \(\source'\in U_1(\bar\source)\).  Since \(\norm{\normal},\norm{\normal'}<\varrho_1(\bar\source)\) as well, both \((\source,\normal)\) and \((\source',\normal')\) belong to \(U(\zers{\Normal_f}\bar\source)\).  Hence \((\source,\normal) = (\source',\normal')\) as desired.

  We have constructed a neighbourhood of the zero section in~\(\Normal_f\) on which~\(h\) is invertible.  Next we claim that there is a \(\Grd\)\nb-invariant smooth function \(\varrho\colon \Source\to(0,\infty)\) such that~\(U'\) contains the open neighbourhood
  \[
  U_\varrho\defeq \bigl\{(\source,\normal)\in \Normal_f \bigm|
  \norm{\normal}<\varrho(\source) \bigr\}.
  \]
  Then \(\hat{f}(\source,\normal) \defeq h\bigl(\source, \normal \varrho(\source)/ (1+\norm{\normal})\bigr)\) has the required properties.  We first construct a possibly non-equivariant function~\(\varrho\) with \(U_\varrho\subseteq U'\), using that~\(\Source\) is paracompact by our standing assumption on all topological spaces.  Then we make~\(\varrho\) \(\Grd\)\nb-equivariant by averaging as in the proof of Proposition~\ref{pro:extend_section_vb}.  Since the average is bounded above by the maximum, the averaged function still satisfies \(U_\varrho\subseteq U'\) as needed.
\end{proof}

\subsection{Embedding theorems}
\label{sec:embedding}

Recall that a \(\Grd\)\nb-space for a compact group~\(\Grd\) has \emph{finite orbit type} if only finitely many conjugacy classes of subgroups of~\(\Grd\) appear as stabilisers.  This condition is closely related to equivariant embeddability:

\begin{theorem}[George Mostow~\cite{Mostow:Equivariant_embeddings}, Richard Palais~\cite{Palais:Imbedding}]
  \label{the:Mostow}
  Let~\(\Grd\) be a compact group.  A \(\Grd\)\nb-space admits a \(\Grd\)\nb-equivariant embedding into a linear representation of~\(\Grd\) if and only if it has finite orbit type and finite covering dimension and is locally compact and second countable.
\end{theorem}

\begin{example}
  \label{exa:counterexample_not_embeddable}
  Let \(\Grd=\Torus\) and consider the disjoint union
  \[
  \Tot\defeq \bigsqcup_{n=1}^\infty \Torus \bigm/ \{\exp(2\pi\ima k/n) \mid k=0,\dotsc,n-1\},
  \]
  on which~\(\Torus\) acts by multiplication in each factor.  All the finite cyclic subgroups appear as stabilisers in~\(\Tot\).  But a linear representation of~\(\Torus\) can only contain finitely many orbit types by \cite{Palais:Slices}*{Theorem 4.4.1}.  Hence there is no injective \(\Torus\)\nb-equivariant map from~\(\Tot\) to any linear representation of~\(\Torus\).
\end{example}

Now let~\(\Grd\) be a numerably proper groupoid.  Let~\(\Tot\) be a smooth \(\Grd\)\nb-manifold with anchor map \(\anchor\colon \Tot\to\Base\).  Subject to some conditions, we will construct a smooth embedding of~\(\Tot\) into the total space of a \(\Grd\)\nb-vector bundle over~\(\Base\).  This is the natural way to extend Mostow's Embedding Theorem to proper groupoids.  We may also ask for embeddings into~\(\R^n\) with some linear representation as in~\cite{Kankaanrinta:Embeddings}.  But such embeddings need not exist because a general locally compact group need not have any non-trivial finite-dimensional representations.  And even if they exist, linear actions on~\(\R^n\) are never proper, so that we leave the world of proper actions with such embedding theorems.

\begin{lemma}
  \label{lem:embedding_general}
  Let~\(\Tot\) be a smooth \(\Grd\)\nb-manifold with anchor map \(\anchor\colon \Tot\to\Base\).  Assume that~\(\Grd\backslash\Tot\) has finite covering dimension.

  Then~\(\Tot\) admits a fibrewise smooth equivariant embedding into the total space of an equivariant vector bundle over~\(\Base\) if and only if there is a \(\Grd\)\nb-vector bundle~\(\Triv\) over~\(\Base\) such that, for any \(\tot\in\Tot\), the fibre~\(\Triv_{\anchor(\tot)}\) contains
  \begin{itemize}
  \item a vector whose stabiliser in~\(\Grd_{\anchor(\tot)}^{\anchor(\tot)}\) is equal to~\(\Grd_\tot^\tot\);

  \item the representation of~\(\Grd_\tot^\tot\) on the vertical tangent bundle \(\Tvert_\tot\Tot\).
  \end{itemize}
\end{lemma}

\begin{proof}
  We must construct a \(\Grd\)\nb-equivariant smooth embedding \(f\colon \Source\to\Triv'\) for some \(\Grd\)\nb-vector bundle~\(\Triv'\) over~\(\Base\); then the Tubular Neighbourhood Theorem~\ref{the:tubular_for_embedding} provides an open smooth embedding \(\hat{f}\colon \Normal_f\opem\Triv'\).

  If there is a \(\Grd\)\nb-equivariant smooth embedding \(\Tot\to\Triv\), then~\(\Triv\) has the two properties required above.  Conversely, we will construct a \(\Grd\)\nb-equivariant smooth embedding \(\Tot\to\Triv^{2n}\oplus\R^n\) for some \(n\in\N\) assuming the existence of a vector bundle~\(\Triv\) as above.

  Let~\(\Triv\) be a \(\Grd\)\nb-vector bundle over~\(\Base\).  A \(\Grd\)\nb-equivariant smooth map \(\hat{f}\colon \Tot\to\Triv\) is equivalent to a \(\Grd\)\nb-equivariant smooth section of the vector bundle~\(\Triv^\Tot\).

  First we construct such a section locally near a single orbit.  Let \(\tot\in\Tot\) and let \(\base\defeq \anchor(\tot)\in\Base\).  Then \(\Tot_\base\defeq \anchor^{-1}(\base)\) is a smooth manifold, on which the stabiliser~\(\Grd_\base^\base\) of~\(\base\) acts by diffeomorphisms.  The group~\(\Grd_\base^\base\) is compact because~\(\Grd\) is proper, and~\(\Grd_\tot^\tot\) is a closed subgroup of~\(\Grd_\base^\base\).

  Our assumptions on \(\Grd\) and~\(\Tot\) imply that there is a \(\Grd\)\nb-invariant Riemannian metric on the fibres of~\(\Tot\).  The action of~\(\Grd_\base^\base\) on~\(\Tot_\base\) is isometric for this fibrewise Riemannian metric, so that it factors through a Lie group.  Therefore, the orbit \(\Grd_\base^\base\cdot\tot\) is an embedded submanifold of~\(\Tot_\base\); let~\(\Normal\) be its normal bundle and let~\(\Normal_\tot\) be the fibre of this normal bundle at~\(\tot\).  Then \(\Normal = \Grd_\base^\base \times_{\Grd_\tot^\tot} \Normal_\tot\).  The exponential map provides a \(\Grd_\base^\base\)\nb-equivariant diffeomorphism between~\(\Normal\) and an open neighbourhood of the orbit~\(\Grd_\base^\base\tot\) in~\(\Tot_\base\).

  A \(\Grd_\base^\base\)\nb-equivariant section of~\(\Triv^\Normal\) is equivalent to a \(\Grd_\tot^\tot\)\nb-equivariant section of~\(\Triv^{\Normal_\tot}\) because \(\Normal = \Grd_\base^\base \times_{\Grd_\tot^\tot} \Normal_\tot\).  This is equivalent to a \(\Grd_\tot^\tot\)\nb-equivariant map \(\Normal_\tot\to\Triv_\base\) because we work in a single fibre, where~\(\Triv\) is constant.  By assumption, there is a vector \(\triv\in\Triv_\base\) whose stabiliser is exactly~\(\Grd_\tot^\tot\), and there is a linear \(\Grd_\tot^\tot\)\nb-invariant embedding \(l\colon \Normal_\tot\subseteq \Tvert_\tot\Tot \to \Triv_\base\).  Then the \(\Grd_\tot^\tot\)\nb-equivariant map \(\Normal_\tot\to\Triv_\base\oplus\Triv_\base\) that maps \(\normal\mapsto \bigl(\triv,l(\normal)\bigr)\) leads to a \(\Grd_\base^\base\)\nb-equivariant map \(\hat{f}_\tot\colon \Normal\to\Triv_\base\oplus\Triv_\base\) that is injective and has injective derivative on the \(\Grd_\base^\base\)\nb-orbit of~\(\tot\).  We check injectivity: suppose that \(\normal, \normal' \in \Normal\) are points with the same image in \(\Triv_\base\oplus \Triv_\base\).  They need not be in the same fibre of~\(\Normal\), of course, so let~\(\normal\) be in the fibre over~\(g\tot\) and~\(\normal'\) be in the fibre over~\(g'\tot\).  If they have the same image, then \(ge=g'e\) and hence \(g\tot = g'\tot\) because by assumption the stabiliser of~\(e\) is exactly~\(\Grd^\tot_\tot\).  Thus \(\nu\) and~\(\nu'\) are in the same fibre of~\(\Normal\).  But our map is clearly an embedding on each fibre of~\(\Normal\) because it is conjugate to the map on~\(\Normal_\tot\) we started with.  Hence \(\nu = \nu'\).

  It follows that there is a \(\Grd_\base^\base\)\nb-invariant open neighbourhood~\(W_\tot\) of~\(\tot\) in~\(\Tot_\base\) on which~\(\hat{f}_\tot\) is a smooth embedding.

  So far, we have worked on a single fibre~\(\Tot_\base\).  Next we use a smooth version of Proposition~\ref{pro:extend_section_vb} to extend our map to a \(\Grd\)\nb-equivariant, fibrewise smooth, continuous map \(\hat{f}_\tot\colon \Tot\to\Triv\oplus\Triv\) that is a smooth embedding on the orbit of~\(\tot\).  We claim that it remains a smooth embedding on some \(\Grd\)\nb-invariant neighbourhood of~\(\tot\).  The proof of this observation is similar to the corresponding argument in the proof of the Tubular Neighbourhood Theorem~\ref{the:tubular_for_embedding}.  Hence we omit further details.

  We have constructed a \(\Grd\)\nb-equivariant smooth embedding \(\hat{f}_\tot\colon U_\tot\to\Triv\oplus\Triv\) on a \(\Grd\)\nb-invariant neighbourhood~\(U_\tot\) of~\(\tot\) for each \(\tot\in\Tot\).  We may add the constant map~\(1\) with values in the constant \(1\)\nb-dimensional bundle and rescale to get a smooth embedding~\(\hat{f}_\tot\) from~\(U_\tot\) into the unit sphere bundle of \(\Triv'\defeq \Triv\oplus\Triv\oplus\R\).  Hence we assume in the following that~\(\hat{f}_j\) is a smooth embedding into the unit sphere bundle of~\(\Triv'\).  Now we patch these local solutions together to a global one.

  We view the \(\Grd\)\nb-invariant neighbourhoods~\(U_\tot\) as an open covering of the orbit space \(\Grd\backslash\Tot\).  Since the latter is finite-dimensional, we may refine this covering to one with finite Lebesgue number.  Since our construction uses the same target vector bundle~\(\Triv'\) for each \(\tot\in\Tot\), we may combine our local solutions on disjoint open subsets without any problems.  Thus we may assume that we have a finite \(\Grd\)\nb-invariant covering \(U_0,\dotsc,U_n\) of~\(\Tot\) and embeddings \(\hat{f}_j\colon U_j\to\Triv'\).  Proposition~\ref{pro:invariant_partition_unity} provides a \(\Grd\)\nb-invariant partition of unity \(\varphi_0,\dotsc,\varphi_n\) on~\(\Tot\) subordinate to this covering.  We can arrange for the functions~\(\varphi_n\) to be fibrewise smooth.  We let \(\hat{f}\defeq \bigoplus \varphi_j\hat{f}_j\colon \Tot\to(\Triv')^{n+1}\).  This map is injective and a fibrewise immersion because all the maps~\(\hat{f}_j\) are embeddings into the unit sphere bundles of~\(\Triv'\).
\end{proof}

\begin{theorem}
  \label{the:embedding_full}
  Let~\(\Tot\) be a smooth \(\Grd\)\nb-manifold.  Assume that~\(\Grd\backslash\Tot\) has finite covering dimension, and that there is a full \(\Grd\)\nb-vector bundle on~\(\Base\).  Then~\(\Tot\) admits a fibrewise smooth \(\Grd\)\nb-equivariant embedding into the total space of a \(\Grd\)\nb-vector bundle over~\(\Base\).
\end{theorem}

\begin{proof}
  Let~\(\Triv\) be a full \(\Grd\)\nb-vector bundle over~\(\Base\) and let~\(n\)~be the maximum of the ranks of \(\Tvert\Tot\) and~\(\Triv\).  As in the proof of Theorem~\ref{the:vb_subtrivial_full}, the fibres of the vector bundle~\(\Triv^n\) contain the induced representations of~\(\Grd_\tot^\tot\) on~\(\Tvert_\tot\Tot\) for all \(\tot\in\nobreak\Tot\).  Furthermore, any closed subgroup of~\(\Grd_\tot^\tot\) is the stabiliser of a vector in the regular representation of~\(\Grd_\tot^\tot\) and hence of a vector in~\(\Triv^n_{\anchor(\tot)}\); here we use Lemma~\ref{lem:full_finite}.  Hence~\(\Triv^n\) satisfies the two assumptions of Lemma~\ref{lem:embedding_general}, which yields the desired embedding.
\end{proof}

\begin{theorem}
  \label{the:embedding_cocompact}
  Let~\(\Grd\) be a numerably proper groupoid with enough \(\Grd\)\nb-vector bundles over its object space~\(\Base\).  Then any cocompact smooth \(\Grd\)\nb-manifold~\(\Tot\) admits a fibrewise smooth embedding into the total space of a \(\Grd\)-vector bundle over~\(\Base\).
\end{theorem}

\begin{proof}
  Argue as in the proof of Lemma~\ref{lem:embedding_general}, but using different vector bundles in the local construction.  The compactness of~\(\Grd\backslash\Tot\) ensures that this creates no problems, compare the proof of Theorem~\ref{the:vb_subtrivial_cocompact}.
\end{proof}

\begin{example}
  \label{exa:not_enough_vb_conclusion}
  We continue from Example~\ref{exa:not_enough_vb} and consider the groupoid~\(\Grd_A\) with \(A \defeq \bigl(\begin{smallmatrix}2&1\\1&1\end{smallmatrix}\bigr)\).  In this case, \(\Grd_A\) acts trivially on any equivariant vector bundle over~\(\Base\) and hence on any space that embeds equivariantly into one.  But there are non-trivial smooth actions of~\(\Grd_A\) on bundles of smooth manifolds over~\(\Base\).  An obvious example is~\(\Grd_A\) itself with the smooth action of~\(\Grd_A\) by translations.
 \end{example}

\subsection{The Factorisation Theorem}
\label{sec:factorisation_theorem}

The factorisations in the following theorem motivate our definition of a normally non-singular map.

\begin{theorem}
  \label{the:normal_factorisation}
  Let~\(\Grd\) be a \textup(proper\textup) groupoid, let \(\Source\) and~\(\Target\) be smooth \(\Grd\)\nb-manifolds, and let \(f\colon \Source \to \Target\) be a smooth \(\Grd\)\nb-equivariant map.  Suppose that
  \begin{enumerate}[label=\textup{(\alph{*})}]
  \item \(\Base\) has enough \(\Grd\)\nb-vector bundles and \(\Grd\backslash\Source\) is compact, or
  \item \(\Base\) has a full \(\Grd\)\nb-vector bundle and \(\Grd\backslash\Source\) has finite covering dimension.
  \end{enumerate}

  Then there are
  \begin{itemize}
  \item a smooth \(\Grd\)-vector bundle~\(\VB\) over~\(\Source\),
  \item a smooth \(\Grd\)-vector bundle~\(\Triv\) over~\(\Base\),
  \item a smooth, \(\Grd\)\nb-equivariant, open embedding \(\hat{f}\colon \total{\VB}\to \total{\Triv^\Target}\),
  \end{itemize}
  such that
  \begin{equation}
    \label{eq:normal_factorisation}
    f = \proj{\Triv^\Target}\circ \hat{f} \circ \zers{\VB}.
  \end{equation}
  Furthermore, any \(\Grd\)\nb-vector bundle over~\(\Source\) is subtrivial.
\end{theorem}

We call a factorisation of the form~\eqref{eq:normal_factorisation} a \emph{normal factorisation} of~\(f\).

\begin{proof}
  Let \(\varrho \colon \Source \to \total{\Triv}\) be a fibrewise smooth embedding into the total space of a \(\Grd\)\nb-vector bundle over~\(\Base\).  This exists by Theorem~\ref{the:embedding_cocompact} in the first case and by Theorem~\ref{the:embedding_full} in the second case.  Then the map
  \[
  \varrho'\colon \Source\to \Target\times_\Base \total{\Triv} = \total{\Triv^\Target},\qquad
  \varrho'(x) \defeq  \bigl(f(x), \varrho(x) \bigr)
  \]
  is a smooth equivariant embedding.  The Tubular Neighbourhood Theorem~\ref{the:tubular_for_embedding} applied to~\(\varrho'\) supplies the required factorisation of~\(f\).  The subtriviality of \(\Grd\)\nb-vector bundles over~\(\Source\) follows from Theorems \ref{the:vb_subtrivial_cocompact} and~\ref{the:vb_subtrivial_full}.
\end{proof}

We will discuss the amount of uniqueness of such factorisations in Theorem~\ref{the:normal_map_unique}.

\section{Normally non-singular maps}
\label{sec:normal_maps}

In this section, we study factorisations of maps as in Theorem~\ref{the:normal_factorisation}.  Recall that~\(\Grd\) is a numerably proper groupoid with object space~\(\Base\).  The spaces \(\Source\) and~\(\Target\) are \(\Grd\)\nb-spaces.

Before we define normally non-singular maps, we mention a generalisation that we will use later, which depends on an additional auxiliary structure.  Let \(\Tot\mapsto \Vect_\Grd(\Tot)\) be a functor that maps each \(\Grd\)\nb-space to a monoid \(\Vect_\Grd(\Tot)\), which comes together with an additive functor to the monoid of \(\Grd\)\nb-vector bundles over~\(\Tot\).  We think of \(\Vect_\Grd(\Tot)\) as a monoid of \(\Grd\)\nb-vector bundles with some additional structure, and of the functor above as a forgetful functor.

In this section, we let \(\Vect_\Grd(\Tot)\) be the monoid of \emph{subtrivial} \(\Grd\)\nb-vector bundles over~\(\Tot\).  The subtriviality assumption will become important in~\cite{Emerson-Meyer:Correspondences}, and it is mostly harmless because for many \(\Grd\)\nb-spaces all \(\Grd\)\nb-vector bundles are subtrivial (see Section~\ref{sec:vb_subtrivial}).  When we study \(\Coh\)\nb-oriented normally non-singular maps for an equivariant cohomology theory~\(\Coh\), the only change will be that \(\Vect_\Grd(\Tot)\) is replaced by the monoid of subtrivial \emph{\(\Coh\)\nb-oriented} \(\Grd\)\nb-vector bundles over~\(\Tot\).  This example explains why we want to choose other monoids \(\Vect_\Grd(\Tot)\) later.  The whole theory carries over to the more general case of an arbitrary monoid \(\Vect_\Grd(\Tot)\) with a forgetful map to the monoid of \(\Grd\)\nb-vector bundles.

We let \([\Vect^0_\Grd(\Source)]\) be the Grothendieck group of \(\Vect_\Grd(\Source)\).

\begin{definition}
  \label{def:normal_map}
  A \emph{normally non-singular \(\Grd\)\nb-map} from~\(\Source\) to~\(\Target\) consists of the following data:
  \begin{itemize}
  \item \(\VB\), a subtrivial \(\Grd\)\nb-vector bundle over~\(\Source\), that is, \(\VB\in\Vect_\Grd(\Tot)\);

  \item \(\Triv\), a \(\Grd\)\nb-vector bundle over~\(\Base\), that is, \(\Triv\in\Vect_\Grd(\Base)\);

  \item \(\hat{f}\colon \total{\VB} \opem \total{\Triv^\Target}\), an open embedding (that is, \(\hat{f}\) is a \(\Grd\)\nb-equivariant map from~\(\total{\VB}\) onto an open subset of \(\total{\Triv^\Target} = \total{\Triv}\times_\Base\Target\) that is a homeomorphism with respect to the subspace topology from~\(\total{\Triv^\Target}\)).
  \end{itemize}
  In addition, we assume that the dimensions of the fibres of the \(\Grd\)\nb-vector bundles \(\VB\) and~\(\Triv\) are bounded above by some \(n\in\N\).

  The \emph{trace} of a normally non-singular map is the \(\Grd\)\nb-map
  \[
  f\defeq \proj{\Triv^\Target}\circ\hat{f}\circ\zers{\VB}\colon
  \Source\mono\total{\VB}\opem\total{\Triv^\Target}\epi\Target.
  \]
  Its \emph{degree} is \(\dim \VB-\dim\Triv\) if this locally constant function on~\(\Source\) is constant (otherwise the degree is not defined).

  The normally non-singular \(\Grd\)\nb-map \((\VB,\Triv,\hat{f})\) is called a \emph{normally non-singular embedding} if \(\Triv=0\), so that \(\proj{\Triv^\Target} = \Id_{\Target}\) and \(f=\hat{f}\circ\zers{\VB}\); it is called a \emph{special normally non-singular submersion} if \(\VB=0\), so that \(\zers{\VB}=\Id_\Source\) and \(f=\proj{\Triv^\Target}\circ \hat{f}\).
\end{definition}

We need the vector bundle~\(\Triv\) over the target space to be trivial, even pulled back from~\(\Base\), in order for the composition of normally non-singular maps to work.  This requires extending the vector bundle over the target space from an open subset.  This is easy for trivial bundles, and impossible in general.  The assumption that the vector bundle~\(\VB\) be subtrivial may be dropped for the purposes of this article.  It becomes important in~\cite{Emerson-Meyer:Correspondences} in order to bring correspondences into a standard form.

We abbreviate ``normally non-singular \(\Grd\)\nb-map'' to ``normally non-singular map'' if the groupoid~\(\Grd\) is clear.

\begin{definition}
  \label{def:stable_normal_bundle}
  The \emph{stable normal bundle} of a normally non-singular map \((\VB,\Triv,\hat{f})\) with trace~\(f\) is the class \([\VB]-[\Triv^\Source]\) in \([\Vect^0_\Grd(\Source)]\).
\end{definition}

Notice that the degree of a normally non-singular map depends only on its stable normal bundle.

We do not require any smooth structure on the spaces \(\Source\) and~\(\Target\) and, as a result, cannot ask for~\(\hat{f}\) to be a diffeomorphism.  We pay for this lack of smoothness by making \(\VB\) and~\(\Triv\) part of our data.  The following simple examples clarify the relationship between smooth maps and normally non-singular maps in the non-equivariant case.  We will examine the equivariant case in Section~\ref{sec:smooth_normal_map}.  Here we only remark that if~\(\Grd\) is a compact group, then a \(\Grd\)\nb-manifold~\(\Source\) admits a smooth normally non-singular map to a point if and only if it has finite orbit type, due to the Mostow Embedding Theorem~\ref{the:Mostow}.

\subsection{Examples of normally non-singular maps}
\label{sec:exa_normal}

\begin{example}
  An open \(\Grd\)\nb-equivariant embedding \(\Source \to \Target\) is the trace of an obvious \(\Grd\)\nb-equivariant normally non-singular map: both vector bundles \(\VB\) and~\(\Triv\) are the zero bundles.
\end{example}

\begin{example}
  \label{exa:embed_smooth_submanifold}
  Let~\(\Target\) be a smooth manifold and let \(\Source\subseteq\Target\) be a smooth submanifold.  Let~\(\VB\) be its normal bundle.  The Tubular Neighbourhood Theorem~\ref{the:tubular_for_embedding} provides an open embedding \(\hat{f}\colon \total{\VB}\opem\Target\) extending the embedding of~\(\Source\) on the zero section.  The triple \((\VB,0,\hat{f})\) is a normally non-singular embedding from~\(\Source\) to~\(\Target\), whose trace is the inclusion map \(\Source\to\Target\).  Thus normally non-singular embeddings generalise closed submanifolds.
\end{example}

\begin{example}
  \label{exa:special_normal_submersion}
  The constant map \(\Source\to \pt\) is the trace of a special normally non-singular submersion if and only if~\(\Source\) is homeomorphic to an open subset of~\(\R^n\).  Thus our special normally non-singular submersions are very special indeed.
\end{example}

\begin{example}
  \label{exa:smooth_normal_map}
  Let \(f\colon \Source\to\Target\) be a smooth map between two smooth manifolds.  Recall that any smooth manifold is diffeomorphic to a smooth submanifold of~\(\R^n\) for sufficiently high~\(n\).  If \(h\colon \Source\to\R^n\) is such an embedding, then \((f,h)\colon \Source\to\Target\times\R^n\) is an embedding as well.  As in Example~\ref{exa:embed_smooth_submanifold}, the Tubular Neighbourhood Theorem provides a diffeomorphism~\(\hat{f}\) from~\(\total{\VB}\) onto an open subset of \(\Target\times\R^n\), where~\(\VB\) is the normal bundle of \((f,h)\).  Thus \((\VB,\R^n,\hat{f})\) is a normally non-singular map from~\(\Source\) to~\(\Target\) with trace~\(f\).  Its stable normal bundle is \(f^*[\Tvert\Target]-[\Tvert\Source]\) because \(\Tvert\Source\oplus \VB \cong f^*(\Tvert\Target)\oplus\R^n\).
\end{example}

\begin{example}
  \label{exa:point_to_interval}
  The map \(\pt\to\{0\}\in[0,\infty)\) from the one-point space to \([0,\infty)\) is not the trace of a normally non-singular map from~\(\pt\) to~\([0,\infty)\).  For this would entail an open embedding \(\R^k \xrightarrow{\subseteq} [0,\infty) \times \R^l\) into a closed Euclidean half-space mapping \(0 \in \R^k\) to a point on the boundary.  The range of such an embedding would be an open neighbourhood of the origin in the closed half-space, which is homeomorphic to~\(\R^k\), and this is impossible by the theorem of Invariance of Domain.
\end{example}

\begin{example}
  \label{exa:non-singular_normal}
  It is observed in~\cite{Baum-Block:Bicycles} that normally non-singular maps between stratified singular algebraic varieties are normally non-singular in the sense of Definition~\ref{def:normal_map}.
\end{example}

\begin{example}
  \label{exa:stable_smooth_structure}
  We claim that a normally non-singular map~\(\Source \to \pt\) determines a smooth structure on \(\Source\times\R^k\) for some \(k\in\N\), and vice versa.  In the equivariant case, a normally non-singular map from~\(\Source\) to~\(\Base\) should therefore be viewed as a \emph{stable smooth structure} on the fibres of~\(\Source\) compatible with the action of~\(\Grd\).

  Let \((\VB,\Triv,\hat{f})\) be a normally non-singular map from a space~\(\Source\) to the one-point space.  That is, \(\VB\) is a subtrivial vector bundle over~\(\Source\), \(\Triv=\R^n\) for some \(n\in\N\), and~\(\hat{f}\) is an open embedding of~\(\total{\VB}\) into~\(\R^n\).  Let~\(\VB^\bot\) be another vector bundle over~\(\Source\) such that \(\VB\oplus\VB^\bot\cong\Source\times\R^k\) for some \(k\in\N\).  The total space~\(\total{\VB}\) is an open subset of~\(\R^k\) and \(\total{\VB\oplus\VB^\bot}\) is the total space of a vector bundle over~\(\total{\VB}\).  The total space of any vector bundle over a smooth manifold admits a canonical smooth structure, by an easy argument with the holomorphically closed subalgebra \(\CONT^\infty(\Tot)\) in \(\CONT(\Tot)\).  Hence, the space \(\Source\times\R^k\) inherits a canonical smooth structure.  Thus a normally non-singular map from~\(\Source\) to~\(\pt\) determines a smooth structure on \(\Source\times\R^k\) for some \(k\in\N\).

  Conversely, if \(\Source\times\R^k\) has a smooth structure, then there is a smooth embedding \(\Source\times\R^k\to\R^n\) for some \(n\in\N\).  Since \(\Tot \times \R^k \to \Tot\) is a homotopy equivalence, the normal bundle to this embedding is isomorphic to the pull-back \(\pr_\Tot^*(\VB)\) of a vector bundle~\(\VB\) over~\(\Tot\).  Hence there is an open embedding \(\hat{\mapr}\colon \total{\pr_\Tot^*(\VB)} \opem \R^n\), and since \(\total{\pr_\Tot^*(\VB)}\) is the total space of the vector bundle \(\VB\oplus \R^k\), we can regard~\(\hat{\mapr}\) as an open embedding \(\total{\VB\oplus \R^k}\opem \R^n\), so that \((\VB\oplus \R^k, \R^n, \hat{f})\) is a normally non-singular map.  This suggests to view the set of equivalence classes of normally non-singular maps from~\(\Source\) to~\(\pt\) as a structure set of stable smooth structures on~\(\Source\).
\end{example}

\begin{example}
  \label{exa:smooth_manifold_diagonal}
  Let~\(\Source\) be a smooth manifold and let \(\Target\defeq \Source\times\Source\).  The diagonal map \(\Source\to\Source\times\Source\) is a smooth embedding.  Hence it is the trace of a normally non-singular embedding \(\NM=(\VB,0,\hat{f})\) as in Example~\ref{exa:embed_smooth_submanifold}.  Here~\(\VB\) is the vertical tangent bundle~\(\Tvert\Source\) of~\(\Source\).  Since different smooth structures on~\(\Source\) may yield non-isomorphic tangent bundles (see~\cite{Milnor:Microbundles}), the stable normal bundle \([\Tvert\Source] \in \K^0(\Source)\) of~\(\NM\) tells us something about the smooth structure on~\(\Source\) and cannot be recovered from the trace of~\(\NM\).
\end{example}

\subsection{Equivalence of normally non-singular maps}
\label{sec:normal_map_equivalence}

A smooth map between two smooth manifolds lifts to a normally non-singular map in many different ways (see Example~\ref{exa:smooth_normal_map}), but we expect all these liftings to be equivalent in a suitable sense.  Here we develop a suitable notion of equivalence.

As above, let~\(\Grd\) be a numerably proper topological groupoid with object space~\(\Base\) and let \(\Source\) and~\(\Target\) be \(\Grd\)\nb-spaces.  Let
\[
\NM_0\defeq (\VB_0,\Triv_0,\hat{f}_0)
\quad\text{and}\quad
\NM_1\defeq (\VB_1,\Triv_1,\hat{f}_1)
\]
be normally non-singular \(\Grd\)\nb-maps from~\(\Source\) to~\(\Target\) (see Definition~\ref{def:normal_map}).

First we define isomorphism and stable isomorphism of normally non-singular maps.

\begin{definition}
  \label{def:normal_map_isomorphism}
  The normally non-singular \(\Grd\)\nb-maps \(\NM_0\) and~\(\NM_1\) are called \emph{isomorphic} if there are \(\Grd\)\nb-vector bundle isomorphisms \(\VB_0\cong\VB_1\) and \(\Triv_0\cong\Triv_1\) that intertwine the open embeddings \(\hat{f}_0\) and~\(\hat{f}_1\).
\end{definition}

\begin{definition}
  \label{def:normal_map_lift}
  Let~\(\Triv^+\) be a \(\Grd\)\nb-vector bundle over~\(\Base\).  The \emph{lifting} of a normally non-singular map \(\NM\defeq (\VB,\Triv,\hat{f})\) along~\(\Triv^+\) is the normally non-singular \(\Grd\)\nb-map
  \[
  \NM \oplus \Triv^+ \defeq
  \bigl(\VB\oplus (\Triv^+)^\Source, \Triv\oplus\Triv^+,
  \hat{f}\times_\Base \Id_{\total{\Triv^+}}\bigr).
  \]

  Two normally non-singular maps \(\NM_0\) and~\(\NM_1\) are called \emph{stably isomorphic} if there are \(\Grd\)\nb-vector bundles \(\Triv_0^+\) and~\(\Triv_1^+\) over~\(\Base\) such that \(\NM_0\oplus\Triv_0^+\) and \(\NM_1\oplus\Triv_1^+\) are isomorphic.
\end{definition}

The total spaces \(\total{\VB\oplus (\Triv^+)^\Source}\) and \(\total{\Triv\oplus\Triv^+}\) are \(\Grd\)\nb-equivariantly homeomorphic to \(\total{\VB}\times_\Base \total{\Triv^+}\) and \(\total{\Triv}\times_\Base \total{\Triv^+}\), respectively.  Hence the open embedding \(\hat{f}\times_\Base \Id_{\total{\Triv^+}}\) has the right source and target spaces.  Lifting~\(\NM\) first along~\(\Triv^+_1\) and then along~\(\Triv^+_2\) is equivalent to lifting~\(\NM\) along \(\Triv^+_1\oplus \Triv^+_2\), that is,
\[
(\NM \oplus \Triv^+_1) \oplus \Triv^+_2 \cong
\NM \oplus (\Triv^+_1 \oplus \Triv^+_2).
\]
Hence stable isomorphism is an equivalence relation for normally non-singular maps.  It is clear that stably isomorphic normally non-singular maps have the same trace, the same stable normal bundle, and the same degree.

\begin{definition}
  \label{def:normal_map_isotopy}
  An \emph{isotopy} (or homotopy) between the normally non-singular \(\Grd\)\nb-maps \(\NM_0\) and~\(\NM_1\) is a normally non-singular \(\Grd\times[0,1]\)-map from \(\Source\times[0,1]\) to \(\Target\times[0,1]\) whose restrictions to \(\Source\times\{t\}\) for \(t=0,1\) are isomorphic to~\(\NM_t\).

  Two normally non-singular maps are called \emph{isotopic} if there is an isotopy between them.
\end{definition}

Proposition~\ref{pro:vb_homotopy} shows that an isotopy is isomorphic to a triple \((\VB\times[0,1],\Triv,\hat{f})\) where \(\VB\) and~\(\Triv\) are \(\Grd\)\nb-vector bundles over \(\Source\) and~\(\Base\), respectively, and~\(\hat{f}\) is an open embedding of \(\total{\VB}\times[0,1]\) into \(\total{\Triv^\Target}\times[0,1]\).  As a result, isotopic normally non-singular maps involve isomorphic vector bundles \(\VB\) and~\(\Triv\) and thus have the same stable normal bundle and degree.  Moreover, their traces are homotopic.

\begin{definition}
  \label{def:normal_map_equivalent}
  Two normally non-singular maps from~\(\Source\) to~\(\Target\) are called \emph{equivalent} if they have isotopic liftings, that is, their liftings along two suitable \(\Grd\)\nb-vector bundles over~\(\Base\) become isotopic.
\end{definition}

\begin{lemma}
  \label{lem:normal_map_equivalent}
  Both isotopy and equivalence of normally non-singular maps are equivalence relations.
\end{lemma}

\begin{proof}
  Isotopy is reflexive because we have constant isotopies, symmetric because we may revert isotopies, and transitive because isotopies \(\NM_0\sim\NM_1\sim\NM_2\) assemble to an isotopy \(\NM_0\sim\NM_2\).  Equivalence of normally non-singular maps is clearly reflexive and symmetric as well.  If \(\NM_1\) and~\(\NM_2\) are isotopic normally non-singular maps and~\(\Triv\) is a \(\Grd\)\nb-vector bundle over~\(\Base\), then the liftings \(\NM_1\oplus\Triv\) and \(\NM_2\oplus\Triv\) are isotopic as well via the lifting of the isotopy along~\(\Triv\); and stable isomorphism of correspondences is an equivalence relation.  This implies that equivalence of correspondences is transitive.
\end{proof}

\begin{example}
  \label{exa:normal_Triv_Target}
  Let \((\VB,\Triv_1,\hat{f}_1)\) be a normally non-singular map and let~\(\Triv_2\) be another \(\Grd\)\nb-vector bundle over~\(\Base\) with \(\Triv_1^\Target \cong \Triv_2^\Target\).  Use this isomorphism to view~\(\hat{f}_1\) as an open embedding \(\hat{f}_2\colon \total{\VB} \opem \total{\Triv_2^\Target}\).  We claim that the normally non-singular maps \((\VB,\Triv_j,\hat{f}_j)\) for \(j=1,2\) are equivalent.  Hence only~\(\Triv^\Target\) and not~\(\Triv\) is relevant for the equivalence class of a normally non-singular map.

  To establish the claim, lift \((\VB,\Triv_1,\hat{f}_1)\) along~\(\Triv_2\) and \((\VB,\Triv_2,\hat{f}_2)\) along~\(\Triv_1\) to get isomorphic \(\Grd\)\nb-vector bundles on~\(\Base\).  This leads to the normally non-singular maps
  \[
  (\VB\oplus\Triv_2^\Source,\Triv_1\oplus\Triv_2,\hat{f}),
  \qquad
  (\VB\oplus\Triv_2^\Source,\Triv_1\oplus\Triv_2,
  \sigma\circ\hat{f}),
  \]
  where~\(\sigma\) is the coordinate flip on \(\Triv_1^\Target\oplus\Triv_2^\Target \cong (\Triv_2^\Target)^2\).  We may ensure that~\(\sigma\) is homotopic to the identity map along \(\Grd\)\nb-vector bundle automorphisms by first lifting \((\VB,\Triv_j,\hat{f}_j)\) along~\(\Triv_j\) to double both \(\Triv_1\) and~\(\Triv_2\).
\end{example}

\begin{example}
  \label{exa:smooth_normal_equivalence}
  Let \(\Source\) and~\(\Target\) be smooth manifolds.  A \emph{smooth normally non-singular map} from~\(\Source\) to~\(\Target\) is a triple \((\VB,\R^n,\hat{f})\) as above (vector bundles over the one-point space are identified with~\(\R^n\) here), where~\(\VB\) carries a smooth structure and~\(\hat{f}\) is a diffeomorphism from~\(\total{\VB}\) onto an open subset of \(\Target\times\R^n\).  \emph{Smooth isotopies} of smooth normally non-singular maps are defined similarly.  Liftings of smooth normally non-singular maps are again smooth normally non-singular maps.  Lifting and smooth isotopy generate an equivalence relation of \emph{smooth equivalence} for smooth normally non-singular maps.

  Example~\ref{exa:smooth_normal_map} shows that any smooth map \(\Source\to\Target\) is the trace of a smooth normally non-singular map.  Furthermore, two smooth normally non-singular maps are smoothly equivalent if and only if their traces are smooothly homotopic.  We omit the argument because we will prove more general statements in the equivariant case in Section~\ref{sec:smooth_normal_map}.  In particular, for a suitable class of groupoids~\(\Grd\) and smooth \(\Grd\)\nb-manifolds \(\Source\) and~\(\Target\), we will show that every smooth map \(\mapr\colon \Source \to \Target\) is normally non-singular in an essentially unique way.
\end{example}

\begin{example}
  \label{exa:smooth_to_point}
  Let~\(\Source\) be a smooth manifold and let~\(\pt\) be the one-point space.  According to the last example, all \emph{smooth} normally non-singular maps from~\(\Source\) to~\(\pt\) are equivalent because there is a unique map \(\Source\to\pt\).  The stable normal bundle of such a smooth normally non-singular map is \(-[\Tvert\Source] \in \KO^0(\Source)\).  If another smooth structure on~\(\Source\) yields a different tangent bundle in \(\KO^0(\Source)\), then the resulting normally non-singular map is not equivalent to a smooth normally non-singular map for the old smooth structure.
\end{example}

\subsection{Composition of normally non-singular maps}
\label{sec:compose_normal_maps}

Let \(\NM_j=(\VB_j,\Triv_j,\hat{f}_j)\) for \(j=1,2\) be normally non-singular maps from~\(\Source\) to~\(\Target\) and from~\(\Target\) to~\(\Third\), respectively; let \(f_1\colon \Source\to\Target\) and \(f_2\colon \Target\to\Third\) be their traces.  We are going to define a normally non-singular map \(\NM_2\circ\NM_1 = (\VB,\Triv,\hat{f})\) from~\(\Source\) to~\(\Third\) whose trace is \(f=f_2\circ f_1\).  Let
\[
\VB\defeq\VB_1\oplus f_1^*(\VB_2) \in \Vect_\Grd(\Tot),\qquad
\Triv\defeq \Triv_1\oplus\Triv_2 \in \Vect_\Grd(\Tot);
\]

The open embedding \(\hat{f}\colon \total{\VB}\opem\total{\Triv^\Third}\) is the composition of the open embedding
\[
\Id_{\total{\Triv_1}}\times_\Base\hat{f}_2 \colon
\total{\Triv_1^\Target \oplus \VB_2}
\cong \total{\Triv_1} \times_\Base \total{\VB_2}
\opem \total{\Triv_1} \times_\Base \total{\Triv_2^\Third}
\cong \total{\Triv_1} \times_\Base \total{\Triv_2}
\times_\Base\Third
\cong \total{\Triv^\Third}
\]
with an open embedding \(\total{\VB} \opem \total{\Triv_1^\Target \oplus \VB_2}\) that we get by lifting~\(\hat{f}_1\) along the non-trivial \(\Grd\)\nb-vector bundle~\(\VB_2\) over~\(\Target\).  This lifting operation is slightly more subtle than for trivial \(\Grd\)\nb-vector bundles and is only defined up to isotopy.

The open embedding \(\hat{f}_1\colon \total{\VB_1} \to \total{\Triv_1^\Target}\) is a map over~\(\Target\) when we view~\(\total{\VB_1}\) as a space over~\(\Target\) via \(\proj{\Triv_1^\Target}\circ\hat{f}_1\colon \total{\VB_1} \opem \total{\Triv_1^\Target}\epi \Target\).  This allows us to form a \(\Grd\)\nb-map \(\hat{f}_1\times_\Target \Id_{\total{\VB_2}}\), which is again an open embedding and has the right target space \(\total{\Triv_1^\Target}\times_\Target \total{\VB_2} \cong \total{\Triv_1^\Target \oplus \VB_2}\).  Its domain \(\total{\VB_1}\times_\Target\total{\VB_2}\) is the total space of the pull-back of the \(\Grd\)\nb-vector bundle~\(\VB_2\) to~\(\total{\VB_1}\) along \(\proj{\Triv_1^\Target}\circ\hat{f}_1\).  Since the zero section \(\zers{\VB_1}\colon \Source\mono \total{\VB_1}\) and bundle projection \(\proj{\VB_1}\colon \total{\VB_1}\epi\Source\) are inverse to each other up to a natural \(\Grd\)\nb-homotopy, \(\proj{\Triv_1^\Target}\circ\hat{f}_1\) is \(\Grd\)\nb-equivariantly homotopic to \(\proj{\Triv_1^\Target}\circ\hat{f}_1 \circ \zers{\VB_1}\circ\proj{\VB_1} = f_1\circ\proj{\VB_1}\).  The homotopy invariance of pull-backs of vector bundles (Proposition~\ref{pro:vb_homotopy}) provides a \(\Grd\)\nb-vector bundle isomorphism between the corresponding two pull-backs of~\(\VB_2\), which is unique up to isotopy.  Since the total space of \((f_1\circ\proj{\VB_1})^*(\VB_2)\) is~\(\total{\VB}\), we get a homeomorphism \(\total{\VB}\to \total{\VB_1}\times_\Target\total{\VB_2}\) that, when composed with the open embeddings \(\hat{f}_1\times_\Target \Id_{\total{\VB_2}}\) and \(\Id_{\total{\Triv_1}}\times_\Base\hat{f}_2\), provides the required open embedding
\[
\hat{f}\colon \total{\VB} \xopem{\cong}
\total{\VB_1}\times_\Target\total{\VB_2}
\xopem{\hat{f}_1\times_\Target \Id_{\total{\VB_2}}}
\total{\Triv_1^\Target \oplus \VB_2}
\xopem{\Id_{\total{\Triv_1}}\times_\Base\hat{f}_2}
\total{\Triv^\Third}.
\]
Our construction shows that~\(\hat{f}\) is unique up to isotopy.

\begin{theorem}
  \label{the:normal_maps_category}
  Equivalence classes of normally non-singular maps with the above composition form a category.  The trace and the degree of a normally non-singular map define functors to the homotopy category of \(\Grd\)\nb-maps and to the group~\(\Z\).
\end{theorem}

\begin{proof}
  The same recipe as above defines the composition of isotopies.  Hence products of isotopic normally non-singular maps remain isotopic.  Moreover, if we lift one of the factors along a \(\Grd\)\nb-vector bundle over~\(\Base\), the product will only change by a lifting along the same \(\Grd\)\nb-vector bundle over~\(\Base\).  Hence products of equivalent normally non-singular maps remain equivalent.  Thus the composition of equivalence classes of normally non-singular maps is well-defined.

  The stable normal bundle is additive for composition:
  \begin{equation}
    \label{eq:compose_normal_bundle}
    \nu_{\NM_2\circ\NM_1} \defeq
    [\VB] - [\Triv^\Source]
    = [\VB_1] - [\Triv_1^\Source] +
    f_1^*\bigl([\VB_2] - [\Triv_2^\Target]\bigr)
    \eqdef \nu_{\NM_1} + f_1^*(\nu_{\NM_1}).
  \end{equation}
  Hence the degree is additive as well.  It is also clear that the trace of the product is the product of the traces.

  The identity on~\(\Tot\) is the normally non-singular map \((0,0,\Id_\Tot)\) with \(\total{\VB}=\total{\Triv}=\Tot\).  It behaves like an identity by definition of the composition.

  It is routine to check that the composition of normally non-singular maps is associative.  The products \((\NM_1\circ\NM_2)\circ\NM_3\) and \(\NM_1\circ (\NM_2\circ\NM_3)\) both involve \(\Grd\)\nb-vector bundles isomorphic to \(\VB_1 \oplus f_1^*(\VB_2) \oplus f_1^*\bigl(f_2^*(\VB_3)\bigr)\) and \(\Triv_1\oplus\Triv_2\oplus\Triv_3\) -- here we use \(f_1^*\bigl( f_2^*(\VB_3)\bigr) \cong (f_2f_1)^*(\VB_3)\); the open embeddings in both products are composites of liftings of \(\hat{f}_1\), \(\hat{f}_2\), and~\(\hat{f}_3\) along the same \(\Grd\)\nb-vector bundles \(\VB_2\oplus f_3^*(\VB_3)\), \(\Triv_1\oplus\VB_3\), \(\Triv_1\oplus\Triv_2\); here we use the following observation about lifting in stages:

  Let \(\VB'\) and~\(\VB''\) be \(\Grd\)\nb-vector bundles over a \(\Grd\)\nb-space~\(\Target\) and let \((\VB,\Triv,\hat{f})\) be a normally non-singular map from~\(\Source\) to~\(\Target\).  First lift~\(\hat{f}\) along~\(\VB'\) to an open embedding
  \[
  \total{\VB\oplus f^*(\VB')}\opem \total{\Triv^\Target\oplus\VB'}
  \]
  as above, then lift the latter along~\(\VB''\) to an open embedding
  \[
  \total{\VB\oplus f^*(\VB')\oplus f^*(\VB'')}\opem \total{\Triv^\Target\oplus\VB' \oplus \VB''};
  \]
  the result is isotopic to the lifting of~\(\hat{f}\) along \(\VB'\oplus\VB''\).

  We leave it to the reader to check this observation and to fill in the remaining details of the proof of associativity.
\end{proof}

\begin{definition}
  \label{def:normal_category}
  Let~\(\Nor(\Grd)\) denote the category whose objects are \(\Grd\)\nb-spaces and whose morphisms are normally non-singular \(\Grd\)\nb-maps with the above composition.
\end{definition}

\begin{remark}
  \label{rem:composition_embedding_submersion}
  Normally non-singular embeddings and special normally non-singular submersions form subcategories, that is, products of normally non-singular embeddings are normally non-singular embeddings and products of special normally non-singular submersions are special normally non-singular submersions.
\end{remark}

\begin{remark}
  \label{rem:homeo_normal}
  An open embedding \(\varphi\colon \Source\opem\Target\) yields a normally non-singular map \(\varphi!\defeq (0,0,\varphi)\) that is both a normally non-singular embedding and a special normally non-singular submersion.  This construction is a functor, that is, \(\Id_\Tot!\) is the identity normally non-singular map on~\(\Tot\) and \(\varphi_1!\circ\varphi_2!  = (\varphi_1\circ\varphi_2)!\) for open embeddings \(\varphi_2\colon \Source\opem\Target\) and \(\varphi_1\colon \Target\opem\Third\).
\end{remark}

\begin{example}
  \label{exa:up_down_vb}
  Let~\(\VB\) be a \(\Grd\)\nb-vector bundle over~\(\Tot\).  The zero section \(\zers{\VB}\colon \Tot\mono\total{\VB}\) is the trace of a normally non-singular embedding \((\VB,0,\Id_{\VB})\) if~\(\VB\) is subtrivial.  We still denote this normally non-singular embedding by~\(\zers{\VB}\).

  The projection \(\proj{\VB}\colon\total{\VB}\epi\Tot\) is the trace of a special normally non-singular submersion if~\(\VB\) is trivial.  If~\(\VB\) is only subtrivial, then there is a canonical normally non-singular map with trace~\(\proj{\VB}\).  Let \(\VB\oplus\VB^\bot\cong\Triv^\Tot\) for a \(\Grd\)\nb-vector bundle~\(\Triv\) over~\(\Base\).  Then the relevant normally non-singular map is \((\proj{\VB}^*(\VB^\bot),\iota,\Triv)\), where \(\proj{\VB}^*(\VB^\bot)\) denotes the pull-back of~\(\VB^\bot\) to~\(\VB\), which has total space \(\total{\VB\oplus\VB^\bot}\), and \(\iota\colon \total{\VB\oplus\VB^\bot}\congto\total{\Triv^\Tot}\) is the isomorphism.  We also denote this normally non-singular map by~\(\proj{\VB}\).

  The equivalence classes of these normally non-singular maps \(\zers{\VB}\colon \Tot\to\total{\VB}\) and \(\proj{\VB}\colon \total{\VB} \to \Tot\) are inverse to each other: both composites are liftings of the identity map along~\(\Triv\).  The details are a good exercise to get familiar with composing normally non-singular maps.  The equivalence class of~\(\proj{\VB}\) cannot depend on \(\VB^\bot\) and~\(\Triv\) because inverses are unique.  Checking this directly is another good exercise.
\end{example}

\subsection{Exterior products and functoriality}
\label{sec:exterior_functor}

Now we study \emph{exterior products} of normally non-singular maps and show that \(\Nor(\Grd\ltimes\Tot)\) is a contravariant homotopy functor in~\(\Tot\).

The exterior product of two \(\Grd\)\nb-spaces is their fibre product over the object space~\(\Base\), equipped with the induced action of~\(\Grd\).  Let \(\NM_j=(\VB_j,\Triv_j,\hat{f}_j)\) for \(j=1,2\) be a normally non-singular map from~\(\Source_j\) to~\(\Target_j\).  Then we get a normally non-singular \(\Grd\)\nb-map
\[
\NM_1\times_\Base \NM_2 \defeq
(\pi_1^*\VB_1\oplus \pi_2^*\VB_2,
\Triv_1\oplus\Triv_2,
\hat{f}_1\times_\Base\hat{f}_2)
\]
from \(\Source\defeq \Source_1\times_\Base\Source_2\) to \(\Target\defeq \Target_1\times_\Base\Target_2\), where \(\pi_j\colon \Source\to\Source_j\) for \(j=1,2\) are the canonical projections.  The total spaces of \(\pi_1^*\VB_1\oplus \pi_2^*\VB_2\) and \((\Triv_1\oplus\Triv_2)^\Target\) are \(\total{\VB_1}\times_\Base\total{\VB_2}\) and \(\total{\Triv_1^{\Target_1}}\times_\Base \total{\Triv_2^{\Target_2}}\), so that \(\hat{f}_1\times_\Base\hat{f}_2\) has the right domain and target.

Taking the trace commutes with exterior products, that is, if \(f_1\) and~\(f_2\) are the traces of \(\NM_1\) and~\(\NM_2\), then \(\NM_1\times_\Base\NM_2\) has trace \(f_1\times_\Base f_2\).  Taking the stable normal bundle commutes with exterior products as well: if \(\Normal_1\) and~\(\Normal_2\) are the stable normal bundles of \(\NM_1\) and~\(\NM_2\), then the stable normal bundle of \(\NM_1\times_\Base\NM_2\) is \(\pi_1^*(\Normal_1)\oplus \pi_2^*(\Normal_2)\).  The degree of \(\NM_1\times_\Base\NM_2\) is the sum of the degrees of \(\NM_1\) and~\(\NM_2\).

\begin{remark}
  \label{rem:exterior_embedding_submersion}
  By definition, exterior products of normally non-singular embeddings remain normally non-singular embeddings, and exterior products of special normally non-singular submersions remain special normally non-singular submersions.
\end{remark}

\begin{proposition}
  \label{pro:normal_maps_smc}
  The exterior product~\(\times_\Base\) is an associative, symmetric and monoidal bifunctor \(\Nor(\Grd)\times\Nor(\Grd)\to\Nor(\Grd)\) with unit object~\(\Base\), that is, there are natural isomorphisms
  \begin{equation}
    \label{eq:exterior_product_properties}
    \begin{aligned}
      \Tot_1 \times_\Base \Tot_2 &\cong  \Tot_2 \times_\Base
      \Tot_2,\\
      (\Tot_1 \times_\Base \Tot_2) \times_\Base \Tot_3 &\cong
      \Tot_1 \times_\Base (\Tot_2 \times_\Base \Tot_3),\\
      \Base \times_\Base \Tot &\cong \Tot \cong
      \Tot \times_\Base \Base,
    \end{aligned}
  \end{equation}
  which satisfy various coherence conditions, so that \(\Nor(\Grd)\) becomes a symmetric monoidal category \textup(see \textup{\cite{Saavedra:Tannakiennes}}\textup).
\end{proposition}

\begin{proof}
  The exterior product preserves isotopy because we can form exterior products of isotopies.  Since it commutes with liftings as well, it descends to a well-defined operation on equivalence classes.  Functoriality of exterior products with respect to composition of normally non-singular maps is routine to check.  It is obvious that there are \(\Grd\)\nb-equivariant homeomorphisms as in~\eqref{eq:exterior_product_properties} that satisfy the coherence conditions for a symmetric monoidal category and that are natural with respect to \(\Grd\)\nb-maps.  But \(\Grd\)\nb-equivariant homeomorphisms are normally non-singular, and the homeomorphisms in~\eqref{eq:exterior_product_properties} are natural with respect to \emph{normally non-singular} maps as well.
\end{proof}

\begin{remark}
  \label{rem:product_not_categorical}
  The exterior product is \emph{not} a product operation in \(\Nor(\Grd)\) in the sense of category theory.  The coordinate projections from \(\Target_1\times_\Base \Target_2\) to \(\Target_1\) and~\(\Target_2\) need not be traces of normally non-singular maps.
\end{remark}

\begin{lemma}
  \label{lem:normal_map_coproduct}
  The disjoint union operation is a coproduct in the category \(\Nor(\Grd)\):
  \[
  \Nor_\Grd(\Source_1\sqcup\Source_2,\Target) \cong
  \Nor_\Grd(\Source_1,\Target) \times
  \Nor_\Grd(\Source_2,\Target)
  \]
  for all \(\Grd\)\nb-spaces \(\Source_1\), \(\Source_2\), and~\(\Target\).  The empty \(\Grd\)\nb-space is an initial object.
\end{lemma}

\begin{proof}
  Since the embeddings \(\Source_j\opem \Source_1\sqcup\Source_2\) for \(j=1,2\) are open, they are normally non-singular \(\Grd\)\nb-maps and thus induce a natural map
  \begin{equation}
    \label{eq:normal_map_union}
    \Nor_\Grd(\Source_1\sqcup\Source_2,\Target) \to
    \Nor_\Grd(\Source_1,\Target) \times
    \Nor_\Grd(\Source_2,\Target).
  \end{equation}

  Conversely, let \(\NM_j=(\VB_j,\Triv_j,\hat{f}_j)\) be normally non-singular maps from~\(\Source_j\) to~\(\Target\) for \(j=1,2\).  We may lift~\(\NM_1\) along~\(\Triv_2\oplus\R\) and~\(\NM_2\) along~\(\Triv_1\oplus\R\), so that both now involve the \(\Grd\)\nb-vector bundle \(\Triv_1\oplus\Triv_2\oplus\R\) over~\(\Base\).  By an isotopy, we can arrange that the \(\R\)\nb-components of the liftings of \(\NM_1\) and~\(\NM_2\) have values in \((0,1)\) and \((1,2)\), respectively, so that their ranges are disjoint.  After these modifications, \(\hat{f}_1\sqcup \hat{f}_2\) becomes an open embedding on \(\total{\VB_1}\sqcup\total{\VB_2}\), so that we get a normally non-singular map \(\NM_1\sqcup\NM_2\) from \(\Source_1\sqcup\Source_2\) to~\(\Target\).  Hence the map in~\eqref{eq:normal_map_union} is surjective.  The same construction may be applied to isotopies, so that \(\NM_1\sqcup\NM_2\) is isotopic to \(\NM_1'\sqcup \NM_2'\) if \(\NM_j\) is isotopic to~\(\NM_j'\) for \(j=1,2\).  When we lift~\(\NM_1\) along~\(\Triv'_1\) and~\(\NM_2\) along~\(\Triv'_2\), then we lift \(\NM_1\sqcup\NM_2\) along \(\Triv'_1\oplus\Triv'_2\).  Hence \(\NM_1\sqcup\NM_2\) is equivalent to \(\NM_1'\sqcup \NM_2'\) if \(\NM_j\) is equivalent to~\(\NM_j'\) for \(j=1,2\).  Thus the map in~\eqref{eq:normal_map_union} is injective as well.

  More or less by convention, there is, up to equivalence, a unique normally non-singular \(\Grd\)\nb-map from the empty \(\Grd\)\nb-space to any other \(\Grd\)\nb-space.
\end{proof}

Let \(\Source\) and~\(\Target\) be two \(\Grd\)\nb-spaces and let \(h\colon \Source\to\Target\) be a \(\Grd\)\nb-map.  If \(\Third\) is another \(\Grd\ltimes\Target\)\nb-space, then \(h^*(\Third) \defeq \Source\times_\Target\Third\) is a \(\Grd\ltimes\Source\)\nb-space, equipped with a canonical map \(\hat{h}\colon h^*(\Third)\to\Third\).  We pull a normally non-singular \(\Grd\ltimes\Target\)-map \((\VB,\Triv,\hat{f})\) from~\(\Third_1\) to~\(\Third_2\) back to a normally non-singular map \(\bigl(\hat{h}^*(\VB), h^*(\Triv), \Id_\Source\times_\Target\hat{f}\bigr)\) from \(h^*(\Third_1)\) to \(h^*(\Third_2)\); here we use that the total spaces of \(\hat{h}^*(\VB)\) and \(h^*(\Triv)^{h^*(\Third_2)}\) are \(\Source\times_\Target \Third_1 \times_{\Third_1} \total{\VB} \cong \Source\times_\Target\total{\VB}\) and \(\Source\times_\Target \Third_2 \times_\Target \total{\Triv} \cong \Source\times_\Target\total{\Triv^{\Third_2}}\), respectively.

This construction yields a functor
\[
h^*\colon \Nor(\Grd\ltimes\Target) \to
\Nor(\Grd\ltimes\Source),
\qquad
(\VB,\Triv,\hat{f})\mapsto
\bigl(\hat{h}^*(\VB), h^*(\Triv),
\Id_\Source\times_\Target\hat{f}\bigr).
\]
It is symmetric monoidal (see~\cite{Saavedra:Tannakiennes}) because the canonical homeomorphisms
\[
h^*(\Third_1) \times_\Source h^*(\Third_2) \cong
h^*(\Third_1\times_\Target \Third_2)
\]
for two \(\Grd\ltimes\Target\)-spaces \(\Third_1\) and~\(\Third_2\) are natural with respect to normally non-singular maps and compatible with the unit, commutativity, and associativity isomorphisms in the symmetric monoidal categories \(\Nor(\Grd\ltimes\Target)\) and \(\Nor(\Grd\ltimes\Source)\).

Thus \(\Tot\mapsto \Nor(\Grd\ltimes\Tot)\) is a functor from the category of \(\Grd\)\nb-spaces to the category of symmetric monoidal categories.

\begin{lemma}
  \label{lem:Nor_homotopy_functor}
  The functor \(\Tot\mapsto \Nor(\Grd\ltimes\Tot)\) is a homotopy functor, that is, \(\Grd\)\nb-homotopic maps \(h_0,h_1\colon \Source\rightrightarrows\Target\) induce equivalent functors \(\Nor(\Target) \to \Nor(\Source)\).
\end{lemma}

\begin{proof}
  Let \(h\colon \Source\times[0,1]\to\Target\times[0,1]\) be a homotopy between \(h_0\) and~\(h_1\) and let~\(\NM\) be a normally non-singular \(\Grd\ltimes\Target\)-map.  Then the normally non-singular map \(h^*(\NM)\) is an isotopy between the normally non-singular maps \(h_0^*(\NM)\) and \(h_1^*(\NM)\).
\end{proof}

We may also view \(\Nor(\Grd)\) as a functor of~\(\Grd\), both with respect to strict groupoid homomorphisms (continuous functors) and Hilsum--Skandalis morphisms.  Since we will not need this here, we omit the proof.

Finally, we consider certain forgetful functors.

Let \(h\colon \Source\to\Target\) be a \(\Grd\)\nb-map.  Recall that a \(\Grd\ltimes\Source\)-space is nothing but a \(\Grd\)\nb-space with a \(\Grd\)\nb-map to~\(\Source\).  Composing the latter with~\(h\), we may view a \(\Grd\ltimes\Source\)-space as a \(\Grd\ltimes\Target\)-space.  In particular, for \(\Target=\Base\) this views a \(\Grd\ltimes\Source\)-space as a \(\Grd\)\nb-space.  For vector bundles, it makes no difference whether we require \(\Grd\ltimes\Source\)-, \(\Grd\ltimes\Target\)-, or just \(\Grd\)\nb-equivariance.  Hence it would appear that we get a forgetful functor from \(\Nor(\Grd\ltimes\Source)\) to \(\Nor(\Grd\ltimes\Target)\).  But there is one technical problem: a (sub)trivial \(\Grd\ltimes\Source\)-vector bundle need not be (sub)trivial as a \(\Grd\ltimes\Target\)-vector bundle.

\begin{example}
  \label{exa:forget_loses_subtrivial}
  Let \(\Source\) be a \(\Grd\)\nb-space.  Any \(\Grd\)\nb-vector bundle over~\(\Source\) is trivial as a \(\Grd\ltimes\Source\)\nb-vector bundle, but not necessarily as a \(\Grd\)\nb-vector bundle.
\end{example}

\begin{proposition}
  \label{pro:forgetful_normal}
  Let \(h\colon \Source\to\Target\) be a \(\Grd\)\nb-map between two \(\Grd\)\nb-spaces.  If any \(\Grd\)\nb-vector bundle over~\(\Source\) is subtrivial as a \(\Grd\ltimes\Target\)-vector bundle, then there is a forgetful functor \(h_*\colon \Nor(\Grd\ltimes\Source)\to\Nor(\Grd\ltimes\Target)\).
\end{proposition}

\begin{proof}
  Let \(\NM=(\VB,\Triv,\hat{f})\) be a normally non-singular \(\Grd\ltimes\Source\)\nb-map.  Then~\(\Triv\) is a \(\Grd\ltimes\Source\)-vector bundle over~\(\Source\).  By assumption, it is subtrivial as a \(\Grd\ltimes\Target\)\nb-vector bundle, that is, there are \(\Grd\)\nb-vector bundles \(\Triv'\) and~\(\Triv''\) over \(\Source\) and~\(\Target\), respectively, with \(\Triv\oplus\Triv' \cong h^*(\Triv'')\).  We lift~\(\NM\) along~\(\Triv'\) and put
  \[
  h_*(\NM) \defeq (\VB\oplus(\Triv')^\Source,\Triv'',
  \hat{f}\times_\Target \Id_{\total{\Triv'}}).
  \]
  We leave it to the reader to observe that this construction is independent of the auxiliary choices of \(\Triv'\) and~\(\Triv''\), descends to equivalence classes (compare Example~\ref{exa:normal_Triv_Target}), and is functorial.
\end{proof}

\begin{remark}
  \label{rem:forgetful_without_subtrivial}
  We can avoid the problem with subtriviality of vector bundles
  if we use another monoid \(\Vect_\Grd(\Tot)\) of
  \(\Grd\)\nb-vector bundles: instead of subtrivial
  \(\Grd\ltimes\Source\)\nb-vector bundles, use only those
  bundles that are direct summands in
  \(\Grd\ltimes\Source\)\nb-vector bundles pulled back
  from~\(\Target\).
\end{remark}

\subsection{Smooth normally non-singular maps}
\label{sec:smooth_normal_map}

Now we extend the discussion of smooth normally non-singular maps in Example~\ref{exa:smooth_normal_equivalence} to the equivariant case.  In general, neither existence nor uniqueness up to equivalence of normal factorisations for smooth maps is clear: we need additional technical assumptions.  Let \(\Source\) and~\(\Target\) be smooth \(\Grd\)\nb-manifolds.  We assume that there is a smooth normally non-singular map \(\Source\to\Base\) (see Theorem~\ref{the:normal_factorisation} for sufficient conditions) and that the tangent bundle over~\(\Target\) is subtrivial or that, for some reason, all \(\Grd\)\nb-vector bundles over~\(\Source\) are subtrivial.

\begin{definition}
  \label{def:smooth_isotopy_normal_map}
  A \emph{smooth isotopy} of normally non-singular maps from~\(\Source\) to~\(\Target\) is a normally non-singular \(\Grd\times[0,1]\)\nb-map \((\VB,\Triv,\hat{f})\) from \(\Source\times[0,1]\) to~\(\Target\times[0,1]\) such that~\(\VB\) has a fibrewise smooth structure and~\(\hat{f}\) is a fibrewise diffeomorphism onto its range as a map over~\(\Base\).
\end{definition}

We may reparametrise a smooth isotopy so that its higher derivatives at \(0\) and~\(1\) vanish.  This allows to glue together smooth isotopies, showing that smooth isotopy is an equivalence relation.  Combining smooth isotopy and lifting, we get the relation of \emph{smooth equivalence}.  This is an equivalence relation as well.

\begin{theorem}
  \label{the:normal_map_unique}
  Let \(\Source\) and~\(\Target\) be smooth \(\Grd\)\nb-manifolds.  Assume that there is a smooth normally non-singular map from~\(\Source\) to the object space~\(\Base\) of~\(\Grd\) and that \(\Tvert\Target\) is subtrivial or that all \(\Grd\)\nb-vector bundles over~\(\Source\) are subtrivial.  Then any smooth \(\Grd\)\nb-map from~\(\Source\) to~\(\Target\) is the trace of a smooth normally non-singular \(\Grd\)\nb-map, and two smooth normally non-singular \(\Grd\)\nb-maps from~\(\Source\) to~\(\Target\) are smoothly equivalent if and only if their traces are smoothly \(\Grd\)\nb-homotopic.
\end{theorem}

Example~\ref{exa:smooth_manifold_diagonal} shows that Theorem~\ref{the:normal_map_unique} fails for non-smooth normally non-singular maps: for a smooth manifold~\(\Source\), there may be several non-equivalent normally non-singular maps \(\Source\to\Source\times\Source\) whose trace is the diagonal embedding.

The technical conditions in the theorem are necessary for a good theory.  If there is no smooth normally non-singular \(\Grd\)\nb-map \(\Source\to\Base\), then there is no smooth normally non-singular map whose trace is the anchor map \(\Source\to\Base\), which is a smooth \(\Grd\)\nb-map.  If the tangent bundle~\(\Tvert\Target\) is not subtrivial, then the diagonal embedding \(\Target\to\Target\times_\Base\Target\) is not the trace of a smooth normally non-singular map because its stable normal bundle would have to be~\(\Tvert\Tot\).  If~\(\VB\) is a \(\Grd\)\nb-vector bundle over~\(\Source\) that is not subtrivial, then the zero section \(\Source\to\total{\VB}\) is not a smooth normally non-singular map because its stable normal bundle would have to be~\(\VB\).

\begin{proof}[Proof of Theorem~\textup{\ref{the:normal_map_unique}}]
  Lifting does not alter the trace of a normally non-singular map, and a smooth isotopy of normally non-singular maps provides a smooth homotopy between their traces.  The main point is that, conversely, any smooth map from~\(\Source\) to~\(\Target\) is the trace of a smooth normally non-singular map and that smooth normally non-singular maps with smoothly homotopic traces are smoothly equivalent.

  Let \((\VB,\Triv,\hat{g})\) be a smooth normally non-singular map from~\(\Source\) to~\(\Base\).  Then \(g\defeq \hat{g}\circ\zers{\VB}\colon \Source\to\total{\Triv}\) is a smooth embedding.  Let \(f\colon \Source\to\Target\) be a smooth \(\Grd\)\nb-map.  Then we get another smooth embedding \((f,g)\colon \Source\to\Target\times_\Base\total{\Triv} = \total{\Triv^\Target}\).  By the Tubular Neighbourhood Theorem~\ref{the:tubular_for_embedding}, it extends to a smooth open embedding on the fibrewise normal bundle of \((f,g)\).  This normal bundle is contained in the pull-back of~\(\Tvert\Target\), so that our assumptions ensure that it is subtrivial.  As a result, we get a smooth normally non-singular map from~\(\Source\) to~\(\Target\) with trace~\(f\).  Thus any smooth map \(f\colon \Source\to\Target\) is the trace of a smooth normally non-singular map.

  Similarly, any fibrewise smooth map \(F = (F_t)_{t\in[0,1]}\colon \Source\times[0,1] \to \Target\times[0,1]\) is the trace of a smooth normally non-singular map \(\Source\times[0,1] \to \Target\times[0,1]\).  Thus a smooth homotopy between two smooth normally non-singular maps lifts to a smooth isotopy of normally non-singular maps.  It remains to check that two smooth normally non-singular maps with the \emph{same} trace are smoothly equivalent.

  Let \(\NM_j=(\VB_j,\Triv_j,\hat{f}_j)\) for \(j=1,2\) be smooth normally non-singular maps with the same trace.  Lifting~\(\NM_1\) along~\(\Triv_2\) and~\(\NM_2\) along~\(\Triv_1\), we can arrange that both involve the same \(\Grd\)\nb-vector bundle \(\Triv\defeq\Triv_1\oplus\Triv_2\) over~\(\Base\).  Let \(\tilde{f}_j\defeq \hat{f}_j\circ \zers{\VB_j}\colon \Source\mono\total{\VB_j}\opem \total{\Triv^\Target}\) for \(j=1,2\).  These are smooth embeddings.  They are \(\Grd\)\nb-equivariantly homotopic via the \(\Grd\ltimes [0,1]\)\nb-equivariant smooth embedding
  \[
  \tilde{f}\colon
  \Source\times[0,1]\to\total{\Triv^\Target}\times[0,1], \qquad
  (\source,t)\mapsto \bigl((1-t)\tilde{f}_1(\source),t\tilde{f}_2(\source),t\bigr).
  \]
  Its normal bundle restricts to \(\VB_1\) and~\(\VB_2\) at \(0\) and~\(1\); hence \(\VB_1\cong \VB_2\) by Proposition~\ref{pro:vb_homotopy}.  The Tubular Neighbourhood Theorem~\ref{the:tubular_for_embedding} applied to~\(\tilde{f}\) yields a smooth isotopy between \(\NM_j'=(\VB_j,\hat{f}'_j,\Triv_j)\) for \(j=1,2\) with \(\hat{f}'_1\circ \zers{\VB_1} = \tilde{f}_1\) and \(\hat{f}'_2\circ \zers{\VB_2} = \tilde{f}_2\).

  It remains to show that~\(\NM_j\) is smoothly isotopic to~\(\NM_j'\) for \(j=1,2\).  Equivalently, we must show that \(\NM_1\) and~\(\NM_2\) are smoothly isotopic if \(\VB_1=\VB_2\), \(\Triv_1=\Triv_2\), and \(\tilde{f}_1=\tilde{f}_2\), where \(\tilde{f}_j \defeq \hat{f}_j\circ\zers{\VB_j}\).  This means that the Tubular Neighbourhood of an embedded submanifold is unique up to isotopy.  The proof in \cite{Hirsch:Diff_Top}*{p.~113f} carries over to the equivariant case almost literally.  We may equip~\(\VB\) with a fibrewise smooth inner product by Proposition~\ref{pro:vb_inner_product}, and we can find a \(\Grd\)\nb-invariant fibrewise smooth function \(\varrho\colon \Source\to[0,1]\) such that \(\hat{f}_2(\total{\VB})\) contains \(\hat{f}_1(\vb)\) for all \(\vb\in\total{\VB}\) with \(\norm{\vb}\le\varrho\bigl(\proj{\VB}(\vb)\bigr)\), where \(\proj{\VB}\colon \total{\VB}\epi\Source\) is the bundle projection.  The map~\(\hat{f}_1\) is isotopic to the map
  \[
  \vb\mapsto \hat{f}_1\Bigl(\vb\cdot \varrho\bigl(p(\vb)\bigr) \bigm/ \sqrt{1+\norm{\vb}^2}\Bigr),
  \]
  whose range is contained in \(\hat{f}_2(\total{\VB})\) by construction.  Hence we may assume without loss of generality that \(\hat{f}_1(\total{\VB})\subseteq \hat{f}_2(\total{\VB})\), so that \(\hat{f}_1 = \hat{f}_2\circ \Psi\) for a smooth map \(\Psi\colon \total{\VB}\opem\total{\VB}\) that is a diffeomorphism onto an open subset of~\(\total{\VB}\) and restricts to the identity map between the zero sections.  The derivative of~\(\Psi\) on the zero section restricts to a vector bundle automorphism~\(\Psi_0\) of \(\VB\subseteq \Tvert\VB\).  An Alexander homotopy as in \cite{Hirsch:Diff_Top} shows that~\(\Psi\) is isotopic to~\(\Psi_0\).  Thus \((\VB,\hat{f}_1,\Triv)\) is isotopic to \((\VB,\hat{f}_2\circ \Psi_0,\Triv)\), and the latter is isomorphic to \((\VB,\hat{f}_2,\Triv)\) via~\(\Psi_0\).  This finishes the proof.
\end{proof}

\section{Oriented normally non-singular maps and their wrong-way maps}
\label{sec:orientations}

A normally non-singular map only induces a map on \(\K\)\nb-theory or \(\KO\)\nb-theory if it comes with additional orientation information, which depends on the choice of a cohomology theory.  In this section, we first fix our notation regarding equivariant cohomology theories; then we define oriented normally non-singular maps and let them act on the appropriate cohomology theory by wrong-way maps.  Our main examples are equivariant \(\K\)\nb-theory and equivariant \(\KO\)\nb-theory.  If~\(\Grd\) is a group, then Bredon cohomology provides equivariant versions of cohomology as well.

Equivariant representable \(\K\)\nb-theory for proper actions of locally compact groupoids with a Haar system on locally compact spaces is studied in~\cite{Emerson-Meyer:Equivariant_K}.  The treatment of equivariant \(\K\)\nb-theory in~\cite{Emerson-Meyer:Equivariant_K} carries over literally to equivariant \(\KO\)\nb-theory.  We do not want to discuss here how to extend this theory to more general \(\Grd\)\nb-spaces.  We do not need our theories to be defined for all spaces~-- all locally compact spaces is enough.  When we specialise to equivariant \(\K\)\nb-theory in the following, we assume all groupoids and spaces to be locally compact to ensure that it is defined.

\subsection{Equivariant cohomology theories}
\label{sec:equiv_cohomology}

Let~\(\Grd\) be a numerably proper groupoid.  Let \(\Coh^* = (\Coh^n)_{n\in\Z}\) be a sequence of contravariant functors from pairs of \(\Grd\)\nb-spaces to the category of Abelian groups (or some other Abelian category).  We shall assume that they have the following properties:
\begin{enumerate}[label=(\roman{*})]
\item \label{Coh_1} \(\Coh^n\) is invariant under \(\Grd\)\nb-equivariant homotopies for all \(n\in\Z\).

\item \label{Coh_2} For each pair of \(\Grd\)\nb-spaces \((\Tot,\Other)\), there is a natural long exact sequence
  \[
  \dotsb
  \to \Coh^{n+1}(\Tot)
  \to \Coh^{n+1}(\Other)
  \to \Coh^n(\Tot,\Other)
  \to \Coh^n(\Tot)
  \to \Coh^n(\Other)
  \to \dotsb;
  \]
  this implies more general long exact sequences
  \[
  \dotsb
  \to \Coh^{n+1}(\Tot,\Third)
  \to \Coh^{n+1}(\Other,\Third)
  \to \Coh^n(\Tot,\Other)
  \to \Coh^n(\Tot,\Third)
  \to \Coh^n(\Other,\Third)
  \to \dotsb
  \]
  for nested closed subsets \(\Third\subseteq\Other\subseteq\Tot\).

\item \label{Coh_3} Let \(\varphi\colon (\Source,A)\to (\Target,B)\) be a map of pairs of \(\Grd\)\nb-spaces with \(\Target = \Source\cup_A B\).  Then~\(\varphi\) induces isomorphisms \(\Coh^*(\Target,B) \congto \Coh^*(\Source,A)\).

\item \label{Coh_4} Let \(\Other\subseteq\Tot\) be a closed \(\Grd\)\nb-invariant subset.  Let~\(C_\Other\) be the directed set of closed \(\Grd\)\nb-invariant neighbourhoods of~\(\Other\).  Then
  \[
  \Coh^*(\Tot,\Other) \cong \varinjlim_{\Third\in C_\Other}
  \Coh^*(\Tot,\Third);
  \]
  that is, any class in \(\Coh^*(\Tot,\Other)\) lifts to \(\Coh^*(\Tot,\Third)\) for some closed \(\Grd\)\nb-invariant neighbourhood~\(\Third\) of~\(\Other\), and if two classes in \(\Coh^*(\Tot,\Third)\) become equal in \(\Coh^*(\Tot,\Other)\), then they already become equal in \(\Coh^*(\Tot,\Third')\) for some closed \(\Grd\)\nb-invariant neighbourhood~\(\Third'\) of~\(\Other\) contained in~\(\Third\).

\item \label{Coh_5} \(\Coh^*\) is multiplicative, that is, equipped with natural associative and graded commutative exterior product operations
  \[
  \otimes_\Base\colon \Coh^i(\Tot_1,\Other_1) \otimes
  \Coh^j(\Tot_2,\Other_2) \to
  \Coh^{i+j}(\Tot_1\times_\Base\Tot_2,
  \Other_1\times_\Base\Tot_2 \cup \Tot_1\times_\Base\Other_2)
  \]
  for all \(i,j\in\Z\), which are compatible with the boundary maps in the long exact sequences for pairs.  (The exterior product operation is part of the data of~\(\Coh\).)

\item \label{Coh_6} There is \(1\in\Coh^0(\Base,\emptyset)\) such that \(1\otimes_\Base \Kclass=\Kclass= \Kclass\otimes_\Base 1\) for all pairs of \(\Grd\)\nb-spaces \((\Tot,\Other)\) and all \(\Kclass\in\Coh^i(\Tot,\Other)\).
\end{enumerate}

At least if we restrict attention to second countable, locally compact spaces and groupoids, then equivariant \(\K\)- and \(\KO\)\nb-theory have these properties, as shown in~\cite{Emerson-Meyer:Equivariant_K}.  The excision statement in~\cite{Emerson-Meyer:Equivariant_K} is weaker than~\ref{Coh_3}, but the more general statement follows for the same reasons.  Property~\ref{Coh_4} is not stated explicitly in~\cite{Emerson-Meyer:Equivariant_K}, but it follows immediately from the description by maps to Fredholm operators.  It is equivalent to Lemma~\ref{lem:excision_Coh_supported} below.

Let \(\Coh^*(\Tot,\Other) \defeq \bigoplus_{j\in\Z} \Coh^j(\Tot,\Other)\).  The composite map
\[
\Coh^*(\Tot,\Other) \otimes \Coh^*(\Tot,\Other)
\xrightarrow{\otimes_\Base}
\Coh^*(\Tot\times_\Base\Tot, \Other\times_\Base\Tot\cup
\Tot\times_\Base\Other)
\xrightarrow{\Delta^*}
\Coh^*(\Tot,\Other),
\]
where~\(\Delta\) is the diagonal map \(\Tot\to \Tot\times_\Base\Tot\), turns \(\Coh^*(\Tot,\Other)\) into a graded ring.  For \(\Other=\emptyset\), this ring is unital with unit element \(\anchor^*(1)\), where \(\anchor\colon \Tot\to\Base\) is the anchor map.

In order to define oriented vector bundles or oriented correspondences, we need a variant of~\(\Coh^*\) with built-in support conditions.

Let \(f\colon \Source\to\Target\) be a \(\Grd\)\nb-map.  A subset \(A\subseteq\Source\) is called \emph{\(\Target\)\nb-compact} if \(f|_A\colon A\to\Target\) is proper (see Definition~\ref{def:proper_map}), and \emph{relatively \(\Target\)\nb-compact} if its closure is \(\Target\)\nb-compact.  We let \(\Coh^j_\Target(\Source)\) be the inductive limit of the relative groups \(\Coh^j(\Source,\Source\setminus A) \cong \Coh^j(\cl{A},\bd A)\), where~\(A\) runs through the directed set of open, relatively \(\Target\)\nb-compact subsets of~\(\Source\).

A \(\Grd\)-map \(\Source\to \Source'\) between spaces over~\(\Target\) (not necessarily respecting the projections to~\(\Target\)) induces a map \(f^*\colon \Coh^*_\Target(\Source') \to \Coh^*_\Target(\Source)\) if \(f^{-1}(A) \subset \Source\) is \(\Target\)\nb-compact whenever \(A \subset \Source'\) is \(\Target\)\nb-compact.  If~\(f\) \emph{is} compatible with the projections, this is equivalent to~\(f\) being a proper map.

If~\(\Coh\) is representable \(\K\)\nb-theory and all spaces and groupoids involved are locally compact, then \(\Coh^*_\Target(\Source)\) as defined above agrees with the equivariant \(\K\)\nb-theory of~\(\Source\) with \(\Target\)\nb-compact support by \cite{Emerson-Meyer:Equivariant_K}*{Theorem 4.19}.

\begin{lemma}
  \label{lem:excision_Coh_supported}
  Let \(A\subseteq\Source\) be a closed \(\Grd\)\nb-invariant subspace.  Then there is a natural long exact sequence
  \[
  \dotsb
  \to \Coh^{n+1}_\Target(\Source)
  \to \Coh^{n+1}_\Target(A)
  \to \Coh^n_\Target(\Source\setminus A)
  \to \Coh^n_\Target(\Source)
  \to \Coh^n_\Target(A)
  \to \dotsb.
  \]
\end{lemma}

Hence there is no need for a relative version of~\(\Coh^*_\Target\), and \(\Coh^*_\Target(\blank)\) is covariantly functorial for open embeddings.

\begin{proof}
  Any relatively \(\Target\)\nb-compact open subset of \(\Source\), \(A\) or \(\Source\setminus A\) is contained in some \(\Target\)\nb-compact closed subset \(B\subseteq\Source\).  By excision, all groups we are dealing with are inductive limits of corresponding subgroups where we replace~\(\Source\) by~\(B\).  Hence we may assume without loss of generality that~\(\Source\) itself is \(\Target\)\nb-compact, that is, the map \(\Source\to\Target\) is proper.  Then its restriction to the closed subset~\(A\) is proper as well.  Thus \(\Coh^*_\Target(\Source)\cong \Coh^*(\Source)\) and \(\Coh^*_\Target(A)\cong \Coh^*(A)\), and a closed subset of \(\Source\setminus A\) is \(\Target\)\nb-compact if and only if it remains closed in~\(\Source\).  It remains to identify \(\Coh^*_\Target(\Source\setminus A)\) with \(\Coh^*(\Source,A)\).

  By definition, \(\Coh^*_\Target(\Source\setminus A)\) is the inductive limit of \(\Coh^*(\Source\setminus A,(\Source\setminus A)\setminus \Third)\), where~\(\Third\) runs through the directed set of open, \(\Grd\)\nb-invariant, relatively \(\Target\)\nb-compact subsets of~\(\Source\).  Being relatively \(\Target\)\nb-compact simply means \(\cl{\Third}\cap A=\emptyset\).  Thus~\(\Third\) is the complement of a closed \(\Grd\)\nb-invariant neighbourhood of~\(A\).  Since~\(\Third\) has the same closures in \(\Source\) and~\(\Source\setminus A\), we may use excision twice to rewrite
  \[
  \Coh^*(\Source\setminus A,(\Source\setminus A)\setminus \Third)
  \cong \Coh^*(\cl{\Third},\bd \Third)
  \cong \Coh^*(\Source,\Source\setminus \Third).
  \]
  Finally, we take the inductive limit where \(\Source\setminus \Third\) runs through the directed set of closed \(\Grd\)\nb-invariant neighbourhoods of~\(A\).  This agrees with \(\Coh^*(\Source,A)\) by Property~\ref{Coh_4} on page~\pageref{Coh_4}.
\end{proof}

\begin{definition}
  \label{def:orientation}
  Let~\(\VB\) be a \(d\)\nb-dimensional \(\Grd\)\nb-vector bundle over a \(\Grd\)\nb-space~\(\Tot\).  An \emph{\(\Coh\)\nb-orientation} for~\(\VB\) is a class \(\tau\in\Coh^d_\Tot(\total{\VB})\) such that for each \(\Grd\)\nb-map \(f\colon \Other\to\Tot\), multiplication with \(f^*(\tau)\in\Coh^*_\Other(\total{f^*\VB})\) induces an isomorphism \(\Coh_\Tot^*(\Other)\to\Coh_\Tot^*(\total{f^*\VB})\) (shifting degrees by~\(d\), of course).
\end{definition}

If~\(\Coh\) is representable \(\K\)\nb-theory, then a \(\K\)\nb-orientation in the usual sense (specified by a complex spinor bundle or by a principal \(\Spinc\)-bundle) is one in the sense of Definition~\ref{def:orientation}, and the isomorphism \(\Coh^*(\Tot)\to\Coh_\Tot^*(\VB)\) is a variant of the familiar Thom isomorphism for equivariant \(\K\)\nb-theory.

In the non-equivariant case, an orientation is defined by requiring~\(\tau\) to induce fibrewise cohomology isomorphisms \(\Coh^*\bigl(\{\tot\}\bigr) \cong \Coh_{\{\tot\}}^*(\VB_\tot)\) for all \(\tot\in\Tot\).  It is not clear whether this characterisation extends to the equivariant case.  As we shall see, the definition above ensures that orientations have the expected properties.

Let~\(\VB\) be a \(\Grd\)\nb-vector bundle over~\(\Tot\) with a \(\Grd\)\nb-invariant inner product.  If \(A\subseteq\VB\) is \(\Grd\)\nb-invariant and \(\Tot\)\nb-compact, then there is a \(\Grd\)\nb-invariant function \(\varrho\colon \Tot\to(0,\infty)\) with
\[
A\subseteq \Disk_\varrho(\VB) \defeq
\bigl\{ \vb\in\VB \bigm| \norm{\vb} \le
\varrho\bigl(\proj{\VB}(\vb)\bigr)
\bigr\}.
\]
Hence an \(\Coh\)\nb-orientation for~\(\VB\) will be supported in \(\Disk_\varrho(\VB)\) for some function~\(\varrho\) as above.  Since rescaling by~\(\varrho\) is homotopic to the identity map, we can find another representative that is supported in the closed unit ball~\(\Disk\VB\) of~\(\VB\).  Furthermore, since~\(\total{\VB}\) is homeomorphic to the open unit ball, the proof of Lemma~\ref{lem:excision_Coh_supported} yields
\begin{equation}
  \label{eq:Coh_support_VB}
  \Coh^*_\Tot(\total{\VB}) \cong \Coh^*(\Disk\VB,\Sphere\VB),
\end{equation}
where \(\Disk\VB\) and~\(\Sphere\VB\) denote the unit disk and unit sphere bundles in~\(\VB\).

\begin{lemma}
  \label{lem:orientation_Thom_iso}
  Let~\(\VB\) be an \(\Coh\)\nb-oriented \(\Grd\)\nb-vector bundle over~\(\Tot\) and let \(\varphi\colon \Tot\to\Other\) be a \(\Grd\)\nb-map.  Then exterior product with~\(\tau\) induces an isomorphism \(\Coh^*_\Other(\Tot) \cong \Coh^*_\Other(\total{\VB})\).
\end{lemma}

\begin{proof}
  Recall that \(\Coh^*_\Other(\Tot)\) is the inductive limit of \(\Coh^*(\Tot,\Tot\setminus\Third) \cong \Coh^*(\cl{\Third},\bd\Third)\), where~\(\Third\) runs through the directed set of open, \(\Grd\)\nb-invariant, relatively \(\Other\)\nb-compact subsets of~\(\Tot\).  Arguing as in the proof of~\eqref{eq:Coh_support_VB}, we get
  \[
  \Coh^*_\Other(\total{\VB}) \cong
  \varinjlim \Coh^*(\Disk\VB|_{\cl{\Third}},
  \Disk\VB|_{\bd\Third}\cup\Sphere\VB|_{\cl{\Third}})
  \]
  with~\(\Third\) as above.  The \(\Coh\)\nb-orientation and excision yield isomorphisms
  \begin{align*}
    \Coh^n(\bd\Third)
    &\cong \Coh^n(\Disk\VB|_{\bd\Third},\Sphere\VB|_{\bd\Third})
    \cong \Coh^n(\Disk\VB|_{\bd\Third} \cup
    \Sphere\VB|_{\cl\Third},\Sphere\VB|_{\cl\Third}),\\
    \Coh^n(\cl\Third) &\cong
    \Coh^n(\Disk\VB|_{\cl\Third},\Sphere\VB|_{\cl\Third}).
  \end{align*}
  These groups fit into long exact sequences
  \[
  \dotsb
  \to \Coh^{n+1}(\bd\Third)
  \to \Coh^n(\cl\Third,\bd\Third)
  \to \Coh^n(\cl\Third)
  \to \Coh^n(\bd\Third)
  \to \dotsb
  \]
  and
  \begin{multline*}
    \dotsb
    \to \Coh^{n+1}(\Disk\VB|_{\bd\Third} \cup
    \Sphere\VB|_{\cl\Third},\Sphere\VB|_{\cl\Third})
    \to \Coh^n(\Disk\VB|_{\cl{\Third}},
    \Disk\VB|_{\bd\Third}\cup\Sphere\VB|_{\cl{\Third}})
    \\\to \Coh^n(\Disk\VB|_{\cl\Third},\Sphere\VB|_{\cl\Third})
    \to \Coh^n(\Disk\VB|_{\bd\Third} \cup
    \Sphere\VB|_{\cl\Third},\Sphere\VB|_{\cl\Third})
    \to \dotsb.
  \end{multline*}
  Multiplication with the \(\Coh\)\nb-orientation provides a chain map between these exact sequences.  This is invertible on two of three entries by definition of an \(\Coh\)\nb-orientation.  It is an isomorphism \(\Coh^*(\cl\Third,\bd\Third) \cong \Coh^*(\Disk\VB|_{\cl{\Third}}, \Sphere\VB|_{\cl{\Third}} \cup \Disk\VB|_{\bd\Third})\) as well by the Five Lemma.
\end{proof}

\begin{lemma}
  \label{lem:pull-back_orientation}
  If \(f\colon \Source\to\Target\) is a \(\Grd\)\nb-map and \(\tau\in\Coh^*_\Target(\total{\VB})\) is an \(\Coh\)\nb-orientation for a \(\Grd\)\nb-vector bundle~\(\VB\) over~\(\Target\), then \(f^*(\tau)\in\Coh^*_\Source(\total{f^*\VB})\) is an \(\Coh\)\nb-orientation for~\(f^*(\VB)\).
\end{lemma}

\begin{proof}
  Trivial from the definition.
\end{proof}

\begin{lemma}
  \label{lem:glue_orientation}
  Let~\(\Tot\) be a \(\Grd\)\nb-space and let~\(\VB\) be a \(\Grd\)\nb-vector bundle over~\(\Tot\).  Let \(\Tot_1\) and~\(\Tot_2\) be closed \(\Grd\)\nb-invariant subsets of~\(\Tot\) with \(\Tot= \Tot_1\cup \Tot_2\) and let \(\tau_1\) and~\(\tau_2\) be \(\Coh\)\nb-orientations for \(\VB|_{\Tot_1}\) and \(\VB|_{\Tot_2}\), respectively.  Then there is an \(\Coh\)\nb-orientation on \(\Tot_1\cup\Tot_2\) that restricts to \(\tau_1\) and~\(\tau_2\) on \(\Tot_1\) and~\(\Tot_2\).
\end{lemma}

\begin{proof}
  Let \(\Tot_{12} \defeq \Tot_1\cap \Tot_2\) and let \(\VB_j\defeq \VB|_{\Tot_j}\) for \(j\in\{1,2,12\}\).  We have \(\Coh^*_\Tot(\VB_j) \cong \Coh^*_{\Tot_j}(\VB_j)\) for \(j=1,2,12\).  The properties of~\(\Coh\) yield a long exact Mayer--Vietoris sequence
  \[
  \dotsb \to
  \Coh^{n+1}_\Tot(\total{\VB_{12}}) \to
  \Coh^n_\Tot(\total{\VB}) \to
  \Coh^n_\Tot(\total{\VB_1}) \oplus
  \Coh^n_\Tot(\total{\VB_2}) \to
  \Coh^n_\Tot(\total{\VB_{12}}) \to \dotsb.
  \]
  Hence there is \(\tau\in\Coh^*_\Tot(\total{\VB})\) that restricts to \(\tau_j\) on~\(\total{\VB_j}\) for \(j=1,2\).  We claim that any such~\(\tau\) is an \(\Coh\)\nb-orientation for~\(\VB\).

  Let \(f\colon \Other\to\Tot\) be a \(\Grd\)\nb-map and let \(\Other_j\defeq f^{-1}(\Tot_j)\) for \(j=1,2,12\).  We have a commuting diagram of exact Mayer--Vietoris sequences
  \[
  \xymatrix{
  \Coh^n_\Tot(\total{f^*\VB}) \ar[r]&
  \Coh^n_\Tot(\total{f^*\VB_1}) \oplus
  \Coh^n_\Tot(\total{f^*\VB_2}) \ar[r]&
  \Coh^n_\Tot(\total{f^*\VB_{12}}) \ar[r]& \dotsb\\
  \Coh^n_\Tot(\Other) \ar[r] \ar[u]^{\tau}&
  \Coh^n_\Tot(\Other_1) \oplus
  \Coh^n_\Tot(\Other_2) \ar[r] \ar[u]^{\tau_1\oplus\tau_2}&
  \Coh^n_\Tot(\Other_{12}) \ar[r] \ar[u]^{\tau_{12}}&
  \dotsb \ar@{.>}[u]\\
  }
  \]
  By assumption, \(\tau_1\), \(\tau_2\), and~\(\tau_{12}\) induce isomorphisms.  By the Five Lemma, \(\tau\) induces an isomorphism as well, so that~\(\tau\) is an \(\Coh\)\nb-orientation.
\end{proof}

\begin{lemma}
  \label{lem:orient_sum}
  Let \(\VB_1\) and~\(\VB_2\) be two \(\Grd\)\nb-vector bundles over~\(\Tot\), let~\(\VB_2\) be \(\Coh\)\nb-oriented.  Then there is a natural bijection between \(\Coh\)\nb-orientations on \(\VB_1\) and \(\VB_1\oplus\VB_2\).
\end{lemma}

\begin{proof}
  Assume first that \(\VB_1\) and~\(\VB_2\) are \(\Coh\)\nb-oriented by \(\tau_j\in \Coh^*_{\Tot}(\total{\VB_j})\).  The total space of \(\VB_1\oplus\VB_2\) is the total space of the \(\Grd\)\nb-vector bundle \(\proj{\VB_1}^*(\VB_2)\) over~\(\total{\VB_1}\).  By Lemma~\ref{lem:pull-back_orientation}, \(\proj{\VB_1}^*(\tau_2)\in \Coh^*_{\total{\VB_1}}(\total{\VB_1\oplus\VB_2})\) is an \(\Coh\)\nb-orientation for~\(\proj{\VB_1}^*(\VB_2)\).  Its product with~\(\tau_1\) in \(\Coh^*_{\Tot}(\total{\VB_1\oplus\VB_2})\) is an \(\Coh\)\nb-orientation for~\(\VB_1 \oplus\VB_2\).

  Now let \(\tau_2\) and~\(\tau_{12}\) be \(\Coh\)\nb-orientations for \(\VB_2\) and \(\VB_{12}\defeq \VB_1\oplus\VB_2\).  For any \(f\colon \Other\to\nobreak\Tot\), the product of the Thom isomorphism for~\(\VB_{12}\) and the inverse Thom isomorphism for \(\proj{\VB_1}(\VB_2)\) provides an isomorphism \(\Coh^*_\Tot(\Other) \to \Coh^*_\Tot(\total{f^*\VB_1})\).  This is induced by a class \(\tau_1\in \Coh^*_\Tot(\total{\VB_1})\), namely, the image of the identity element in \(\Coh^*_\Tot(\Tot) = \Coh^*(\Tot)\).  Hence~\(\tau_1\) is an \(\Coh\)\nb-orientation on~\(\VB_1\).  The two constructions of \(\Coh\)\nb-orientations on \(\VB_1\) and \(\VB_1\oplus\VB_2\) are inverse to each other.
\end{proof}

We always equip pull-backs and direct sums of \(\Coh\)\nb-oriented \(\Grd\)\nb-vector bundles with the induced \(\Coh\)\nb-orientations described in Lemmas \ref{lem:pull-back_orientation} and~\ref{lem:orient_sum}.

In the presence of \(\Coh\)\nb-orientations, we modify the notions of trivial and subtrivial \(\Grd\)\nb-vector bundles: a \emph{trivial} \(\Coh\)\nb-oriented \(\Grd\)\nb-vector bundle is the pull-back of an \(\Coh\)\nb-oriented \(\Grd\)\nb-vector bundle on~\(\Base\), and a \emph{subtrivial} \(\Coh\)\nb-oriented \(\Grd\)\nb-vector bundle is an \(\Coh\)\nb-oriented direct summand of a trivial \(\Coh\)\nb-oriented \(\Grd\)\nb-vector bundle.

\subsection{Oriented normally non-singular maps}
\label{sec:oriented_normal}

In this section, we let \(\Vect_\Grd(\Tot)\) be the monoid of subtrivial \emph{\(\Coh\)\nb-oriented} \(\Grd\)\nb-vector bundles, that is, we require the \(\Grd\)\nb-vector bundles in the definition of a normally non-singular map to be \(\Coh\)\nb-oriented.  This leads to the theory of \emph{\(\Coh\)\nb-oriented} normally non-singular maps.

\begin{definition}
  \label{def:Coh-oriented_normal_map}
  An \emph{\(\Coh\)\nb-oriented normally non-singular map} from~\(\Source\) to~\(\Target\) consists of
  \begin{itemize}
  \item \(\VB\), a subtrivial \(\Coh\)\nb-oriented \(\Grd\)\nb-vector bundle over~\(\Source\);

  \item \(\Triv\), an \(\Coh\)\nb-oriented \(\Grd\)\nb-vector bundle over~\(\Base\); and

  \item \(\hat{f}\colon \total{\VB} \opem \total{\Triv^\Target}\), a \(\Grd\)\nb-equivariant open embedding.

  \end{itemize}
  The trace, the stable normal bundle, and the dimension of an \(\Coh\)\nb-oriented normally non-singular map are defined as for normally non-singular maps.
\end{definition}

The relations of isotopy, lifting, and equivalence for normally non-singular maps extend to \(\Coh\)\nb-oriented normally non-singular maps; in the definition of lifting, we require the additional trivial vector bundle to be \(\Coh\)\nb-oriented, of course, and equip the direct sums that appear with the induced \(\Coh\)\nb-orientations.  Isotopy and equivalence remain equivalence relations for \(\Coh\)\nb-oriented normally non-singular maps.

A composite or exterior product of \(\Coh\)\nb-oriented normally non-singular maps inherits a canonical \(\Coh\)\nb-orientation because we may add and pull back \(\Coh\)\nb-orientations.

\begin{proposition}
  \label{pro:Coh-orient_normal_maps_smc}
  The category of \(\Grd\)\nb-spaces with \(\Coh\)\nb-oriented normally non-singular maps as morphisms and the exterior product is a symmetric monoidal category.
\end{proposition}

\begin{proof}
  Copy the proof of Proposition~\ref{pro:normal_maps_smc}.
\end{proof}

\begin{definition}
  \label{def:normal_category_K}
  Let \(\Nor_{\Coh}(\Grd)\) be the category of proper \(\Grd\)\nb-spaces with \(\Coh\)\nb-oriented normally non-singular maps as morphisms.
\end{definition}

\begin{example}
  \label{exa:smooth_oriented}
  Let \(f\colon \Source\to\Target\) be a smooth map between two smooth manifolds.  Lift it to a normally non-singular map from~\(\Source\) to~\(\Target\) as in Example~\ref{exa:smooth_normal_map}.  When is this map \(\Coh\)\nb-oriented?  Recall that the normally non-singular map associated to~\(f\) is of the form \(\NM\defeq (\VB,\hat{f},\R^n)\), where~\(n\) is chosen so large that there is a smooth embedding \(h\colon \Source\to\R^n\) and~\(\hat{f}\) is a tubular neighbourhood for the smooth embedding \((f,h)\colon \Source\to\Target\times\R^n\).  Thus~\(\VB\) is the normal bundle of \((f,h)\).  Since the constant vector bundle~\(\R^n\) is canonically \(\Coh\)\nb-oriented, an \(\Coh\)\nb-orientation for~\(f\) is equivalent to one for~\(\VB\).

  Already~\(h\) is a smooth embedding, and its normal bundle~\(\Normal_\Source\) is a stable normal bundle of the manifold~\(\Source\); it has the property that \(\Normal_\Source\oplus\Tvert\Source\) is the constant vector bundle~\(\R^n\).  Since~\(h\) is a smooth embedding, we get a canonical vector bundle extension \(f^*(\Tvert\Target)\mono \VB\epi \Normal_\Source\), so that \(\VB\cong f^*(\Tvert\Target)\oplus\Normal_\Source\).  Thus~\(\NM\) is \(\Coh\)\nb-oriented if and only if \(f^*(\Tvert\Target)\oplus\Normal_\Source\) is \(\Coh\)\nb-oriented.  This does not depend on the choice of~\(\NM\).
\end{example}

We are going to show that an \(\Coh\)\nb-orientation on a lifting of a normally non-singular map is equivalent to a (stable) \(\Coh\)\nb-orientation on its stable normal bundle.

\begin{definition}
  \label{def:stable_orientation}
  Let \((\VB_+,\VB_-)\) be a pair of subtrivial \(\Grd\)\nb-vector bundles over~\(\Source\).  A \emph{stable \(\Coh\)\nb-orientation} for \((\VB_+,\VB_-)\) consists of a subtrivial \(\Grd\)\nb-vector bundle~\(\VB_3\) on~\(\Source\) and \(\Coh\)\nb-orientations on \(\VB_+\oplus\VB_3\) and \(\VB_-\oplus\VB_3\).  Two such stable \(\Coh\)\nb-orientations are called \emph{equivalent} if the induced \(\Coh\)\nb-orientations on \((\VB_+\oplus\VB_3) \oplus (\VB_-\oplus\VB_3')\) and \((\VB_+\oplus\VB_3') \oplus (\VB_-\oplus\VB_3)\) agree (we use Lemma~\ref{lem:orient_sum} to define these induced \(\Coh\)\nb-orientations).

  A \emph{stably \(\Coh\)\nb-oriented normally non-singular map} is a normally non-singular map with a stable \(\Coh\)\nb-orientation on its stable normal bundle.
\end{definition}

\begin{lemma}
  \label{lem:stable_Coh-orientation}
  Assume that any \(\Grd\)\nb-vector bundle over~\(\Base\) is a direct summand of an \(\Coh\)\nb-oriented \(\Grd\)\nb-vector bundle.  Then there is a bijection between lifting classes of \(\Coh\)\nb-oriented normally non-singular maps and stably \(\Coh\)\nb-oriented normally non-singular maps \(\Source\to\Target\).
\end{lemma}

Here we use the equivalence relation generated by lifting, which is contained in the equivalence of normally non-singular maps.

\begin{proof}
  Clearly, an \(\Coh\)\nb-orientation on a normally non-singular map \(\NM=(\VB,\Triv,\hat{f})\) induces a stable \(\Coh\)\nb-orientation on its stable normal bundle \((\VB,\Triv^\Source)\).  Furthermore, lifting along \(\Coh\)\nb-oriented \(\Grd\)\nb-vector bundles over~\(\Base\) does not alter this stable \(\Coh\)\nb-orientation.

  Conversely, suppose that the stable normal bundle \((\VB,\Triv^\Source)\) of~\(\NM\) carries a stable \(\Coh\)\nb-orientation.  We want to construct an \(\Coh\)\nb-orientation on a lifting of~\(\NM\).  We are given \(\Coh\)\nb-orientations on \(\VB\oplus\VB_3\) and \(\Triv^\Source\oplus\VB_3\) for some \(\Grd\)\nb-vector bundle~\(\VB_3\).  Our assumptions provide a \(\Grd\)\nb-vector bundle~\(\Triv_2\) over~\(\Base\) such that \(\Triv\oplus\Triv_2\) is \(\Coh\)\nb-oriented and~\(\VB_3\) is a direct summand in~\(\Triv_2^\Source\).  Let \(\VB_3\oplus\VB_3^\bot \cong \Triv_2^\Source\).  Since
  \[
  (\Triv^\Source \oplus \VB_3) \oplus \VB_3^\bot
  \cong (\Triv\oplus\Triv_2)^\Source
  \]
  and \(\Triv^\Source \oplus \VB_3\) and \(\Triv\oplus\Triv_2\) are \(\Coh\)\nb-oriented, \(\VB_3^\bot\) inherits a canonical \(\Coh\)\nb-orientation by Lemma~\ref{lem:orient_sum}.  Then \(\VB\oplus \Triv_2^\Source \cong (\VB\oplus\VB_3)\oplus \VB_3^\bot\) inherits an \(\Coh\)\nb-orientation as well.  Thus we get an \(\Coh\)\nb-orientation on the lifting of~\(\NM\) along~\(\Triv_2\).

  Now assume that we are given two equivalent stable \(\Coh\)\nb-orientations on \((\VB,\Triv^\Source)\) involving stabilisation by \(\VB_3\) and~\(\VB_3'\).  Construct \(\Grd\)\nb-vector bundles \(\Triv_2\) and~\(\Triv_2'\) over~\(\Base\) for both of them and \(\Coh\)\nb-orientations on \(\NM\oplus\Triv_2\) and \(\NM\oplus\Triv_2'\).  Since \(\Triv\oplus\Triv_2\) and \(\Triv\oplus\Triv_2'\) are \(\Coh\)\nb-oriented, the liftings \((\NM\oplus\Triv_2)\oplus(\Triv\oplus\Triv_2')\) and \((\NM\oplus\Triv_2')\oplus (\Triv\oplus\Triv_2)\) inherit \(\Coh\)\nb-orientations.  By construction, these normally non-singular maps involve the \(\Grd\)\nb-vector bundle \(\Triv\oplus\Triv_2\oplus\Triv\oplus\Triv_2'\) over~\(\Base\) with the same \(\Coh\)\nb-orientation.  The \(\Grd\)\nb-vector bundles over~\(\Source\) are
  \begin{align*}
    (\VB\oplus \Triv_2^\Source) \oplus
    (\Triv\oplus\Triv_2')^\Source &\cong
    (\VB \oplus \VB_3) \oplus \VB_3^\bot
    \oplus (\Triv^\Source\oplus\VB_3') \oplus (\VB_3')^\bot,\\
    (\VB\oplus (\Triv_2')^\Source) \oplus
    (\Triv\oplus\Triv_2)^\Source &\cong
    (\VB \oplus \VB_3') \oplus (\VB_3')^\bot
    \oplus (\Triv^\Source\oplus\VB_3) \oplus \VB_3^\bot,
  \end{align*}
  equipped with the \(\Coh\)\nb-orientations induced by Lemma~\ref{lem:orient_sum}.  These agree by assumption (here we identify the two \(\Grd\)\nb-vector bundles by the obvious isomorphism).

  As a result, the lifting class of the \(\Coh\)\nb-oriented normally non-singular map \(\NM\oplus\Triv_2\) only depends on the equivalence class of the stable \(\Coh\)\nb-orientation of \((\VB,\Triv^\Source)\); the same argument with \(\VB_3=\VB_3'\) shows that the lifting class of the \(\Coh\)\nb-oriented normally non-singular map \(\NM\oplus\Triv_2\) does not depend on the auxiliary choices.  The two constructions above well-define maps between lifting classes of stably \(\Coh\)\nb-oriented and \(\Coh\)\nb-oriented maps.  They are easily seen to be inverse to each other (up to lifting).
\end{proof}

\begin{remark}
  \label{rem:sub_oriented}
  For equivariant \(\K\)\nb-theory or \(\KO\)\nb-theory, \(\Triv^8\) is \(\Coh\)\nb-oriented for any \(\Grd\)\nb-vector bundle~\(\Triv\).  Hence the assumption in Lemma~\ref{lem:stable_Coh-orientation} becomes vacuous for these cohomology theories.
\end{remark}

\begin{corollary}
  \label{cor:oriented_normal_map_unique}
  Let \(\Source\) and~\(\Target\) be smooth \(\Grd\)\nb-manifolds.  Assume that there is a smooth normally non-singular map from~\(\Source\) to the object space~\(\Base\) of~\(\Grd\) and that all \(\Grd\)\nb-vector bundles over \(\Base\) and~\(\Source\) are direct summands in trivial, \(\Coh\)\nb-oriented \(\Grd\)\nb-vector bundles.  Then smooth equivalence classes of\/ \(\Coh\)\nb-oriented smooth normally non-singular \(\Grd\)\nb-maps from~\(\Source\) to~\(\Target\) correspond bijectively to pairs \((f,\tau)\) where~\(f\) is the smooth homotopy class of a smooth \(\Grd\)\nb-map from~\(\Source\) to~\(\Target\) and~\(\tau\) is a stable \(\Coh\)\nb-orientation on its stable normal bundle \(f^*([\Tvert\Target]) - [\Tvert\Source]\).
\end{corollary}

\begin{proof}
  Theorem~\ref{the:normal_map_unique} shows that smooth homotopy classes of smooth \(\Grd\)\nb-maps correspond bijectively to smooth equivalence classes of smooth normally non-singular \(\Grd\)\nb-maps.  The additional \(\Coh\)\nb-orientations reduce to one on the \(\Grd\)\nb-vector bundle \(f^*(\Tvert\Target) \oplus \Normal_\Source\) over~\(\Source\) as in Example~\ref{exa:smooth_oriented}; this is equivalent to the stable normal bundle of~\(f\).  But \(f^*(\Tvert\Target) \oplus \Normal_\Source\) is equal in \([\Vect^\Grd_0(\Source)]\) to \(f^*([\Tvert\Target]) - [\Tvert\Source]\).  Finally, we use Lemma~\ref{lem:stable_Coh-orientation}.
\end{proof}

\subsection{Wrong-way functoriality}
\label{sec:wrong-way_normal}

Let \(\NM= (\VB,\Triv,\hat{f})\) be an \(\Coh\)\nb-oriented normally non-singular map from~\(\Source\) to~\(\Target\).  Let \(\NM!\colon \Coh^*_\Base(\Source)\to\Coh^*_\Base(\Target)\) be the product of
\begin{itemize}
\item the Thom isomorphism \(\Coh^*_\Base(\Source)\to\Coh^*_\Base(\total{\VB})\) (see Lemma~\ref{lem:orientation_Thom_iso});

\item the map \(\Coh^*_\Base(\total{\VB})\to\Coh^*_\Base(\total{\Triv^\Target})\) induced by the open embedding~\(\hat{f}\); and

\item the inverse Thom isomorphism \(\Coh^*_\Base(\total{\Triv^\Target})\to\Coh^*_\Base(\Target)\).
\end{itemize}
This construction is analogous to the definition of the topological Atiyah--Singer index map in~\cite{Atiyah-Singer:I}.  See Section~\ref{sec:index} for further remarks about index theory.

\begin{theorem}
  \label{the:wrong-way_functor}
  The maps \(\Tot\mapsto \Coh^*_\Base(\Tot)\) and \(\NM\mapsto \NM!\) define a functor from \(\Nor_{\Coh}(\Grd)\) to the category of Abelian groups.
\end{theorem}

\begin{proof}
  The construction of~\(\NM!\) does not require the vector bundles involved to be trivial or subtrivial: it still works for triples \((\VB_\Source,\VB_\Target,\hat{f})\) where \(\VB_\Source\) and~\(\VB_\Target\) are \(\Coh\)\nb-oriented \(\Grd\)\nb-vector bundles over~\(\Source\) and~\(\Target\) and~\(\hat{f}\) is an open embedding \(\total{\VB_\Source}\opem\total{\VB_\Target}\).  Since the definition of the composition product involves lifting a normally non-singular map to a triple of this form, it is useful to treat these more general objects.

  Let \(\NM_j\defeq (\VB_{\Source}, \VB_{\Target}, \hat{f}_j)\) for \(j=1,2\) be triples as above with isotopic maps \(\hat{f}_1\) and~\(\hat{f}_2\).  Then \(\NM_1!=\NM_2!\) if~\(\NM_1\) is isotopic to~\(\NM_2\) because~\(\Coh^*\) is homotopy invariant.

  If \(\VB_1\) and~\(\VB_2\) are \(\Grd\)\nb-vector bundles over the same space~\(\Tot\), then the Thom isomorphism for \(\VB_1\oplus\VB_2\) is the product of the Thom isomorphisms \(\Coh^*_\Base(\Tot) \cong \Coh^*_\Base(\total{\VB_1})\) for~\(\VB_1\) and \(\Coh^*_\Base(\total{\VB_1}) \cong \Coh^*_\Base(\total{\VB_1}\oplus\total{\VB_2})\) for \(\proj{\VB_1}^*(\VB_2)\) by the proof of Lemma~\ref{lem:orient_sum}.

  Using this and the naturality of the Thom isomorphism, we conclude that lifting does not change the wrong-way element -- even lifting along non-trivial \(\Grd\)\nb-vector bundles.  Thus \(\NM_1!=\NM_2!\) if \(\NM_1\) and~\(\NM_2\) are equivalent.  This shows that \(\NM\mapsto\NM!\) well-defines a map on equivalence classes of normally non-singular maps.

  It is clear that \(\Id!=\Id\).  The product of two normally non-singular maps involves lifting both factors -- one of them along a non-trivial vector bundle -- and then composing the open embeddings involved.  Functoriality for open embeddings is easy.  We have just observed that the lifting step does not alter the wrong-way elements.  In the second step, the effect of composing is to replace the map
  \[
  \zers{\VB}!\circ \proj{\VB}!\colon
  \Coh^*_\Base(\total{\VB})\to
  \Coh^*_\Base(\Tot) \to \Coh^*_\Base(\total{\VB})
  \]
  for a \(\Grd\)\nb-vector bundle~\(\VB\) over~\(\Tot\) by the identity map on \(\Coh^*_\Base(\total{\VB})\).  Since \(\zers{\VB}!\) and \(\proj{\VB}!\) are the Thom isomorphism for~\(\VB\) and its inverse, we get functoriality.
\end{proof}

\section{Normally non-singular maps and index theory}
\label{sec:index}

Normally non-singular maps formalise the construction of \emph{topological index maps} by Atiyah and Singer in~\cite{Atiyah-Singer:I}.  In Kasparov's bivariant \(\K\)\nb-theory, we may also construct \emph{analytic index maps} for smooth maps that are not necessarily normally non-singular.  Certain index theorems compare both constructions.  Theorem~\ref{the:normal_map_unique} is a prerequisite for such results because it asserts that the smooth maps for which the analytic index map is defined have an essentially unique topological index associated to them.

Let~\(\Grd\) be a proper groupoid, let \(\Source\) and~\(\Target\) be smooth \(\Grd\)\nb-manifolds, and let \(f\colon \Source \to \Target\) be a \(\Grd\)\nb-equivariant, fibrewise smooth, \(\K\)\nb-oriented map.  We want to define a wrong-way map \(f!\in \KK^\Grd_*\bigl(\CONT_0(\Source),\CONT_0(\Target)\bigr)\) associated to~\(f\).  If~\(f\) were a \emph{normally non-singular} map, we could use the construction in Section~\ref{sec:wrong-way_normal} for this purpose, because it only uses Thom isomorphisms and wrong-way functoriality for open embeddings, which do provide classes in \(\KK^\Grd\).  In general, the construction of~\(f!\) follows~\cite{Connes-Skandalis:Longitudinal} and requires a factorisation of~\(f\) through a \(\K\)\nb-oriented smooth embedding and a \(\K\)\nb-oriented smooth submersion.

Let~\(\VB\) be a \(\Grd\)\nb-vector bundle over~\(\Source\) such that \(\VB\oplus f^*(\Tvert\Target)\) is \(\K\)\nb-oriented.  We may factor~\(f\) as \(f=\pi_\Target\circ f'\) with the coordinate projection \(\pi_\Target\colon \total{\VB}\times_\Base\Target\to\Target\) and
\[
f'\colon \Source\to \total{\VB}\times_\Base\Target,\qquad
\source\mapsto \bigl(\zers{\VB}(\source),f(\source)\bigr).
\]
The map~\(f'\) is a smooth immersion with \(\K\)\nb-oriented normal bundle \(\VB\oplus f^*(\Tvert\Target)\) and has a tubular neighbourhood by Theorem~\ref{the:tubular_for_embedding}.  Hence the Thom isomorphism for \(\VB\oplus f^*(\Tvert\Target)\) and the functoriality of \(\KK\) for open embeddings provide a wrong-way element \(f'!\in \KK^\Grd_*\bigl(\CONT_0(\Source),\CONT_0(\total{\VB}\times_\Base\Target)\bigr)\).  The same construction is used in Section~\ref{sec:wrong-way_normal} for smooth normally non-singular maps.  In fact, \(f'\) is the trace of a smooth normally non-singular embedding, and the latter is unique up to equivalence by the proof of Theorem~\ref{the:normal_map_unique}.

It remains to define a wrong-way element for \(\pi_\Target\colon \total{\VB}\times_\Base\Target\to\Target\), which is a smooth submersion.  Its normal bundle is the pull-back of the \(\Grd\)\nb-vector bundle \(\VB\oplus \Tvert\Source\) on~\(\Source\), which is \(\K\)\nb-oriented because \(\VB\oplus \Tvert\Target\) and the map~\(f\) are \(\K\)\nb-oriented.

Since it only remains to study maps like~\(\pi_\Target\), we assume from now on that \(f\colon \Source \to \Target\) be a \(\Grd\)\nb-equivariant, \(\K\)\nb-oriented, fibrewise smooth submersion.  Then~\(\Source\) may be regarded as a \(\Grd\ltimes\Target\)-manifold, since its fibres are smooth \(\Grd\)\nb-manifolds.  Let \(\Tvert_f\Source\) denote the vertical tangent space of this \(\Grd\ltimes \Target\)-manifold; thus~\(\Tvert_f\Source\) is the tangent bundle to the fibres of~\(f\).  It is \(\Grd\ltimes\Target\)\nb-equivariantly \(\K\)\nb-oriented by assumption.

The family of Dirac operators along the fibres of~\(f\) now provides an element
\[
D_f\in \KK^{\Grd\ltimes \Target}_{\dim (\Target)-\dim (\Source)}\bigl(\CONT_0(\Source), \CONT_0(\Target)\bigr),
\]
see~\cite{Connes-Skandalis:Longitudinal}, or~\cite{Emerson-Meyer:Dualities} for a discussion of the groupoid-equivariant case.  We interpret this as an \emph{analytic} wrong-way element associated to the \(\K\)\nb-oriented submersion~\(f\).

If we also assume~\(f\) to be proper, so that composition with~\(f\) provides an equivariant \(^*\)\nb-homomorphism \(f_*\colon \CONT_0(\Target)\to\CONT_0(\Source)\), then we can form the index
\[
\Ind (D_f) \defeq f_*(D_f)
\in \KK^{\Grd\ltimes\Target}_* \bigl(\CONT_0(\Target), \CONT_0(\Target)\bigr)
\cong \RK^*_\Grd (\Target)
\]
(the last isomorphism is from~\cite{Emerson-Meyer:Equivariant_K}).  More generally, we may compose~\(D_f\) with classes in
\[
\RK^*_{\Grd,\Target}(\Source) \defeq
\KK^{\Grd\ltimes \Target}_*\bigl(\CONT_0(\Target), \CONT_0(\Source)\bigr),
\]
equivariant \(\K\)\nb-theory classes on~\(\Source\) with \(\Target\)\nb-compact support, to get elements of \(\RK^*_\Grd (\Target)\).  An equivariant families index theorem is supposed to compute these elements of \(\RK^*_\Grd (\Target)\) or, even better, the class of~\(D_f\) in \(\KK^{\Grd\ltimes \Target}_*\bigl(\CONT_0(\Source), \CONT_0(\Target)\bigr)\), in topological terms from the given data~\(f\).

A \(\K\)\nb-oriented normal factorisation of~\(f\), that is, a \(\K\)\nb-oriented \(\Grd\ltimes \Target\)-equivariant normally non-singular map \((\VB, \Triv, \hat{f})\) with trace~\(f\), provides such a topological formula:
\[
f!\defeq \zers{\VB}\otimes_{\CONT_0(\total{\VB})}
\hat{f}\otimes_{\CONT_0(\Triv^\Source)} \proj{\Triv^\Source}
\in \KK^{\Grd\ltimes \Target}_{\dim (\Target)-\dim (\Source)}\bigl(\CONT_0(\Source), \CONT_0(\Target)\bigr).
\]
This construction is parallel to the Atiyah--Singer topological index and therefore deserves to be called the \emph{topological index} of~\(f\).  The Index Theorem in this case simply states that \(f! = D_f\) in \(\KK^{\Grd\ltimes \Target}_*\bigl(\CONT_0(\Source), \CONT_0(\Target)\bigr)\).  This result may be proved by following a similar argument in~\cite{Connes-Skandalis:Longitudinal} in the non-equivariant case.

Before we sketch this, we return to wrong-way elements.  If~\(f\) is a general smooth \(\K\)\nb-oriented map, then we may factor \(f=\pi_\Target\circ f'\) as above and let~\(f!_\an\) be the Kasparov product of~\(D_{\pi_\Target}\) and~\(f'!\).  This is the \emph{analytic wrong-way element} of a smooth \(\K\)\nb-oriented map.

\begin{theorem}
  \label{the:index_theorem}
  The map \(f\mapsto f!_\an\) is a functor from the category of smooth \(\K\)\nb-oriented maps to \(\KK^\Grd\) and satisfies \(f!_\an=D_f\) for any smooth submersion~\(f\).  If~\(f\) is the trace of a smooth normally non-singular map~\(\hat{f}\), then \(f!_\an = \hat{f}!\).  In particular, if a \(\K\)\nb-oriented submersion~\(f\) is the trace of a smooth normally non-singular map~\(\hat{f}\), then \(\hat{f}!= D_f\).
\end{theorem}

In the non-equivariant case, this is established in~\cite{Connes-Skandalis:Longitudinal}.  The argument easily adapts to the equivariant case.  We merely sketch some of the steps.

\begin{proof}
  The main step of the argument is to show that \(D_{f_1\circ f_2}\) is the Kasparov product of \(D_{f_1}\) and~\(D_{f_2}\) if \(f_1\) and~\(f_2\) are composable \(\K\)\nb-oriented submersions.  This also implies that \(D_f = f!\) for any \(\K\)\nb-oriented submersion~\(f\): the factorisation above merely introduces a Thom isomorphism and its inverse, which cancel.  We also get that \((f_1\circ f_2)!_\an\) is the product of the topological wrong-way element~\(f_2!\) with~\(D_{f_1}\) whenever~\(f_1\) is a smooth submersion and~\(f_2\) is a smooth immersion.  This implies easily that analytic wrong-way elements are functorial for smooth immersions.

  It is clear that the analytic and topological wrong-way elements coincide for zero sections of vector bundles, vector bundle projections, and open embeddings: in the first two cases, both constructions use the Thom isomorphism and its inverse.  Hence they coincide for normally non-singular immersions and special normally non-singular submersions.  We have already seen that topological wrong-way elements are functorial.  The equality of analytic and topological wrong-way elements follows once the analytic wrong-way elements are functorial as well.

  Now consider composable maps \(f_1\colon \Target\to\Third\) and \(f_2\colon \Source\to\Target\) and factor them through smooth immersions and submersions as \(f_j=f_j'\circ f_j''\) with smooth submersions~\(f_j'\) and smooth immersions~\(f_j''\) for \(j=1,2\).  We may even assume that~\(f_j''\) is the zero section of a \(\Grd\)\nb-vector bundle and let~\(\pi_j\) be the corresponding vector bundle projection.  Since we already know functoriality of wrong-way elements for immersions and submersions separately, it remains to prove that the Kasparov product of \(f_1''!\) and~\(D_{f_2'}\) equals \((f_1''\circ f_2')!_\an\).  Since~\(f_1''!\) is a Thom isomorphism, it is invertible and we lose nothing if we compose both sides with its inverse~\(D_{\pi_1}\).  This reduces the issue to \(D_{\pi_1}\circ (f_1''\circ f_2')!_\an = D_{f_2'}\).  Since \(\pi_1\circ f_1''\circ f_2' = f_2'\), we conclude that functoriality of analytic wrong-way elements holds in general once it holds if the first map is a smooth submersion.  By definition, this reduces further to the case where both maps are smooth submersions, which we already know.
\end{proof}

The equality \(D_f=f!\) for \(\K\)\nb-oriented submersions actually follows immediately from the duality isomorphisms in \cites{Emerson-Meyer:Dualities, Emerson-Meyer:Correspondences}.  This means that all the analysis needed to prove the index theorem is already embedded in the proof of the duality isomorphisms.  The argument goes as follows.  Let \(f\colon \Source\to\Target\) be a \(\Grd\)\nb-equivariant, \(\K\)\nb-oriented, fibrewise smooth submersion.  Replacing~\(\Grd\) by \(\Grd\ltimes\Target\), we may assume that \(\Target=\Base\).  A duality isomorphism is always based on two bivariant \(\K\)\nb-theory classes, Dirac and local dual Dirac.  The Dirac elements for the dualities in~\cite{Emerson-Meyer:Dualities} are the analytic Dirac element~\(D_f\) and the topological wrong-way element~\(f!\), respectively.  The local dual Dirac element is the wrong-way element associated to the diagonal embedding \(\Source\to \Source\times_\Base\Source\).  Since this is a smooth immersion, it makes no difference here whether we work analytically or topologically.  The Dirac and local dual Dirac elements \(\Theta\) and~\(D\) for a duality isomorphism determine each other uniquely by the equation \(\Theta \otimes_\Source D = 1\) in \(\KK^{\Grd\ltimes\Source}\bigl(\CONT_0(\Source),\CONT_0(\Source)\bigr)\).  Since both duals for~\(\Source\) use the same local dual Dirac element, they must both involve the same Dirac element, that is, \(D_f=f!\).

This shows that all the analysis required to prove the Index Theorem is already embedded in the proofs of the duality isomorphisms; these only use the functoriality of analytic Dirac elements~\(f!\) with respect to \emph{open embeddings}, the homotopy-invariance of the construction (independence of the choice of Riemnannian metric), and the Thom isomorphism.

The Index Theorem~\ref{the:index_theorem}, \(f!_\an=f!\), only makes sense for smooth maps with a normal factorisation.  For example, let~\(\Grd_A\) be as in Example~\ref{exa:not_enough_vb} for a hyperbolic matrix \(A\in\Gl(2, \Z)\), and let~\(\Source\) be a smooth \(\Grd\)\nb-manifold with at least some morphism in~\(\Grd_A\) acting non-trivially on~\(\Source\).  Assume that~\(\Source\) is \(\Grd_A\)\nb-equivariantly \(\K\)\nb-oriented.  Then there is a Dirac class~\(D_\Source\) along the fibres of the anchor map \(\Source\to \Base\), but~\(\Source\) admits no smooth embedding in an equivariant vector bundle over~\(\Base\).  Hence there is no topological model~\(f!\) for~\(D_f\).  For instance, \(\Grd_A\) acts on itself, so we can get the example \(\Source \defeq \Grd_A\).  The fibres here are complex tori \(\Torus^2 \cong \C/\Z^2\), and the action by translations preserves the complex structure.  Hence~\(\Grd_A\) carries a \(\Grd_A\)\nb-equivariant complex structure.  The fibrewise Dolbeault operator poses an index problem for which it seems unclear how to define the equivariant topological index.

\end{document}